\documentclass[12pt]{amsart}
\usepackage{latexsym,amssymb,amsmath,amsthm,comment,graphicx,color,morefloats,xypic,listings}
\usepackage[section]{placeins}  
\newtheorem{thm}{Theorem} \newtheorem{coro}[thm]{Corollary} \newtheorem{lemma}[thm]{Lemma}
\newtheorem{propo}[thm]{Proposition} 

\title{The K\"unneth formula for graphs}

\author{Oliver Knill}
\date{May 27, 2015}
\address{ Department of Mathematics \\ Harvard University \\ Cambridge, MA, 02138 }

\begin{document}
\maketitle

\begin{abstract}
We define a Cartesian product $G \times H$ for finite simple graphs which satisfies the
K\"unneth formula $H^k(G \times H) = \oplus_{i+j=k} H^i(G) \otimes H^j(G)$ and so
$p_{G \times H}(x) = p_G(x) p_H(y)$ for the Poincar\'e polynomial $p_G(x)=\sum_{k=0} {\rm dim}(H^k(G)) x^k$
and $\chi(G \times H) = \chi(G) \chi(H)$ for the Euler characteristic $\chi(G)=p_G(-1)$. 
The graph $G_1=G \times K_1$ is homotopic to $G$, has a digraph structure and
satisfies the inequality ${\rm dim}(G_1) \geq {\rm dim}(G)$ and $G_1$.
Hodge theory leads to the K\"unneth identity
using the product $f g$ of harmonic forms of $G$ and $H$. A discrete de Rham cohomology 
and ``partial derivatives" emerge on the product graphs. We show that de Rham cohomology is
equivalent to graph cohomology by constructing a chain homotopy. 
The dimension relation ${\rm dim}(G \times H) = {\rm dim}(G) + {\rm dim}(H)$ holds point-wise
${\rm dim}(G \times H)(x,y) = {\rm dim}(G_1)(x) + {\rm dim}(H_1)(y)$ and implies
the inequality ${\rm dim}(G \times H) \geq  {\rm dim}(G) + {\rm dim}(H)$,
mirroring a Hausdorff dimension inequality dimension in the continuum. The chromatic
number $c(G_1)$ of $G_1$ is smaller or equal than $c(G)$ and $c(G \times H) \leq c(G)+c(H)-1$.
Indeed, $c(G \times H)$ is the maximal $n$ for which there is a
$K_n$ subgraph of $G \times H$. The automorphism group of $G \times H$ contains ${\rm Aut}(G) \times {\rm Aut}(H)$.
If $G \sim H$ and $U \sim V$ are homotopic, then $G \times U$ and $H \times V$ are homotopic, leading
to a product on homotopy classes. If $G$ is $k$-dimensional geometric meaning that all unit spheres
$S(x)$ in $G$ are $(k-1)$-discrete spheres, then $G_1$ is $k$-dimensional geometric.
And if $H$ is $l$-dimensional geometric, then $G \times H$ is geometric 
of dimension $(k+l)$. Because the product writes a graph as a polynomial $f_G$ of $n$ variables 
for which the Euler polynomial $e(x) = \sum_k v_k x^k$ is $e_G(x) = f_G(x,\dots,x)$ and $\chi(G)=e_G(-1)$,
the product extends to a ring of chains which unlike graphs is closed under the boundary operation $\delta$ 
defining the exterior derivative $df(x) = f(\delta x)$ and closed under quotients $G/A$ with $A \subset {\rm Aut}(G)$.
By gluing graphs, joins or fibre bundles are defined with the same
features as in the continuum, allowing to build isomorphism classes of bundles. 
\end{abstract}

\begin{figure}[h]
\scalebox{0.12}{\includegraphics{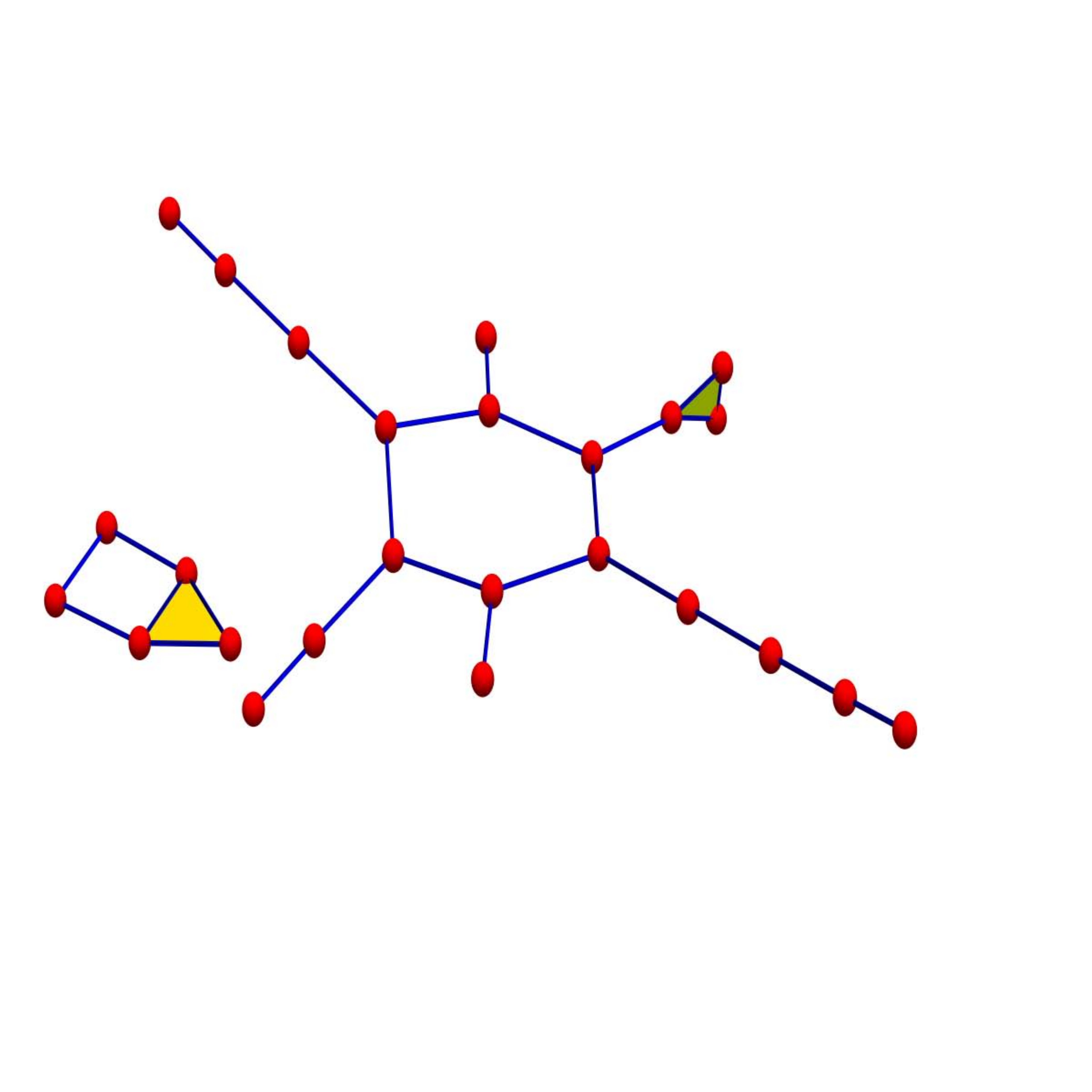}}
\scalebox{0.12}{\includegraphics{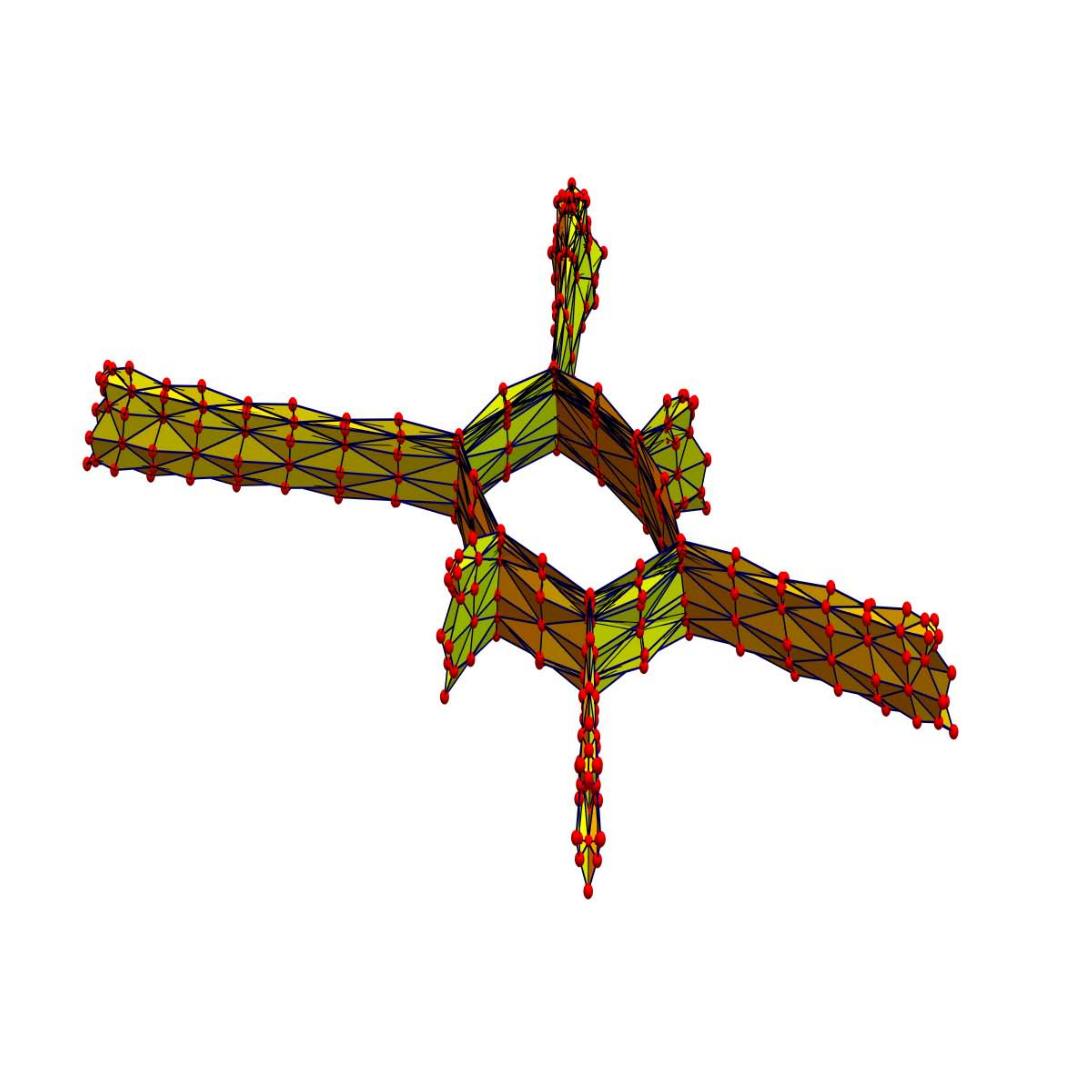}}
\caption{
\label{poster}
The product $G \times H$ of the house graph $G$ of dimension $1.4666 \dots$ and a second graph $H$
of dimension $1.133 \dots$ which is also homotopic to a circle
produces a product graph of dimension $2.70238 \dots$ which is homotopic to a 2-torus. Like Hausdorff
dimension in the continuum, the dimension of the product is larger or equal than the sum $2.6$ of the dimension of
the two factors. The graph $G \times H$ is homotopic to a torus, has Poincar\'e polynomial $1+2x+x^2$
which is the product of the Poincar\'e polynomials $p_G(x)=1+x=p_H(x)$ of the factors. The K\"unneth
theorem tells that if $f,g$ are harmonic $1$-forms representing a nontrivial cohomology class in $H^1(G)$ or
$H^1(H)$ respectively, then $f(x)*1, 1*g(y)$ can be used to construct a basis for $H^1(G \times H)$
and $f(x) g(y)$ can be used to build a $2$-form spanning the $1$-dimensional space $H^2(G \times H)$. The discrete
de Rham theorem tells how to get from the de Rham picture to the simplicial cohomology picture of $G \times H$.
While the linear space of $1$-forms on $G$ and $H$ is $6$ or $21$-dimensional respectively,
the de Rham $1$-forms in $\Omega^1(G) \otimes \Omega^0(H)  \oplus \Omega^0(G) \otimes \Omega^1(H)$ build
a $6*20 + 5*21=225$ dimensional space, while $\Omega^1(G \times H)$ has dimension $2196$, which is almost
$10$ times more. And $\Omega^2(G) \otimes \Omega^0(H) \oplus \Omega^0(G) \otimes \Omega^2(H) \oplus
 \Omega^1(G) \otimes \Omega^1(H)$ is $6*1+5*1+6*21=137$-dimensional while $\Omega^2(G \times H)$ has
dimension $2880$, which is the number of triangles in $G \times H$ and 
more than 20 times the dimension $137$ in the de Rham case.
The discrete Eilenberg-Zilber theorem~(\ref{eilenbergzilber}) assures that $H^1(G \times H)$ is isomorphic to
$H^1(G) \otimes H^0(H) \oplus H^0(G) \otimes H^1(H)$ and
that $H^2(G \times H)$ is isomorphic to $H^1(G) \otimes H^1(H)$ here as $H^2(H)$ and $H^2(G)$ are
$0$-dimensional. The chromatic numbers of $G$ and $H$ are $3$. The chromatic number of $G \times H$ is
$5$ as it contains a $4$-dimensional clique $K_5$ and Theorem~(\ref{chromaticnumberofproduct}).
}
\end{figure}

\section{Introduction}

As acknowledged by nomenclature, Descartes concept of coordinates depends on the notion of a Cartesian product.
Omnipresent in mathematics, it is already used in basic arithmetic to build number systems like 
the field of complex numbers or to define exterior bundles on manifolds. The Cartesian product
allows to build and access higher dimensional features of geometric spaces. Many constructions
in topology like suspensions, joins, fibre bundles, de Rham cohomology or homotopy deformations would not 
work without the concept of a Cartesian product. Of course, we would like to have a product in graph theory 
which shares the properties from the continuum. The new graph product will achieve that. 
It will allow to use ``coordinates" similarly as they are used in the continuum. 
For two arbitrary networks $G,H$ - that is $G$ and $H$ are finite simple graphs -
the coordinates of the product graph $G \times H$ consists of all pairs of complete subgraphs of $G$ and $H$. 
The exterior derivatives in $G$ and $H$ will 
play the role of the ``partial derivatives" in the product and allow to build an exterior de Rham derivative
on the product graph $G \times H$. The actual exterior derivative on the product graph $G \times H$ 
operates on a much larger complex. The later is called the Whitney complex and is 
defined by the simplices in the product.
Having a product allows to work with ``rectangular boxes"s in the product space rather than with 
simplices. Figure~(\ref{poster}) illustrates this for a small example.
To take a picture from the continuum: the curl in the plane is an infinitesimal line integral along a rectangle 
and uses the Leibnitz rule to relate the exterior derivative of the factors with the exterior derivative 
$df g - f dg$ of the product. A simplicial point of view of cohomology integrates around 
infinitesimal triangles to get the curl, ignoring the product structure. 
De Rham establishes equivalence of the two pictures 
on any smooth manifold. In order to prove the K\"unneth formula, we will need to emulate the de Rham theorem
combinatorially.  To do so, we explicitly construct a $k$-form on the Whitney complex of $G \times H$ 
from a $k$-form on the de Rham complex defined by the two graphs $G,H$. The chain map is concrete as
we can from this construct explicit cohomology classes of the product from the cohomology classes in $G$ or $H$. 
The relation is what one calls a chain homotopy. 
The de Rham picture  will be useful if we work with ``discrete n-manifolds" obtained by gluing together local charts of 
products $U_i = G_{i1} \times \cdots \times G_{in}$ of networks or when working with ``fibre bundles" 
$\pi: E \to M$ obtained by gluing locally trivial charts $U_i \times G$ of networks above a discrete manifold 
covered with charts $U_i$. The automorphism group of the fibre $G$ will then play the role of a gauge group on 
$E$, as it does in the continuum. \\

When looking at graphs as geometric structures without dimension restriction,
an amazing similarity with the continuum emerges. It turns out that in the discrete, one can access 
the local structure of space $G$ directly as the complete subgraphs of $G$. These simplices
can serve as fundamental entities playing the role of ``points". Such insight has been promoted 
already in \cite{forman95} but most of the time still graphs are treated as one dimensional 
simplicial complexes. An illustration on how the change of view point allows to emulate results from
the continuum is the fixed point theorem of Brouwer and Lefschetz which looks identical
to the result in the continuum \cite{brouwergraph}.  
Many concepts become elementary: cohomology is part of 
finite dimensional linear algebra, to compute valuations, generalized volumes, one needs integral geometric 
tools in the continuum, while in the discrete is is just count of complete subgraphs.
The discrete Hadwiger theorem \cite{KlainRota} is much easier than the continuum version: the numbers $v_i(G)$ of
$i$-dimensional simplices is a basis for the linear space of valuations.
Differential forms are just functions on a simplex graph and Stokes theorem in the Whitney complex 
is a tautology; it becomes only less obvious when Stokes is considered in a de Rham setup. 
Much Intuition about higher dimensions can be obtained inductively. There are notions of
dimension, cohomology, homotopy, cobordisms, ramified covers, degree and index, spheres, geodesic 
lines and curvature which lead to results mirroring the results in the
continuum. This is not only true for nice geometric graphs but for general undirected networks or finite simple graphs - 
and much works without any exceptions. 
The notion of homotopy of graphs for example immediately leads to the homotopy of a graph
embedded in an other graph and so to homotopy groups for general finite simple graphs. 
The product defined here will actually help to define the homotopy groups as graphs might 
be too small at first to have spheres embedded, so that the graph should 
first be refined.  But the Hurewicz homomorphisms
from the homotopy groups $\pi_k(G)$ to the cohomology groups $H^k(G)$ are then so explicit that one can 
even watch it happen: just apply the heat flow to a 
$k$-forms with support on the $k$-simplices on the embedded $k$-sphere. It converges to a harmonic form which 
by Hodge theory represents a cohomology classes. In the continuum, such a 
proof requires de Rham currents, generalized differential forms which require some
functional analysis. In the graph case, the heat flow is just
a linear ordinary differential equation of the type studied in introductory linear algebra courses.
The definition of homotopy groups which are relevant in coloring questions for graphs 
\cite{knillgraphcoloring,knillgraphcoloring2}
allows to work with spheres in graph theory in the same way as in the continuum. \\

The Euclidean product is not only essential for defining fundamental objects like fibre bundles, 
it is also needed to construct spaces which have the same properties than classical manifolds. 
Examples of such properties are dimension, homotopy, cohomology or Euler characteristic. 
The goal of this note is to give such a product, allowing the use tools like discrete fibre 
bundles in graph theory. We will see that this can be done purely algebraically: 
as we can glue graphs together, this gluing carries over to the product allowing to build discrete bundles.
And if the fibres carry an automorphism group $A$, we get discrete analogues of principle 
bundles on which an enlarged gauge group acts. \\

\begin{figure}[h]
\scalebox{0.22}{\includegraphics{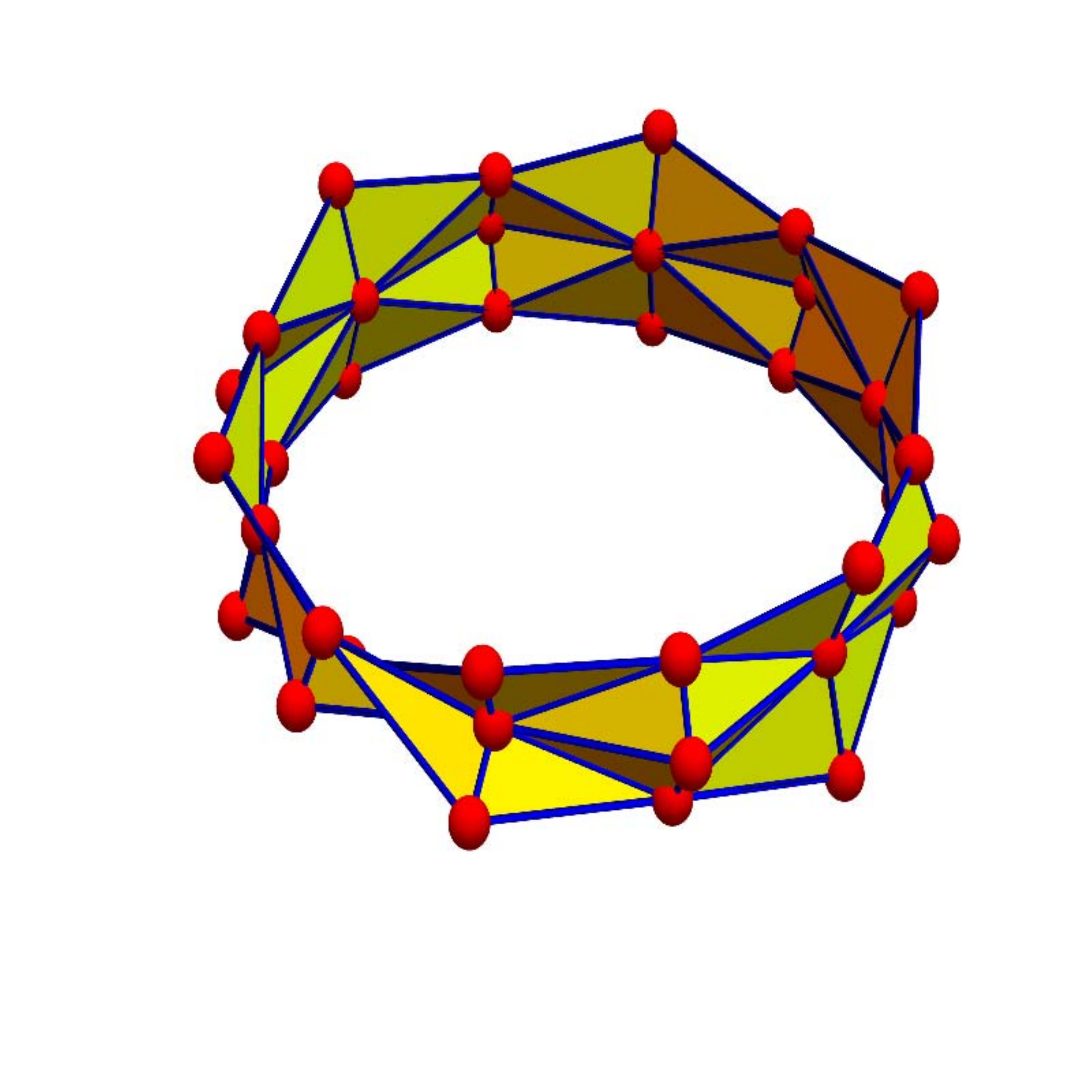}}
\caption{
The M\"obius bundle is an example of a nontrivial bundle. Locally, it is a Cartesian product
of $P_l \times P_k$ where $P_k$ is the $1$-dimensional
line graph with $k+1$ elements. But the graph itself is not orientable.
Its cohomology has the Betti vector $b=(1,1,0)$, like the circle.
While not distinguishable from the cylinder by cohomology nor homotopy
nor dimension, its topology is different as its boundary (the subgraph generated by the vertices for which the
unit sphere is not a sphere) is not connected and it is not orientable.
As in the continuum, a bundle can be constructed by taking a base graph $G$, cover it
with a nice open cover such that its nerve is homotopic to $G$. Now build the product
graphs and make sure that the transition maps are homeomorphisms in the sense  of
\cite{KnillTopology}. }
\end{figure}

A Cartesian product for a category of a geometry needs to be dimension-additive, it needs to 
induce a product on the homotopy classes, it needs to be Euler characteristic multiplicative, 
it must satisfy the K\"unneth formula equating 
the tensor product of the cohomology rings with the cohomology ring of the product and it must have the property 
that the automorphism group contains the automorphism groups of the factors. \\

As for graphs, no previously defined product shares these properties. 
The standard Cartesian product $``\times"$  of the cyclic graph $C_4$ with $C_4$ for example is a 
graph of dimension $1$. Its vertices are the Cartesian product of the vertices and two
points $(x,y),(u,v)$ are connected, if $(x,u) \in E$ or $(y,v) \in E$. The 
cohomology of the standard Cartesian product has little to do with the cohomology of the 
factors: the Betti numbers of $C_4 ``\times" C_4$ for example is $(b_0,b_1)=(1,17)$ while the Betti 
number of our product is $(1,2,1)$, which is identical to the one of the two-dimensional torus in classical
topology. The dimension of the traditional product is $1$ while the dimension of our product is $2$.

\begin{figure}[h]
\scalebox{0.22}{\includegraphics{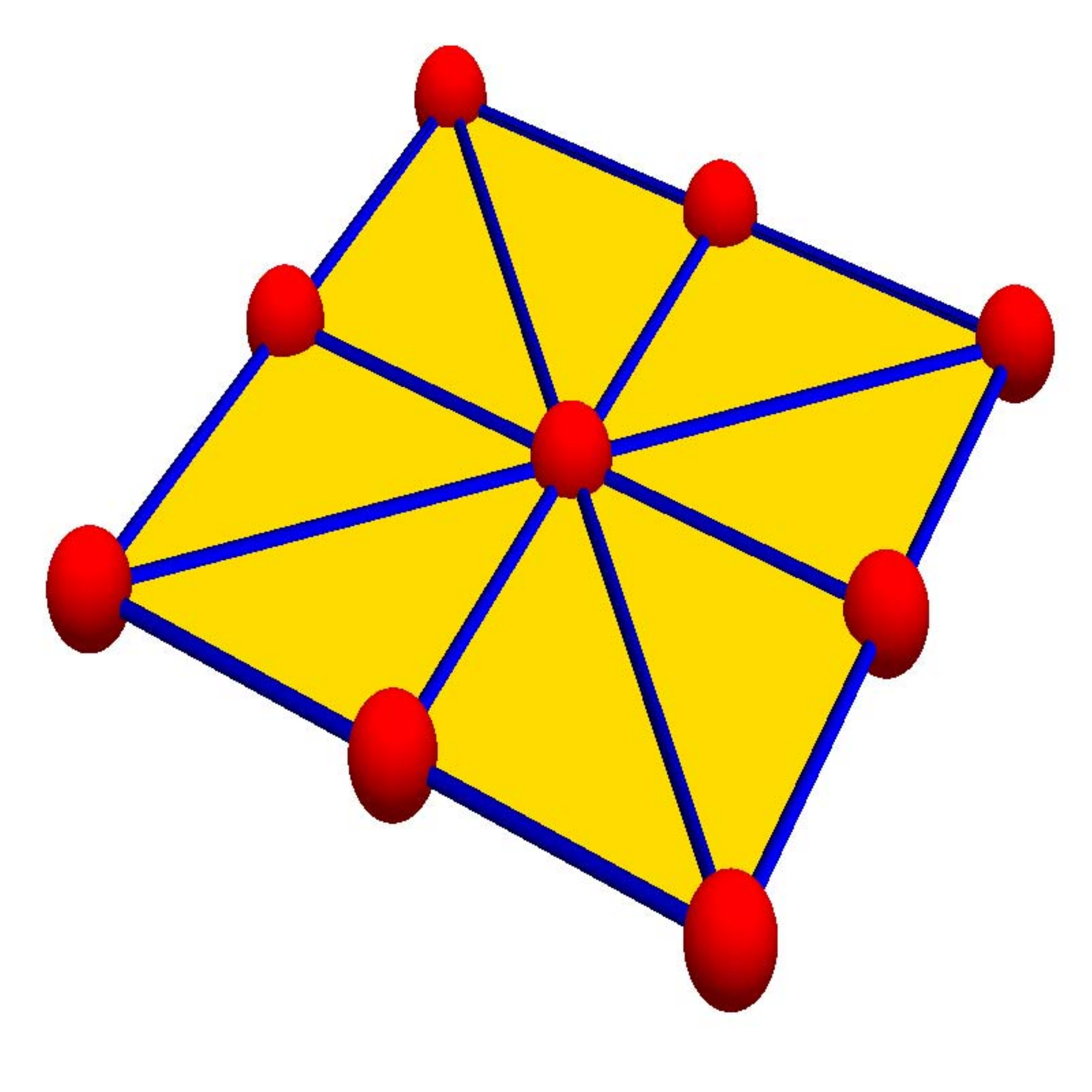}}
\scalebox{0.22}{\includegraphics{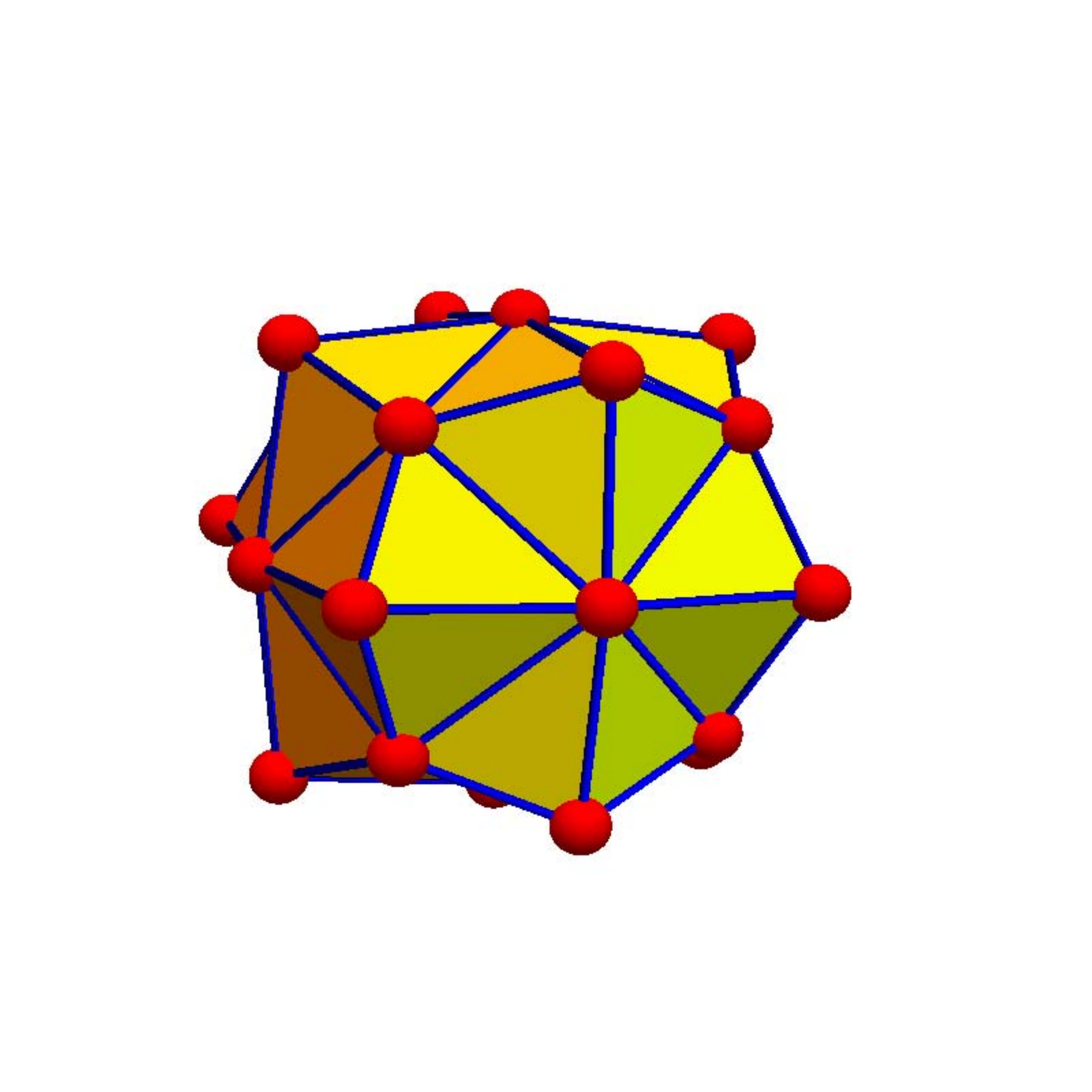}}
\scalebox{0.22}{\includegraphics{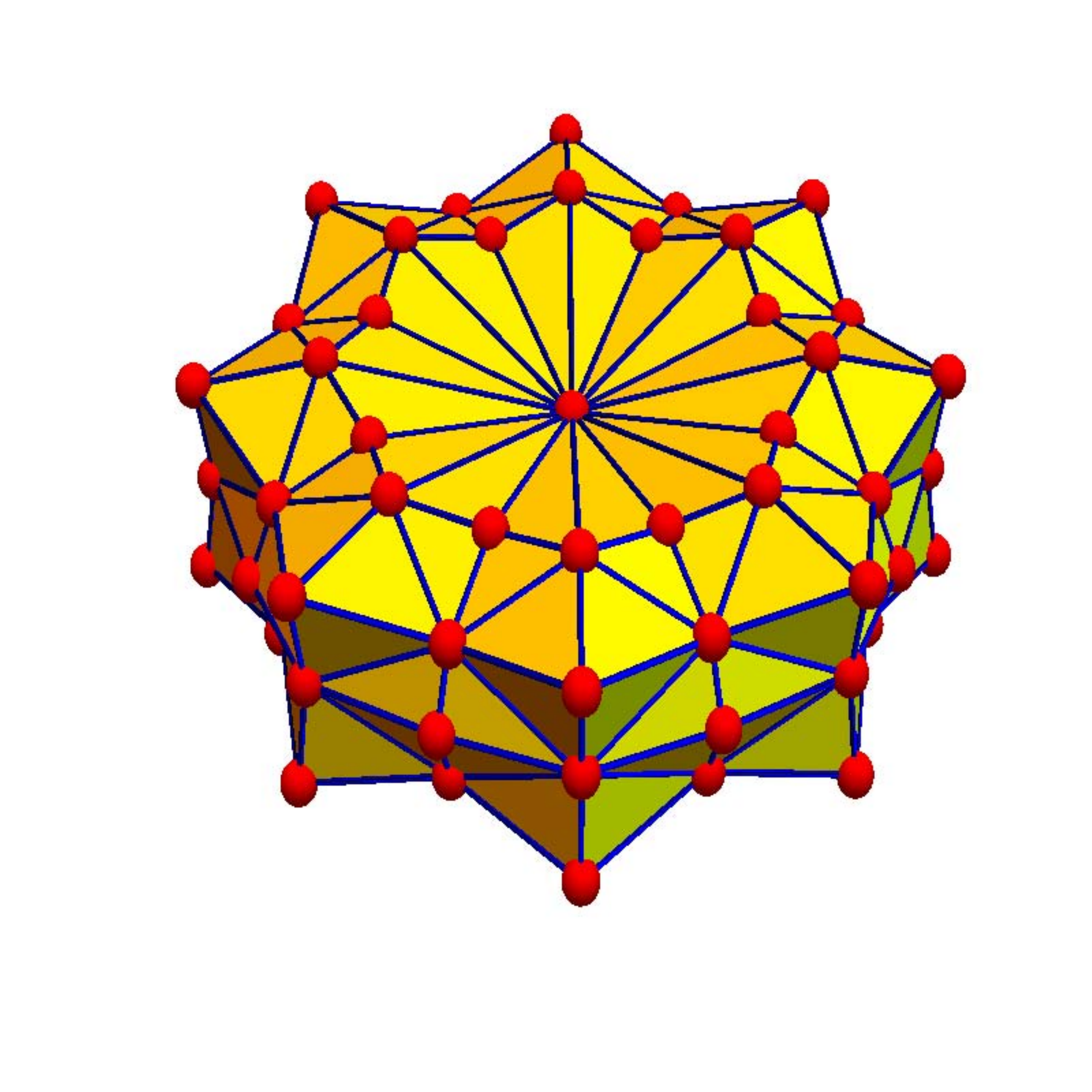}}
\caption{
The product $K_2 \times K_2$ is the wheel graph $W_8$. Its boundary is a $1$-sphere.
The graph $K_2 \times K_2 \times K_2$ is a $3$-dimensional ball with $27$ vertices and $98$ edges.
It resembles a cube as it is $3$-dimensional and has six $2$-dimensional faces of the form $K_2 \times K_2$.
The product $(K_2 \times K_2) \times K_2$ is a $3$-dimensional ball with $99$ vertices and $466$ 
edges. Its boundary is a 2-sphere. As $(K_2 \times K_2)$ has been obtained by returning back to geometry 
before computing the new product $(K_2 \times K_2) \times K_2$ is a refinement of $K_2 \times K_2 \times K_2$. 
similarly as $G \times K_1$ is a refinement of $G$. The boundary of $K_2 \times K_2 \times K_2$ by the way 
agrees with the enhancement $O \times K_1$ of the octahedron. 
}
\end{figure}

Also other constructions like the ``tensor product", the ``direct product" or the ``strong product" of 
graphs do not have the topological properties we want. The reason for the shortcomings of all these products is 
that they do not tap into the lower dimensional building blocks of space: these are the simplices = complete subgraphs 
in the graph case. What in the continuum has to be done with sheaf theoretical constructs, is already 
pre-wired in the graph as we can access the lower dimensional simplices as ``points". The new product 
has the property for example that $C_4 \times C_8 \times C_{11}$ is a $3$-dimensional torus of Euler 
characteristic $0$ and Betti vector $(1,3,3,1)$. It is a triangularization of $T^3$ and
each unit sphere $S(x)$ is a $2$-dimensional sphere of Euler characteristic $2$. The product has practical
use as it allows to construct high dimensional geometric spaces from smaller dimensional ones.
The product works for any pair of graphs and
leaves geometric $d$-dimensional graphs invariant, graphs for which all unit spheres are
$(d-1)$-dimensional homotopy spheres defined in \cite{knillgraphcoloring2}. 
The product will again have this property. Cohomology, dimension and homotopy properties of the product are
identical to the properties in the continuum. Even in the case of fractal dimension, the 
dimension formula matches the corresponding product formula in the continuum 
\cite{Falconer} (formula 7.2) for the Hausdorff dimension of arbitrary sets
in Euclidean space. 
There are more analogues: for Hausdorff dimension, there are sets of dimension zero for which the
product has dimension $1$. While graphs of dimension $0$ are geometric,
so that one gets equality, there are sequences of graphs like 
$G_n=x_0 x_1 +x_0+x_1 + \cdots + x_n$ for which ${\rm dim}(G_n) \to 0$ and 
${\rm dim}(G_n \times G_n) \to 1$, mirroring the continuum again. 

An other useful feature of the product is that it allows to 
refine graphs: the enhanced graph $G \times K_1$ is a barycentric refinement of $G$ and this can be repeated: 
the sequence $G \times K_1, (G \times K_1) \times K_1, \cdots$ produces a finer and finer mesh whose dimension converges
to an integer and honor the original symmetries if $A \subset {\rm Aut}(G)$ is a subgroup of the automorphism
group then $A$ still acts on the refinement. The later is important when looking at discrete principal bundles.  
Moreover, while $G/A$ is in general no more a graph, the quotient $G \times K_1^k/A$ is, so 
that we can as in the continuum form covering spaces ramified over rather general
subgraphs. For example, while $A=Z_n$ acting on $C_n$ has a quotient $C_n/A$ which is only a chain and no more a finite simple
graph, the quotient $(C_n \times K_1 \times K_1)/A$ is a graph. As an other example, the quotient
$O/A$ of the octahedron modulo the $Z_2$ action given by the antipodal reflection group $A=Z_2$ 
is no more a geometric graph without boundary as it is the wheel graph $W_4$ for which we can 
see the octahedron as a double cover ramified over the equator (Riemann-Hurwitz in the discrete
is just the Burnside lemma considered simultaneously for the various simplex sceletons as noted
in \cite{TuckerKnill}),
but $(O \times K_1 \times K_1)/A$ is a geometric graph, a discrete projective plane. 
As unit spheres of geometric graphs are spheres, the product can be used to construct new spheres. By taking refinements and then 
taking quotients, one can get more general discrete graphs like the projective plane in the simplest case.
The product construction can also help for other constructions, like constructing joins or building discrete Hopf fibrations
in arbitrary dimensions. As every unit sphere in the product has the structure $\delta B_1 \times B_2 
\cup B_1 \times \delta B_2$ which is the union of two solid tori glued along the torus $\delta B_1 \times \delta B_2$.
In the case when taking the product of $2$-dimensional geometric graphs, we get so $3$-dimensional 
spheres which have naturally the same Hopf fibration structure as in the continuum. \\

Classically, the notion of a homotopy of two continuous maps $f,g: X \to Y$ is defined using the product: 
if there is a continuous map $F: X \times [0,1] \to Y$, such that $F(x,0)=f(x)$ and $F(x,1) = g(x)$, then $f,g$
are called homotopic. This could be done now also for graphs: two graph homomorphisms $f,g: G \to H$ are
homotopic, if there exists a line graph $L_n$ and a graph homomorphism $F$ from $G \times L_n$ to $H$ such 
that $F(x,0)=f(x)$ and $F(x,n) = g(x)$. We have suggested in \cite{knillgraphcoloring} an other definition: 
two graph homomorphisms $f,g: G \to H$ are homotopic if the ``graph of the graph homomorphisms" are homotopic
graphs. These graphs of homomorphisms have as vertices the union $V(G) \cup V(H)$ and as edges 
all pairs $(x_1,x_2) \in E(G), (y_1,y_2) \in E(H)$ and pairs $(x,y)$ with $y=f(x)$. 
Two graph homomorphisms are now homotopic, if the corresponding graphs are 
homotopic as graphs. We believe that these two definitions are equivalent, but have not yet proven this. 
In any case, we have the Whiteheads theorem that if there is a graph homomorphism $f:G \to H$ which induces
isomorphisms on homotopy groups $\pi_n(G) \to \pi_n(H)$, then $G,H$ are homotopic. And as in the
continuum, having isomorphic homotopy groups does not force a homotopy equivalence. Now, with a product
we can take the same standard counter example $G=P^3 \times S^2$ and $H=P^2 \times S^3$, where $P^k$ are
graph implementations of the $k$-dimensional projective space and $S^k$ are $k$-spheres. Since they have the
same universal cover $S^3 \times S^2$ and the same fundamental group $Z_2$, they have the same homotopy groups
but the K\"unneth formula implies that the Poincar\'e polynomial of $G=P^3 \times S^2$ is $p_G(x) = 1 (1+x^2)$
while the Poincar\'e polynomial of $H=P^2 \times S^3$ is $1 (1+x^3)$. Having different cohomology groups
prevents the two graphs $G,H$ to be homotopic. We see in this example, how useful it is to have a product
in graph theory which shares the properties from the continuum. The point is that one does not have to 
reinvent the wheel in graph theory but that one can piggy-pack on known topology. \\

\begin{figure}[h]
\scalebox{0.16}{\includegraphics{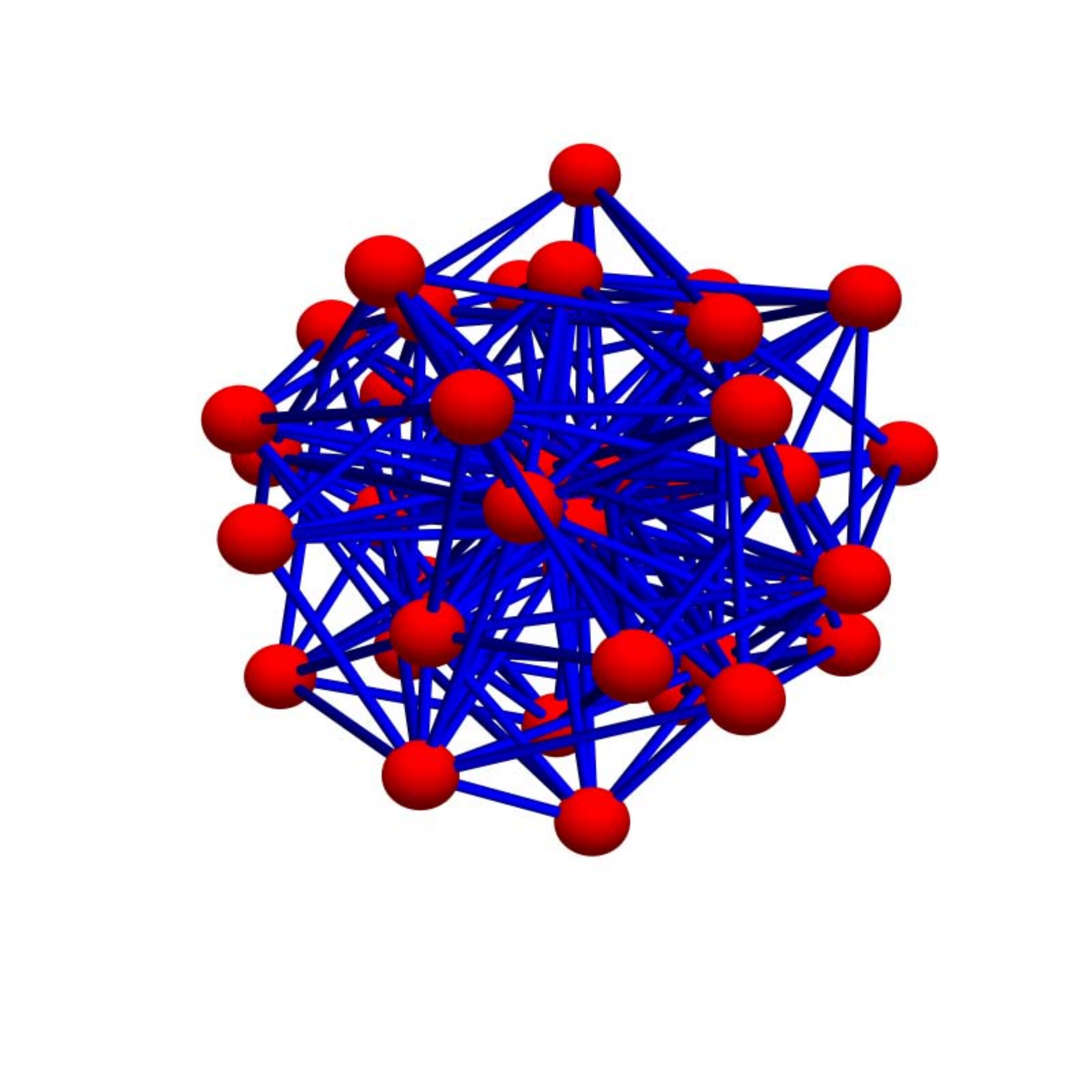}}
\scalebox{0.16}{\includegraphics{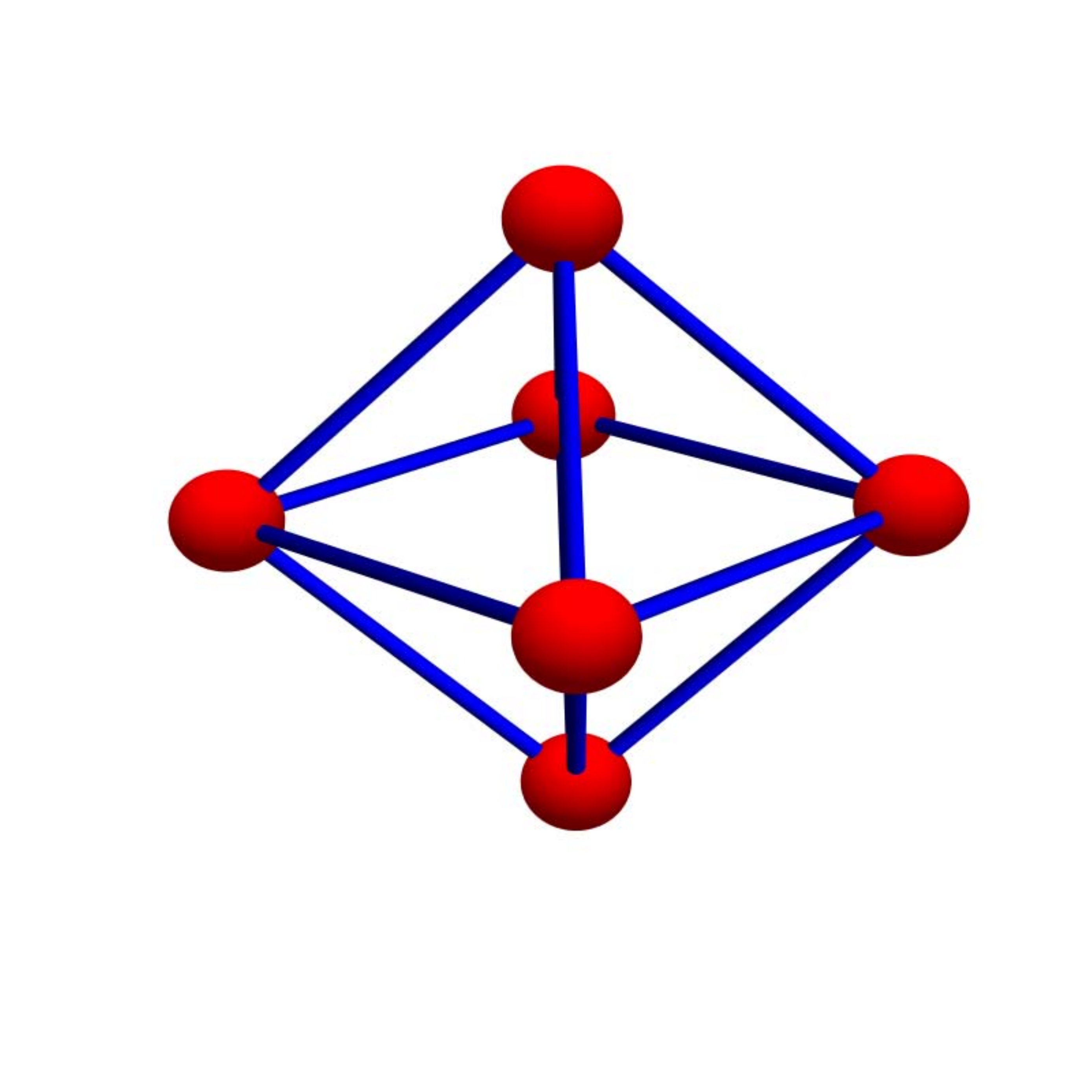}} 

\scalebox{0.16}{\includegraphics{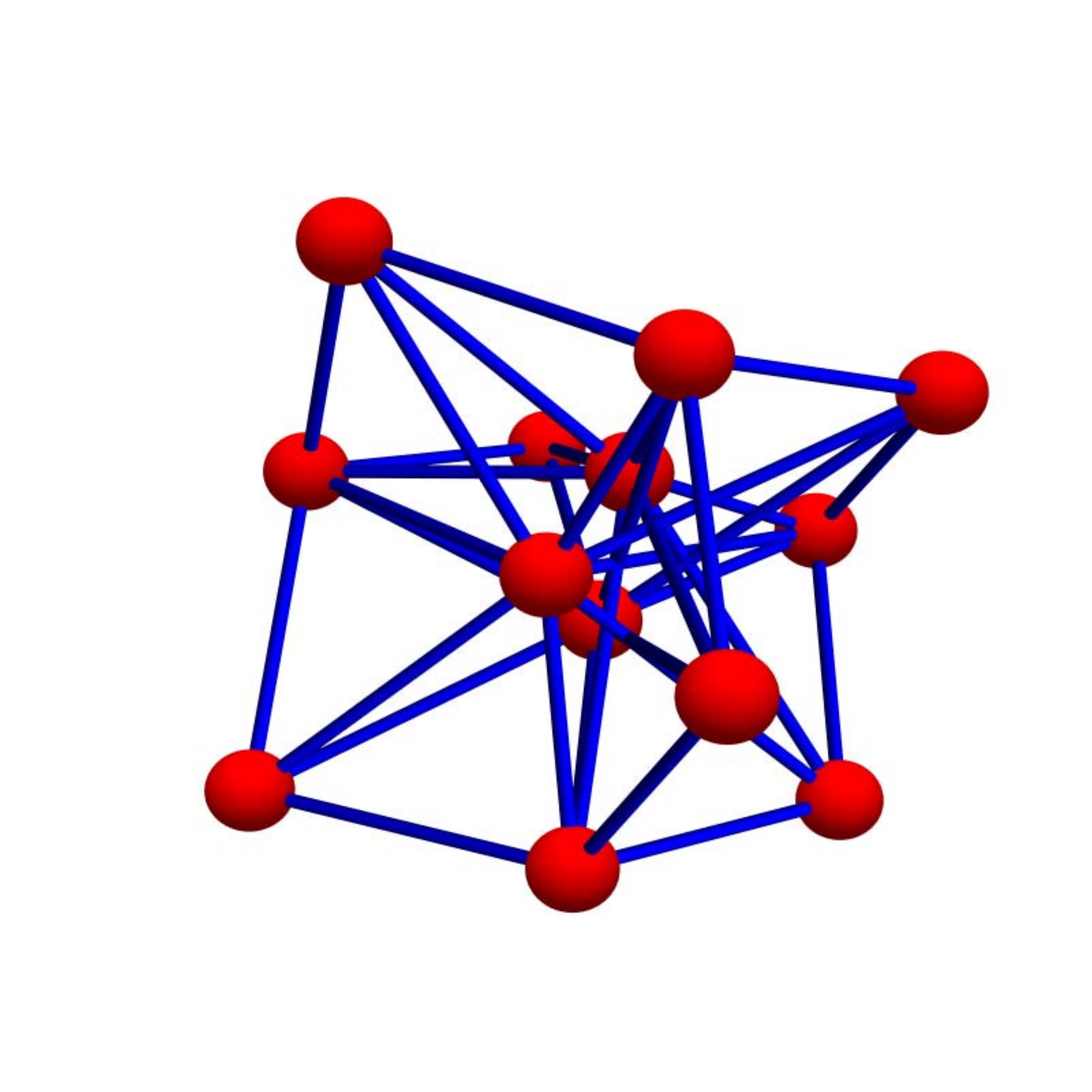}}
\scalebox{0.16}{\includegraphics{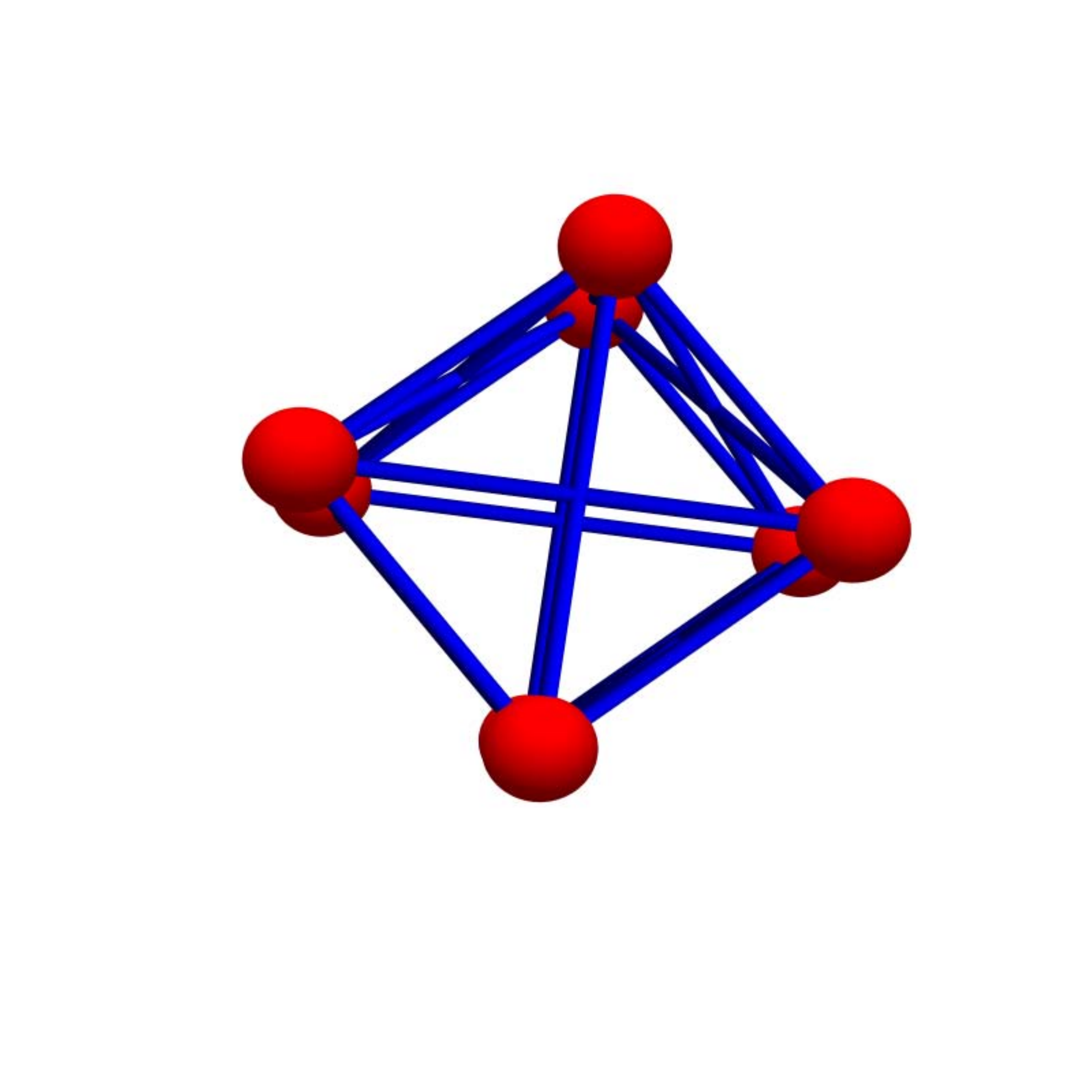}}
\caption{
The product $G$ of $P^3$ with $S^2$ (the two factors are seen in the upper row) 
gives a $5$-dimensional graph which has the same homotopy group 
than the product $H$ of $P^2$ with $S^3$ (the two factors are seen in the lower row)
which is also 5 dimensional.
The graph $G$ has $26*848=22048$ vertices while 
$H$ has $73*80= 5840$ vertices. Both graphs have $Z_2$ as fundamental
group and $S^2 \times S^3$ as universal cover. But the cohomology of
$G$ is the one of $S^2$ and the cohomology of $H$ is the one of $S^3$
by the K\"unneth formula. By the way, we constructed the
projective spaces $P^k$ by taking a refined cross polytop $O_k$ which 
is a $k$-dimensional sphere, then forming $O_k \times K_1$ which is
large enough so that we can factor out the antipodal involution obtaining
the projective space. This example illustrates the use of the product
in concrete graph theoretical constructions shadowing exactly the theory
in the continuum. 
}
\end{figure}

From a practical point of view, the construction of the product only needs a few lines of code for a 
standard computer algebra system. The full computer code is given in detail later on. The product 
works for all finite simple graphs: first construct the ring elements 
$f_G(x),f_H(y)$ which are polynomial in $x_1,\dots,x_n$ representing the vertices of $G$
and $y_1,\dots,y_m$, the vertices of $H$, then take the product $f_G f_H$ and construct
from this a polynomial in the variables $x_1,\dots,x_n,y_1,\dots,y_m$, the new graph $G \times H$ by connecting two monomial 
terms of a polynomial if one divides the other. For example, if $G=K_2$ then $f_G=x+y+xy$ and if $H=K_2$ then
$f_H = u+v+uv$, we get $f_G f_H = ux+vx+uvx+uxy+vxy+uvxy+uy+vy+uvy$ which encodes a wheel graph $W_8$ with central 
vertex represented by $uvxy$. We see that the product naturally extends to chains, group elements in the ring 
and that it corresponds to the product in the tensor product of the two rings. Every
polynomial defines back a graph, but the later does in general not not agree with the original host graph. 
For example, take $G=K_2$ and the chain element $3x + 5 y + 10 xy$, which defines in turn a graph which is the
disjoint union of $K_1$ (represented by $3x$) and $K_2$ represented by $5y+10 xy$. Chains are useful constructs
because natural operations escape the class of graphs: examples are forming the boundary $\delta G$ or
taking quotients $G/A$ with a subgroup $A$ of the automomorphism group of $G$. Chains form a ring
and have properties wanted from a geometric object: a dimension, a cohomology, a notion of homotopy, an 
Euler characteristic, a notion of curvature and these notions match the results we expect from the familiar cases:
Results like Gauss-Bonnet \cite{cherngaussbonnet}, McKean Singer and Hodge-De-Rham \cite{knillmckeansinger}, 
Poincar\'e-Hopf \cite{poincarehopf}, Brower-Lefschetz \cite{brouwergraph}, 
Riemann-Roch (\cite{BakerNorine2007} is a strong enough theory so that it can be extended to higher dimensions),
integrable geometric evolutions \cite{IsospectralDirac,IsospectralDirac2},
Riemann-Hurwitz \cite{KnillBaltimore}, Lusternik-Schnirelmann \cite{josellisknill}, 
mirror the results in the continuum. See also \cite{KnillILAS} for an overview from the linear algebra point
of view. These results are more limited, if one considered graphs as one-dimensional simplicial complexes only, 
a common assumption taken in the 20th century.
As graph theory is a discrete theory, it can surprise at first that the situation parallels the continuum so well
but there are conceptual reasons for that, an example is non-standard analysis, an other is 
integral geometry. Some notions go over pretty smoothly, like cohomology or homotopy: others, 
like the notion of homeomorphisms for graphs needs more adaptations \cite{KnillTopology}. \\

About the history: K\"unneth found the formula in 1921 \cite{KuennethGedenken}, 
where it was his dissertation under the guidance of Heinrich Tietze. The paper published in 1923 
\cite{Kuenneth} is a linear algebra analysis which could probably be simplified considerably 
using Hodge theory. K\"unneth attributes the Cartesian product of manifolds to Steinitz (1908) 
(The definition is indeed given in \cite{Steinitz} p.44 is probably the first appearance of the
simplicial product which our definition is based on).
The notion of manifolds was initiated by Poincar\'e (1895), Weyl (1912),
Veblen and Alexander \cite{VeblenAlexander} (1913) and Whitney \cite{WhitneyManifolds} (1936). 
It was Herbert Seifert who introduced fibre bundles in 1933 \cite{Seifert}. The Eilenberg-Zilber
theorem of 1953 \cite{EilenbergZilber}. (Joseph Abraham Zilber was a Boston born 
Harvard graduate (1943) and PhD (1963). Interestingly, the Eilenberg-Zilber theorem
was authored 10 years before Zilber got his PhD degree under the guidance of Andrew Gleason. 
More information about the history of algebraic topology, see \cite{Dieudonne1989,Scholz}
or the introduction to \cite{Milnor}. 
The classical Cartesian product of graphs was introduced by 
Whitehead and Russell in Principia Mathematica 1912 (it is \cite{ImrichKlavzar} who spotted the construction
on page 384 in Volume 2 of that epic work). It was introduced there in the context of logical relations, 
not so much graph theory. 
This historical observation and many properties of the classical product product are discussed in 
\cite{ImrichKlavzar}. \\

As far as we know, our present paper is the first establishing a K\"unneth formula for finite simple graphs. 
There is a functorial approach to K\"unneth for digraphs, where
\cite{HuangYau2014} use path cohomology to get a functor from digraphs to CW complexes, so that
one can then use the continuum result for the CW complexes. Note however that in their case,
one gets K\"unneth only indirectly by constructing a CW-complex, take the product using the Cartesian
embedding and then pulling the result again to graphs. In particular, the notion of dimension is also
borrowed from the continuum as CW-complexes are topological spaces.
A similar thing could be done for geometric graphs, graphs for which the 
unit spheres $S(x)$ are homotopy $n$-spheres. One can then build for each ball $B(x)$ an open set 
in ${\bf R}^n$ and use this to build a manifold $M$. The open sets form then a nice good cover and the 
topology of the manifold is the same than the topology of the geometric graph. This construction is
restricted however to geometric graphs. \\

The missing Cartesian product bothered us for a while so that we decided to make a targeted search over several
weeks (while procrastinating from an urgent programming job still in need to be finished in geometric graph coloring), 
trying several possibilities and checking with the computer whether the construction works. 
The product described here was obtained by trying random things which look ``beautiful".
The current product is attractive when seen algebraically because it becomes associative on the algebraic level. 
Until now, we would just take the usual product and then fill out the chambers. When working with graphs, 
this is cumbersome, even when staying within the discrete, as it costs programming effort to build 
examples, like stellated higher dimensional cubes. The product which we propose here for graphs 
seems not to be known in the language of graphs. One could of course take the product $|X| \times |Y|$ 
of the corresponding topological spaces as K\"unneth did. After finding the product, we looked around whether it 
already exists: closest to what we do here appeared as an 
exercise in an algebraic topology course by Peter Tennant Johnstone at Cambridge \cite{PTJ}. 
More digging revealed that this simplicial product has been used by Eilenberg and Zilber in 
\cite{EilenbergZilber} (page 204), by De Rham \cite{DeRham1931} (page 191) and earlier 
by Steinitz in \cite{Steinitz} (page 44). 
But this product never made it to graph theory. One could also get a product by escaping to the Euclidean
space. Our reluctance to use Euclidean stuff is not because we feel like Brouwer (who would even refuse to
accept the infinity of natural numbers), 
but simply out of pragmatism: we want to build the structures fast on a computer without having to 
use Euclidean parts. Graphs are natural structures built into computer algebra languages
and the Euclidean embeddings are not needed, when doing computations; only if we want to see the graphs visualized
or modeling traditional geometric objects, the Euclidean embedding is helpful.  
Here is a self-contained full implementation of the new graph product in ``Mathematica", a computer algebra system
which uses graphs as fundamental objects in its core language, structures which are dispatched from 
Euclidean embeddings, unless they are drawn. One procedures allow to translate a graph into a ring element
and an other allows to get from a ring element back a graph. The graph product is the product in the polynomial ring.  

\vspace{12mm}

\begin{tiny}
\lstset{language=Mathematica} \lstset{frameround=fttt}
\begin{lstlisting}[frame=single]
Cliques[s_,k_]:=Module[{n,t,m,u,q,V=VertexList[s],W=EdgeList[s],l}, 
 n=Length[V]; m=Length[W]; u=Subsets[V,{k,k}]; q=Length[u]; l={};
 W=Table[{W[[j,1]],W[[j,2]]},{j,m}];If[k==1,l=Table[{V[[j]]},{j,n}],
 If[k==2,l=W,Do[t=Subgraph[s,u[[j]]]; If[Length[EdgeList[t]]==
 Binomial[k,2],l=Append[l,VertexList[t]]], {j,q}]]];l];

Ring[s_,a_]:=Module[{v,n,m,u,X},v=VertexList[s]; n=Length[v];
 u=Table[Cliques[s,k],{k,n}] /. Table[k->a[[k]],{k,n}];m=Length[u];
 X=Sum[Sum[Product[u[[k,l,m]],
   {m,Length[u[[k,l]]]}],{l,Length[u[[k]]]}],{k,m}]];

GR[f_]:=Module[{s={}},Do[Do[If[Denominator[f[[k]]/f[[l]]]==1 && k!=l,
 s=Append[s,k->l]],{k,Length[f]}],{l,Length[f]}]; 
 UndirectedGraph[Graph[s]]];

GraphProduct[s1_,s2_]:=Module[{f,g,i,fc,tc}, 
 fc=FromCharacterCode; tc=ToCharacterCode;
 i[l_,n_]:=Table[fc[Join[tc[l],IntegerDigits[k]+48]],{k,n}];
 f=Ring[s1,i["a",Length[VertexList[s1]]]];
 g=Ring[s2,i["b",Length[VertexList[s2]]]]; GR[Expand[f*g]]];

NewGraph[s_]:=GraphProduct[s,CompleteGraph[1]];

example = GraphProduct[CompleteGraph[3],StarGraph[4]]
\end{lstlisting}
\end{tiny}

\vspace{12mm}

\begin{figure}[h]
\scalebox{0.46}{\includegraphics{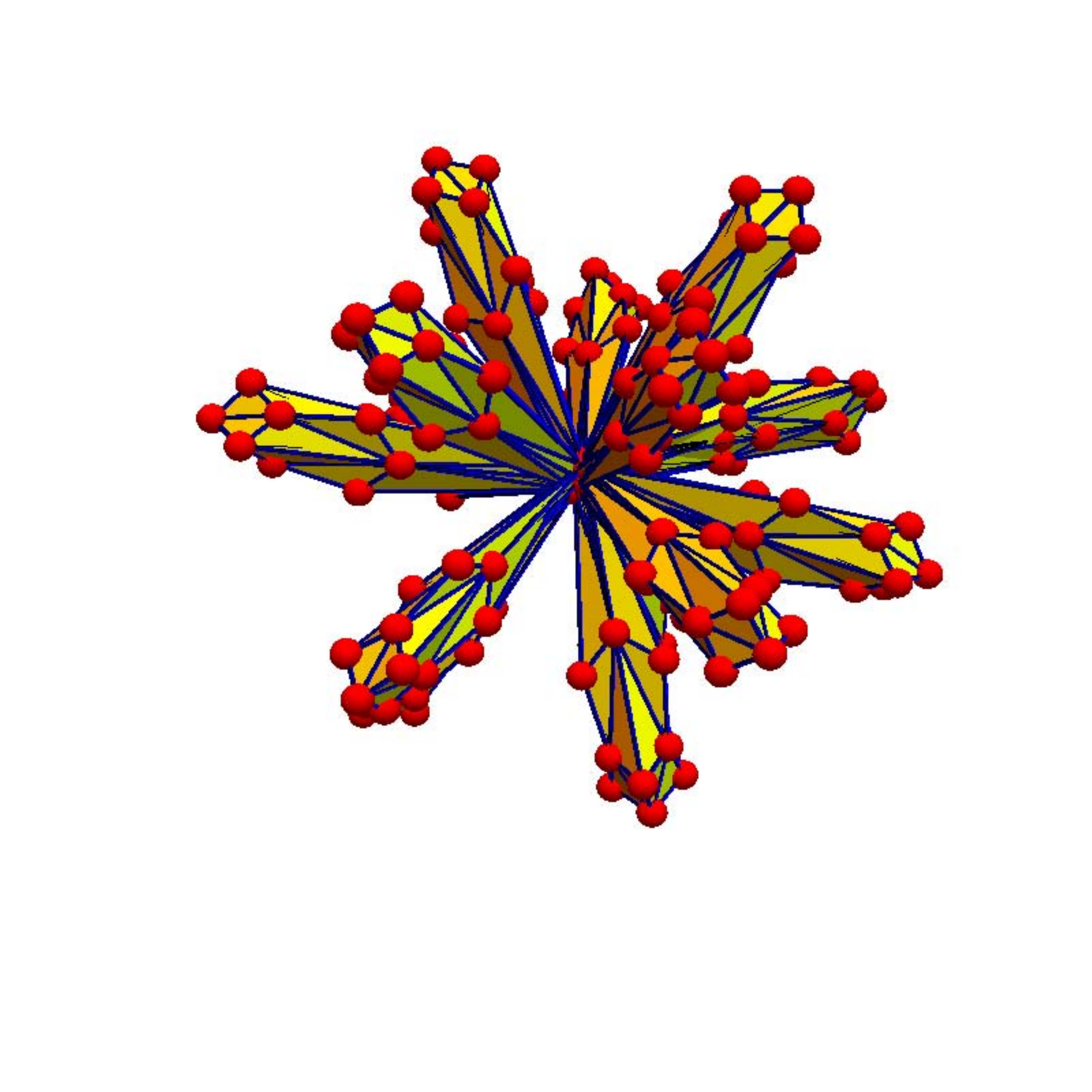}}
\caption{
When invoking the product, like with 
GraphProduct[CompleteGraph[3],StarGraph[8]] 
the program
first computes the algebraic expressions in the ring, which are
$f=a_1 + a_2 + a_1 a_2 + a_3 + a_1 a_3 + a_2 a_3 + a_1 a_2 a_3$ and
$g=b_1 + b_2 + b_1 b_2 + b_3 + b_1 b_3 + b_4 + b_1 b_4 + b_5+ b_1 b_5+ b_6+ b_1 b_6 + b_7 + b_1 b_7 + b_8 + b_1 b_8$,
containing $7$ and $15$ polynomial monoid entries, 
then multiplies them $fg$ and derives from this again a graph 
$G_{fg} = G \times H$.
The product of the $2$-dimensional triangle $G$ and $1$-dimensional
star graph $S_8$ is the graph $G \times H$ displayed in the figure. 
It is $3$-dimensional with $v_0=105$ vertices, 
$v_1=203$ edges, $v_2=1182$ triangles and $v_3=1083$ tetrahedra. 
Because both factors are homotopic to $K_1$, also the product is
contractible and has the Poincar\'e polynomial $p_{G \times H}(x)=1$. 
}
\end{figure}

The code of the above listing can be grabbed by looking at 
the source of the ArXiv submission of this text. \\
Can we write the usual product in an algebraic way? Yes, 
if we have the ring element $f$ and take as vertices the $0$-dimensional simplices and 
as edges the $1$-dimensional edges, we recover the old graph. 
If we take the product $fg$ and take as vertices the pairs $xy$ and as edges
the triples $xyz$, we regain the standard Cartesian product. 
For example, if $f=xy+x+y$ represents $K_2$ and 
$g=u v + v w + w u + u +v + w$ represents $K_3$, then the quadratic and cubic terms
of $fg$ are $ux + vx + wx + uy + vy +wy  uvx + uwx + vwx + uvy + uwy + vwy +  uxy + vxy + wxy$
give the product graph which has $6$ vertices and $9$ edges. 

\section{Construction}

In this section, we describe that the vertices of a graph define a ring in which every element $f$ can be 
seen as its own geometric object which carries cohomology, homotopy, Euler 
characteristic, curvature, dimension and a Dirac operator. As we could run the wave or heat equation
on such a ring element $f$, the chain $f$ should be considered a geometric object with physical context. 
Unlike the category of graphs, the category of these
chains has an algebraic ring structure, is closed under boundary formation as well as taking 
quotients which leads to orbifolds in the continuum. Other constructs in the continuum like 
discrete stratifolds are already implemented as suspensions of geometric graphs and are represented by classical graphs.
We called them exotic in \cite{knillgraphcoloring2}, where we asked whether they might lead to exotic discrete spheres 
(a $k$-dimensional graph with geometric unit spheres which is homeomorphic to a $k$-dimensional homotopy sphere) 
and looked at discrete varieties in \cite{knillgraphcoloring2}.
Somehow, chains form are part of a  list of class of structures which resemble structures in the continuum
geometric graphs $\subset$ stratifolds $\subset$ varieties $\subset$ graphs $\subset$ orbifolds $\subset$ chains.
The product goes all the way to chains and can also be used to lift the 
classical Cartesian product to $1$-dimensional chains: it is obtained by taking the 
product and disregarding higher dimensional parts of the chain: kind of projecting the result onto curves. \\

We use multi-index notation $a_k x^k = a_{k_1, \dots ,k_n} x_1^{k_1} x_2^{k_2} \dots x_n^{k_n}$. 
A finite simple graph $G$ with vertices $x_1, \dots ,x_n$ defines a ring
generated by all non-constant polynomial monoids $x = x_{k_1} x_{k_2} \cdots x_{k_l}$ 
and the chain ring. With a given orientation, 
graphs are always represented by functions of the form $\sum_x x^n$ with $n_i \in \{0,1\}$
and $\sum_i n_i \neq 0$. The star graph for example is $x+y+z+w+xw+yw+zw$. 
The ordering of the terms in the monoid
allows in a convenient way define an orientation of the simplices in the graph $G$ which 
means fixing a basis for discrete differential forms $\Omega^k(G)$. \\

We could add constants to get a ring with $1$ but we don't yet see a use of
the constants for now as we have no geometric interpretation for it. The Euler characteristic formula 
$\chi(f)=-f(-1,\dots,-1)$ shows that the Euler characteristic of $1$ is $-1$. 
The $1$ element can't be interpreted as the empty graph, because the empty graph is 
represented by $0$, and is a $(-1)$-dimensional sphere with Euler characteristic $0$. With a $1$ element in the ring, 
one can produce terms like $(1+x) (1+y) = 1+x + y + xy$, an 
object of Euler characteristic $0$ or $(2+xy) (1+zw) = 2+xy+2zw+xyzw$, an object of Euler characteristic $-6$.  \\

To add an other footnote, we see that some chains represent already special differential forms, but these are
forms taking values in the integers. The triangle $xyz+xy+yz+zx+x+y+z$ is a sum of a $2$-form, a $1$-form and 
$0$-form which are all constant $1$.  
In this discrete setup, geometric objects and differential forms are already very similar, as it is custom 
for quantum calculus setups, where Stokes is just the statement $\langle \delta f,g \rangle = \langle f, d g \rangle$
and where the boundary operation is truly the adjoint of the exterior derivative and both geometric objects and
forms are in the same function space. For classical differential forms and geometric objects, 
the geometric objects are distributions (as curves and 
surfaces for example are infinitely thin) and differential forms 
are smooth or the dual setup used in geometric measure theory where differential forms are distributions and 
smooth functions are geometric objects. Its only on the level where
both parts (geometric objects and differential forms) are represented by the same type of $L^2$ data that we have
true symmetry, but then we are in a quantum setup. \\

Every element $f = \sum_n a_n x^n$ defines a graph $G_f$: the vertices of $G_f$
are the monomials of $f$ and two monomials of $f$ are connected by an edge if one is a factor of the other. 
The graph of $f=4 x y + 2 x + y - 3x$ for example is the union of a line graph with three vertices where $4xy$ represents
the middle vertex, as well as a single $K_1$ represented by $3x$. 
The graph $G$ on the other hand defines the ring element $f_G=\sum_{x} x$, 
where $x=x_{j_1} \dots x_{j_k}$ is a simplex in $G$. The choice of the sign or permutation 
when writing down the polynomial monoid components corresponds to a choice of basis and is irrelevant
for most considerations like for computing cohomology or when running discrete differential equations
like \cite{IsospectralDirac,IsospectralDirac2}.

\begin{figure}[h]
\scalebox{0.12}{\includegraphics{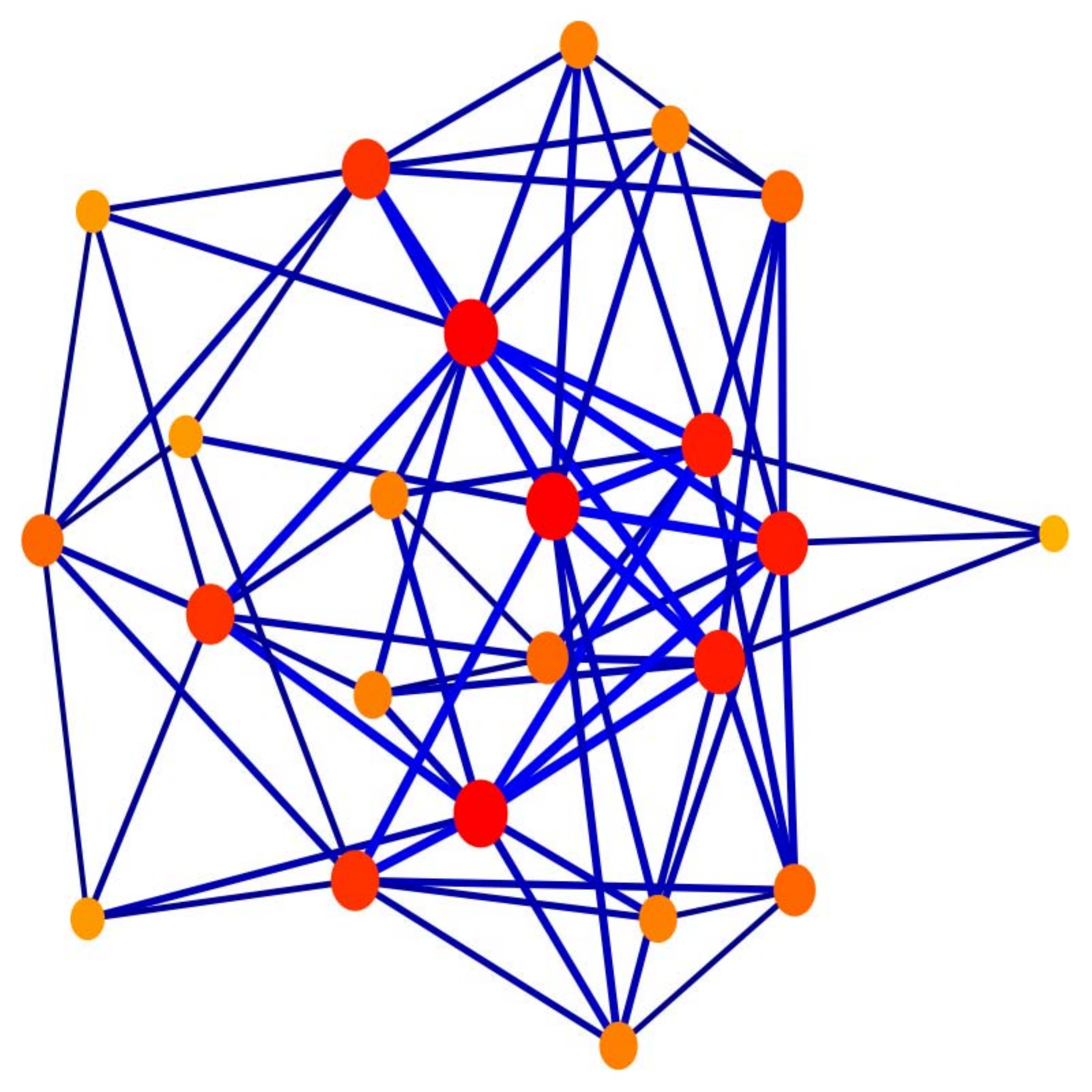}}
\scalebox{0.12}{\includegraphics{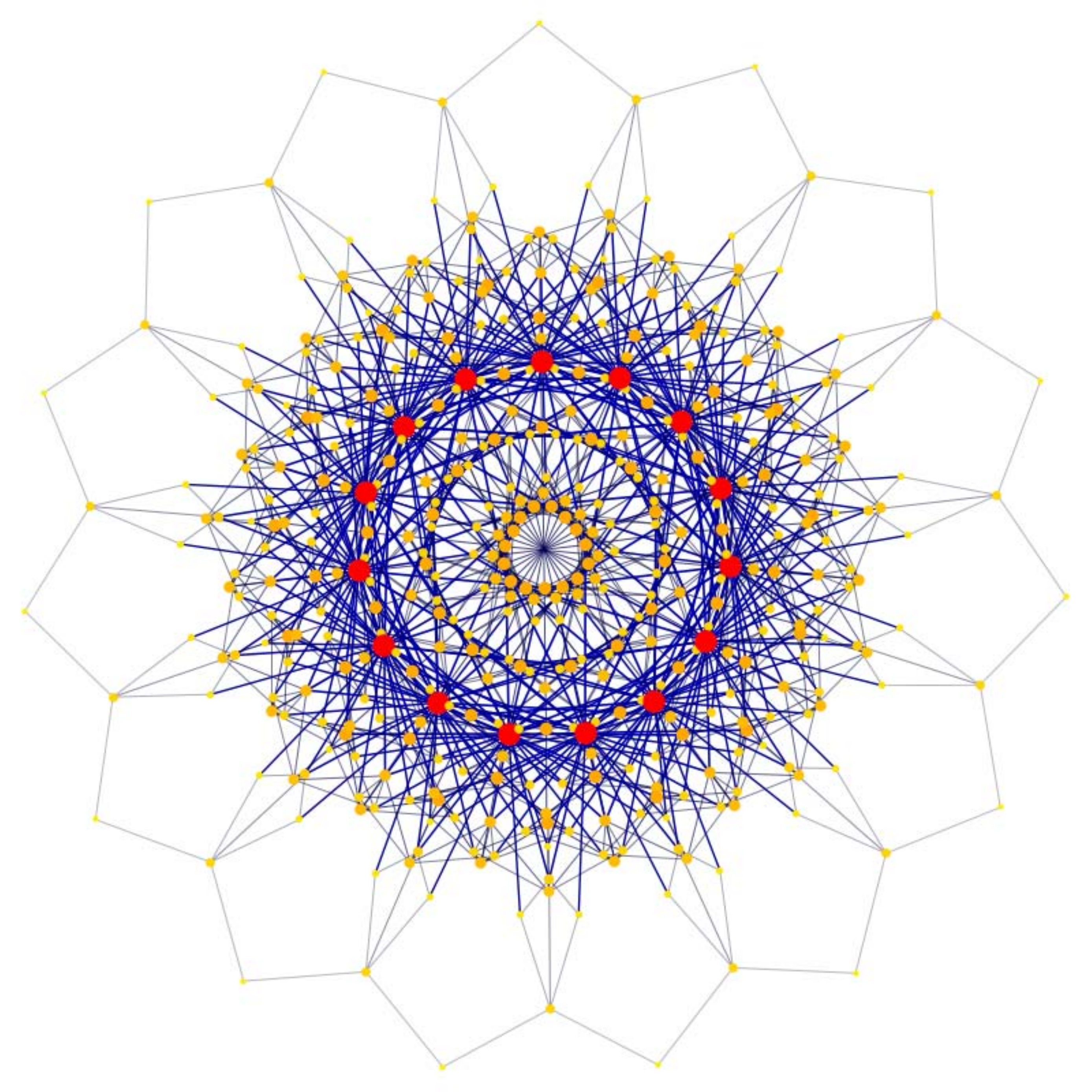}}
\scalebox{0.12}{\includegraphics{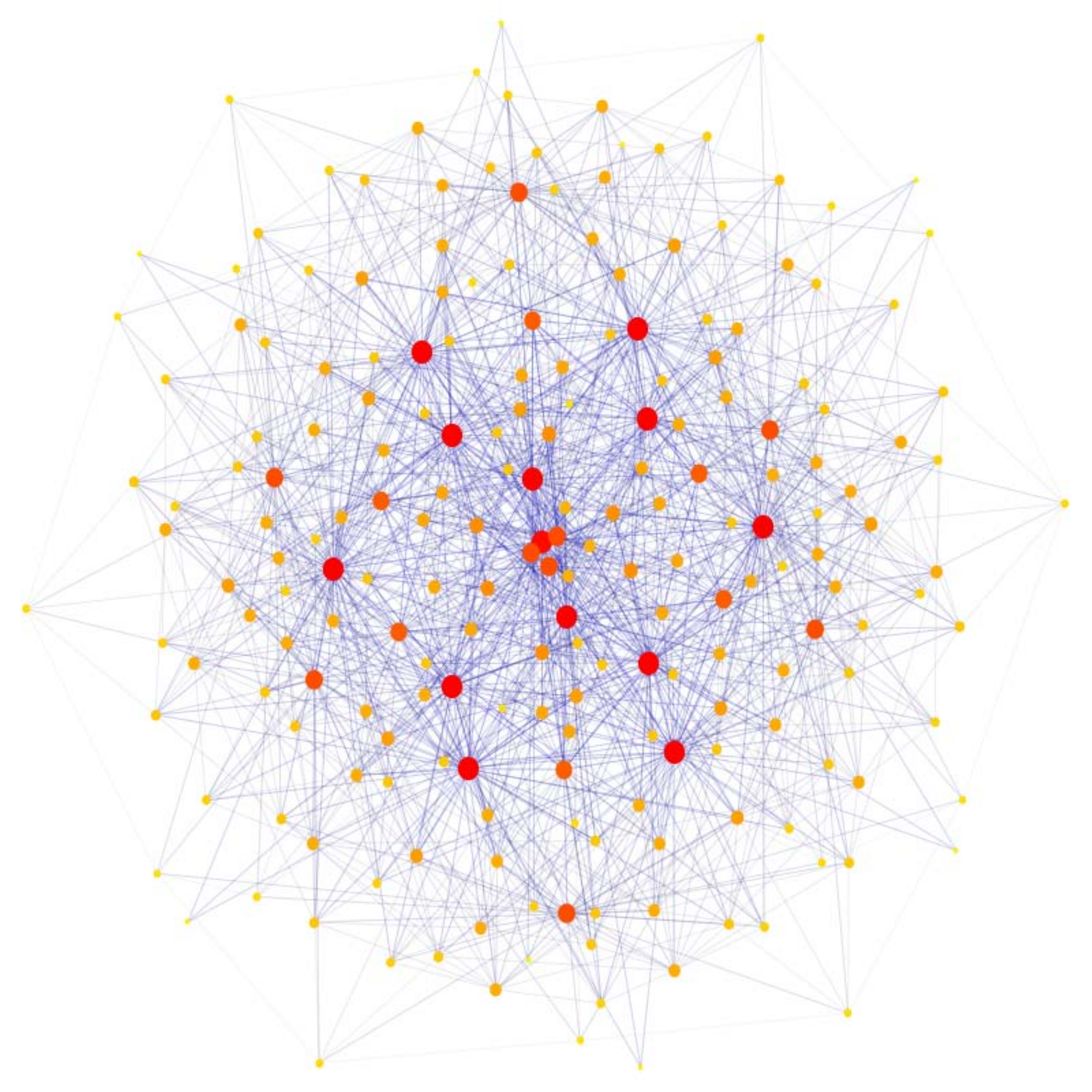}}
\caption{
The idea to associate a graph from a ring element can be done more generally. Let $G$ be a graph.
It defines a ring element $f$ in the ring $Z[x_1,\dots,x_n]$, which is now abelian. 
We can form $f^2$ and look at the graph which this function defines. 
The figures above show examples for the triangle $K_3$, for the circle $G=C_{15}$ and for the octahedron $G=O$. 
For $G=K_3$ in particular, we get the ring element $f=a + b + c + a b + a c + b c + a b c$ and 
compute $f^2=a^2 b^2 c^2+2 a^2 b^2 c+a^2 b^2+2 a^2 b c^2+4 a^2 b c+2 a^2 b+a^2 c^2+2 a^2 c+a^2+2 a b^2
c^2+4 a b^2 c+2 a b^2+4 a b c^2+6 a b c+2 a b+2 a c^2+2 a c+b^2 c^2+2 b^2 c+b^2+2 b
c^2+2 b c+c^2$ which has has $23$ monoidal summands. The incidence relations through 
factorization produces a graph with $23$ vertices. 
}
\end{figure}

After fixing an orientation for each simplex, we get
incidence matrices $d_k$. They implement the exterior derivative. They also determine 
the  Dirac operator $D=d+d^*$, and all such matrices are  unitarily equivalent.
Independent of the basis is the form Laplacian $L=D^2$ which has a block decomposition into 
Laplacians  $L_k$ on $k$-forms. 
The enhanced graph $G_1=G_{f_G}$ is a refinement of $G$ and as we will see, if $G$ is geometric, then the refinement 
is geometric again of the same dimension. For general networks $G$ we will show that the dimension 
of $G_1$ can only increase, the reason being that higher dimensional parts will spawn off more ``new vertices". 
For a triangle $G$ for example, the ring is $Z[x,y,z]$ and
$f_G = x+y+z+xy+yz+xz+xyz$. This ring element defines the graph $H=G_{f_G}$ for which the monomials
$x,y,z,xy,yz,xz,xyz$ are the vertices. The divisor incidence condition leads to the
wheel graph $W_6$ which has the same topological features as $G$. Of course, the refinement process can be repeated, 
leading to larger and larger graphs with the same automorphism group. 
The sequence of graphs $G_1,G_2,G_3,\dots$ define so larger and larger rings generated by more and more variables. 
We will see that the dimension converges to an integer and that for large $n$, the graph is close to a geometric
graph as the larger simplices will overtake all others. See Figure~(\ref{euclideanlimit}).  \\

Given two graphs $G,H$, define the rings $X_G,X_H$,  
take the product $fg$ in the tensor product $X \otimes Y$ of rings and translate that product ring element 
back to a graph. The dimension of $G$ is defined inductively as $1$ plus the average of the dimensions 
of the unit spheres and uses the foundation assumption that the empty graph has dimension $-1$.
A graph is geometric, if every unit sphere is a homotopy sphere. A homotopy sphere is a
geometric graph for which removing any vertex renders the graph contractible. Inductively, a graph $G$
is called contractible, if there exists a vertex $x$ such that $S(x)$ and $G \setminus \{x\}$ 
are both contractible, using the inductive assumption that the $1$-point graph is contractible. 
Homotopy can be defined algebraically. Inductively, a ring element $f_G$ is contractible if there 
exists $x_k$ such that both ring elements $A(x_1,\dots\hat{x}_k,\dots,x_n)) = f(x)  - f(x_1,\dots,x_{k-1},0,x_{k+1},\dots,x_n)$ 
and $B(x) = f(x_1,\dots,x_{k-1},1,x_{k+1},\dots,x_n)$ are contractible. 
The Euler characteristic of $G$ is $\chi(G) = \sum_k (-1)^k v_k(G)$, where $v_k(G)$ is the number of $k$-dimensional
simplices $K_{k+1}$ in $G$. More generally, the Euler characteristic of a ring element $f$ is
$\chi(f) = -f(-1,-1,\dots,-1)$. The Euler characteristic of the triangle $G=x+y+z+xy+yz+zx+xyz$ for example is
$-(-3+3-1)=1$. The Euler characteristic of the chain $3x+5y+4xy$ is $-3-5+4=-4$. 
The boundary of a ring element $f$ is defined as 
$\delta f = \sum_i \partial_{x_i} f$, where $\partial_{x_i}$ are the usual partial derivatives and the result is
projected onto functions satisfying $f(0)=0$. 
Due to the orientation assumption, this gives rise to sign changes. For example,
$\delta (x y z) = y z - x z + x y$ if the edges basis $xy,yz,zx$ was chosen.
We have $\delta \delta f=0$.
The exterior derivative on $\Omega$ is defined as $dF(x) = F(\delta x)$. As $d^2=0$, we have cohomology groups
for any ring element. Let $b_k = {\rm dim}(H^k(G))$ be the Betti numbers. 
The Euler-Poincar\'e formula $\chi(G) = \sum_{k=0} (-1)^k b_k(G)$ follows from linear algebra.
The boundary $\delta f(x_1 x_2 \dots x_k) = \sum_l (-1)^l f(x_1, \dots,\hat{x_l}, \dots,x_k)$
is no more a graph in general. The star graph $g=x+y+z+w+wx+wy+wz$ for example has 
the boundary $\delta g = 3w-x-y-z$ which is a chain.
The space of differential forms $\Omega$ is a direct sum $\Omega = \oplus \Omega^k$, where $\Omega^k$ 
is generated by functions supported on $k$-simplices, polynomial monoid parts in the ring of degree $k+1$. 
Since $\Omega^k(G)$ are finite dimensional, the maps $d_k$ are represented by finite matrices. 
They are called the incidence matrices and were considered by Poincar\'e already for triangulations 
of manifolds.  
The cohomology groups $H^k(G)$ and more generally $H^k(f) = {\rm ker}(d_k)/{\rm im}(d_{k-1})$ 
are independent of the chosen signs, when defining the chain ring element $f_G$. 
The matrix $D=d+d^*$ is called the Dirac matrix of $f$. Its square
$L=D^2$ is the form-Laplacian. It decomposes into matrices $L_k: \Omega^k(f) \to \Omega^k(f)$. 
By Hodge theory, the dimension of the kernel of $L_k$ is the $k$'th Betti number $b_k(G)$. 
Given a ring element $f$ and a monoid part $x$ in $f$, its unit sphere $S(x)$ is the unit sphere of $x$ in the
graph defined by $f$. It consists of all monoids dividing $x$ or which are multiples of $x$, without $x$.
It is again a ring element. 
The dimension ${\rm dim}_f(x)$ is inductively defined as $1$ plus the dimension of the unit sphere.
The dimension of $f$ finally is the average of the dimensions of all monoids in $f$. 
The dimension of the triangle $xyz+xy+yz+zx+x+y+z$ for example is 
$1+[{\rm dim}(S(xyz))+{\rm dim}(S(xy))+{\rm dim}(S(yz))+{\rm dim}(S(zx))+{\rm dim}(S(x))+{\rm dim}(S(y))+{\rm dim}(S(z))]/7$.
We have for example $S(xyz) = xy+yz+zx+x+y+z$ which defines the graph $C_6$ of dimension $1$
and $S(xy) = xy+x+y$ which is a line graph $P_2$ of dimension $1$ etc. We see that the dimension of the triangle is $2$. 
Of course, in the graph case, the dimension can be better computed on a graph level. The point is that the 
dimension extends to a nonnegative functional on the entire ring in such a way that the average of the 
dimensions of unit sphere $S(x)$ of a monoids or generalized vertices of $f$ is the dimension of $f$ minus $1$.  \\
We also need an inner product $\langle f,g \rangle $ on $\Omega^k$ which is defined as
$\langle a_n x^n, b_n y^n \rangle$ $= \sum_n a_n b_n$. It obviously satisfies all properties of
an inner product and especially defines a length $|f| = \sqrt{ \langle f,f \rangle }$.
Of course, the incidence matrices $d,d^*$ are adjoint to each other with respect to this product. 
As graphs are special functions, we could use the inner product for example to define an angle
between two graphs $G,H$ on the same vertex set it is the $\arccos$ of the fraction 
${\rm Cov}[G,H]/(\sigma[G] \sigma[H])$, where ${\rm Cov}[X,Y]$ is the number of common 
simplices of $G,H$ and $\sigma[G]$ is the square root of the number of simplices in $G$. 
Lets summarize the main point: 

\begin{propo}
a) Every graph $G=(V,E)$ defines a ring element $f_G = \sum_{x} x$, a sum over all complete subgraphs of $G$. \\
b) Every element $f=\sum a_{k}x^k$ defines a graph
$G_f$, where $V$ are the monoid entries $a_k x^k$ in $f$, with $|k|>0$ and where two entries are connected 
if one divides the other. \\
c) If $f$ comes from a graph $G$, then $G_1=G_f$ is a graph for which the original simplices are the points and 
which has the same topological features than $G$ and which additionally has a natural digraph structure. 
\end{propo}

In other words, there is a functor $f \to G_f$ from the ring $R$ to the
category of directed graphs given by the division properties (even so we often forget about the
directions) and that there is a second functor $G \to f_G$ from the category of undirected graphs to 
the ring.  They are not inverses of each other but the cohomology agrees. 
For other functorial relations, see the recent paper \cite{HuangYau2014}. \\

{\bf Remarks:} \\
{\bf 1)} Every ring element $f$ also defines its own geometric object which has Euler characteristic, cohomology, dimension,
homotopy as well as curvature. \\
{\bf 2)} When forgetting about the anti-commutativity within the graph which is irrelevant for the graph product,
the ring could be replaced with a more general integral domain. It would allow to 
see the graph product $G \times H$ as an element in $R[y_1,\dots,y_m]$ where $R=Z[x_1,\dots,x_n]$. 
This possibility can be useful when studying fibre bundles as one can work in a ring of the fibre graph. \\
{\bf 3)} The fact that the product of two graphs is obtained by writing the graphs algebraically using different generators 
and producing from it again a graph:
$$ G \times H = G_{f_G \times f_H} \;  $$
is not unfamiliar to us. If we take the Cartesian product of two spaces, we use different variables for the different
directions. If we don't take new variables and take $f \cdot f$ in the ring, then this is in general a chain. 
For a triangle $G=x y + y z + z x + x + y + z + x y z$ for example, we get (using $x^2=y^2=z^2=0$)
the chain $G \cdot G = 2 x y + 2 y z + 2 z x+6 xyz$. The graph which belongs to this chain is the star graph $S_3$
as $6xyz$ is the central vertex and the others the outer points (as they divide the central point). \\
The imposed Pauli principle imposed by anticommutativity $xy=-yx$ 
is irrelevant for the Euler characteristic, both for chains as well as for the graphs 
$G_f$ derived from the ring elements $f$.  \\
{\bf 4)} The graph $G \times K_2$ can be seen as a self-cobordism of $G_1$ with itself as it is a graph of
one dimension more which has two copies of $G_1$ as boundary. It is in general true that any geometric 
graph is self cobordant to itself as we can sandwich two copies $G$ with a completed dual graph $\hat{G}$.
but for $G_1$ we don't have to work on a construction. It is given. \\
{\bf 5)} Denote by $\Pi_{kl}$ the projection onto the linear subspace generated by $k$ until $l$-dimensional
simplices. As pointed out before, we can recover $G$ from $f_G$ by building $h=\Pi_{01} f_G$ and then 
building $G_h$ which is $G$. With $h =\Pi_{23} (f_G f_H)$, then $G_h$ is the classical standard 
Cartesian product of $G$ and $H$. 

\section{Some geometry}

The results for the graph product mirror results in the continuum. First we look at some basic 
constructions which deal with the notion of homotopy sphere or simply $k$-sphere in graph theory.
The definition of a $k$-sphere in graph theory is recursive: 
a $k$-sphere is a $k$-dimensional geometric graph for which every unit 
sphere $S(x)$ is a $(k-1)$-sphere and such that after removing any of its vertices, we get
a graph which is contractible. A $k$-ball is a $k$-dimensional geometric contractible graph with boundary which 
has a $(k-1)$-sphere as its boundary. We call the interior of a ball the part of $B$ which is not in the
boundary. A suspension $SG$ of a graph $G$ is the join $G \star S_0$, a double pyramid construction: 
add two new points $x,y$ and connect the points to all the vertices of $G$. A pyramid construction 
itself is the join $G \star K_1$. From the Cartesian product, we have construct joins 
by just building a product and identifying some variables in the algebraic representation of the graph. 
Examples are given at the end. 

\begin{lemma}[Suspension] 
The join $G \star K_1$ of a $k$-sphere $G$ is a $(k+1)$-ball.
The suspension $G \star K_2$ of a $k$-sphere $G$ with a $0$-sphere $K_2$ is a $(k+1)$-sphere. 
\end{lemma}

\begin{proof}
a) By definition, the boundary of $S \star K_1$ is $S$ which is a sphere. Also, the graph $S \star K_1$ is 
contractible. \\
b) Removing the second point $y$ produces the ball $S \star K_1$ by a). 
\end{proof} 

{\bf Remark:} \\
{\bf 1)} More generally, as in the continuum, and shown below,
the join $S^k \star S^l$ of two spheres is a sphere $S^{k+l+1}$. \\
{\bf 2)} Also as in the continuum, the definition of the join needs the product as the join $A \star B$ is a quotient
of $A \times B \times I$. It generalizes that the 3-sphere can be written as $S^1 \star S^1$, which has an 
interpretation of gluing two solid tori along a torus. \\

{\bf Examples:}  \\
{\bf 1)} The wheel graph $W_n$ is a 2-ball. It is the join $C_n \star K_1$, where $C_n$
is the cyclic graph with $n$ vertices.  \\
{\bf 2)} The $k$-dimensional cross polytope is $S_2 \star S_2 \star \dots \star S_2$, where we have $(k+1)$
factors. The square $C_4$ is equal to $S_2 \star S_2$, the octahedron is $S_2 \star S_2 \star S_2$ etc.

\begin{lemma}[Glueing ball]
Assume $B_i$ are $k$-balls with boundaries $S_i$ 
and assume that $B_3=B_1 \cap B_2$ is a $(k-1)$ ball with boundary $S_3 = S_1 \cap S_2$.
Then $B= B_1 \cup B_2$ is a $k$-ball with boundary $S \subset S_1 \cup S_2$. 
\end{lemma}

\begin{proof} 
This is proven by induction with respect to $k$. There are four things to show: \\
a) $B$ is contractible.  \\
b) every unit sphere in the interior of $B$ is a $(k-1)$-sphere. \\
c) every unit sphere in $S$ is a $(k-2)$-sphere. \\
d) when removing a vertex from $S$, we get a contractible graph.  \\
For a) take a point $y$ in $B_3$. We can retract everything in $B_1,B_2,B_3$ to $y$. \\
For b) we only have to look at a vertex $x$ in $B_3$. The unit ball $B(x)$ decomposes
For c), we only have to look at a vertex $z$ in $S_3$ and see whether its unit sphere in $S$
is a homotopy sphere.  For d), we can retract a pointed part to $S_3$.
\end{proof}

The next statement is the discrete analogue of the classical statement that the boundary of 
the product of two balls $d (B_1 \times B_2)$ is a sphere and that it
can be written as as $dB_1 \times B_2 + B_1 \times dB_2$ which is the union of two 
solid tori glued at a torus. Also in the discrete, we can use the intuition from the continuum:

\begin{lemma}[Cylinder lemma]
If $B_i$ are $k_i$-balls with $(k_i-1)$-spheres $S_i=\delta B_i$ as boundary,
then $(B_1 \times S_2) \cup (B_2 \times S_1)$ is a $(k_1+k_2-1)$-sphere
provided $B_1 \times S_2 \cap B_2 \times S_1 \subset S_1 \times S_2$. 
\end{lemma}

\begin{proof}
Use induction with respect to dimension. 
The union is a graph of dimension $k_1 + k_2-1$. A unit sphere $S((x,y))$
is in $B_1(x) \times S_2(y)$ or $S_1(x) \times B_2(y)$ whose intersection is 
$S_1(x) \times S_2(y)$. 
\end{proof}

\begin{figure}[h]
\scalebox{0.30}{\includegraphics{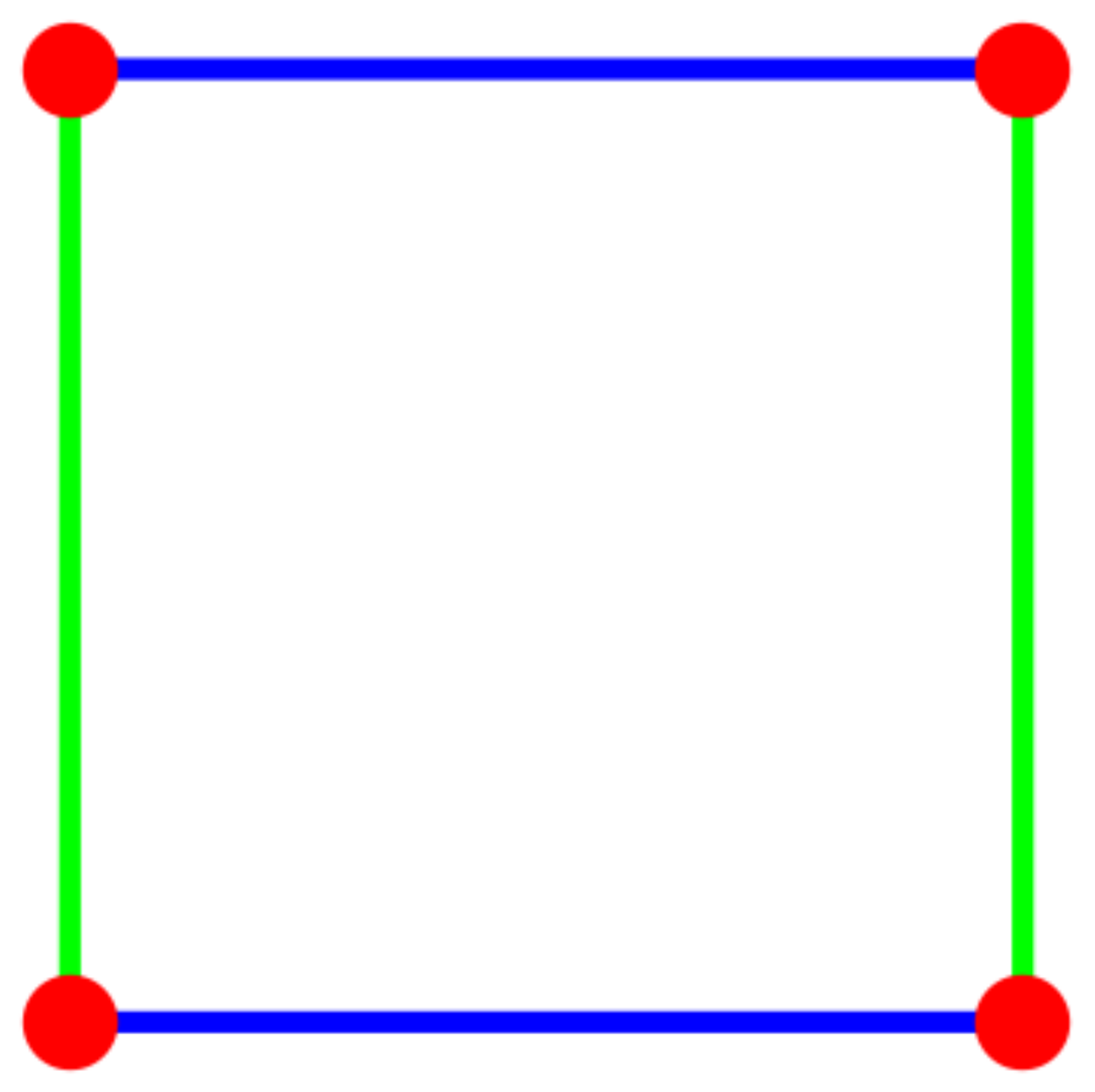}}
\scalebox{0.30}{\includegraphics{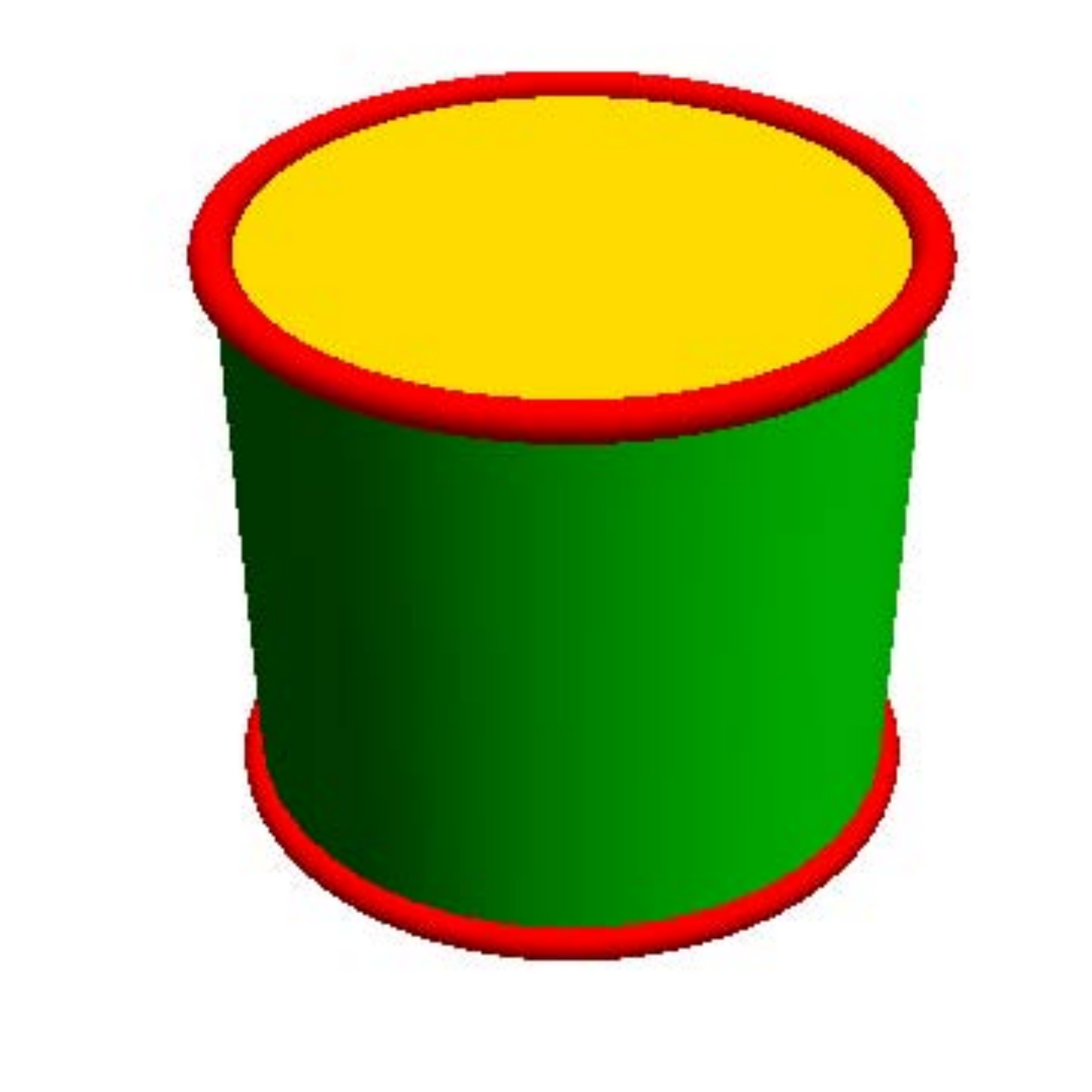}}
\caption{
The Cylinder Lemma tells that ``the union of complementary solid tori glued together at a torus
is a sphere." 
To the left, we see $B_1 \times S_0 \cup S_0 \times B_1$ which 
consist of the union of a bottom-top and left-right pair.
The intersection $B_1 \times S_0 \cap B_1 \times S_0$ is equal to $S_0 \times S_0$. 
The second example seen to the right shows $B_2 \times S_0  \cup S_1 \times B_1$ which is the union of
top-bottom discs and a mantle. The intersection $B_2 \times S_0  \cap  S_1 \times B_1$ agrees
with $S_0 \times S_1$ which is the union of two circles. We are interested
in these type of spheres because unit spheres of product graphs are of this 
``Hopf fibration" type. 
}
\end{figure}

{\bf Examples:}  \\
{\bf 1)} If $B_i$ are two line graphs with boundary $S_i$, then $B_1 \times S_2$ is
a union of two line graphs. Similarly $B_2 \times S_1$ is the union of two
line graphs. The union is a square, the intersection consists of the four
points $S_1 \times S_1$.  \\
{\bf 2)} If $B_1$ is a ball with $2$-dimensional sphere $S_1$ and
$B_2$ is a $1$-dimensional ball with $1$-dimensional sphere $S_2$, then 
$B_1 \times S_2$ is the mantle of the cylinder and $B_2 \times S_1$ 
are the top and bottom cover.  \\
{\bf 3)} If $B_1$ is $1$-dimensional and $B_0$ is $0$-dimensional, then $dB_1 \times B_0$
is a $0$-sphere. \\
{\bf 4)} If $B_1,B_2$ are two-dimensional balls, we can see 
$dB_1 \times B_2 + B_1 \times dB_2 + dB_1 \times dB_2$ is the union of two 
solid tori glued along a $2$-torus. This Hopf fibration is 
classically given as the split of $S^3 = \{ (z,w) \in C^2 \; | \; |z|^2 +|w|^2=1 \}$ 
into two solid tori $B_1 = \{ |z| \geq |w| \}$ 
and $B_2 = \{ |z| \leq |w| \}$ intersecting in the $2$-torus $\{ |z|=|w|=1 \}$. 

\section{De Rham cohomology for graphs} 

Among various other flavors of cohomologies, there are three equivalent cohomologies for 
compact $n$-manifolds: simplicial cohomology, de Rham cohomology and 
\v{C}ech cohomology.  For simplicial cohomology, 
the manifold is triangulated into finitely many $n$-simplices leading to a differential complex.
In de Rham cohomology, one works with the complex of differential forms, partial
derivatives tap into the local product structure of the manifold,
for \v{C}ech cohomology, the manifold is covered with a finite cover of open sets so that the
nerve graph determines the cohomology.
Each of the these cohomologies have advantages over the others: simplicial cohomology
is the computer science or combinatorial point of view which sees space as a mesh of small simplicial building blocks, 
the de Rham cohomology is the analysis or calculus approach, which taps into the bag of techniques used
in calculus. In this flavour, the basic building blocks are cubes obtained from a local product structure
and the different directions are accessed with partial derivatives which we think of the exterior derivatives 
in each factor. The \v{C}ech cohomology finally is the homotopy or topologal point of view which relies 
on the fact that cohomology is more robust and transcends dimension. Simplicial cohomology is 
straightforward and the simplest. De Rham cohomology requires local charts which are products
and taps into the differential structure of the manifold, allowing for an efficient computation of cohomology. 
\v{C}ech cohomology finally is a flexible variant which illustrates best the homotopy invariance of 
cohomology. It also can lead to considerable complexity reduction as on can for example retract a 
space to a much smaller dimensional set. A solid torus for example can be retracted to a circle. 
The equivalence of simplicial cohomology with de Rham cohomology is due to 
de Rham. Later proofs of this theorem use the equivalence of \v{C}ech cohomology with 
simplicial cohomology. \\

All three cohomologies have analogue constructions in graph theory. It is our goal to introduce 
the analogue of de Rham cohomology and give a discrete analogue of the de Rham theorem. Of course
de Rham cohomology emerges only in the discrete if one has a product which is compatible. \\

The simplicial cohomology of a graph $G$ is the clique cohomology of $G$ using the Whitney complex 
of all complete subgraphs. It is the oldest and has already been considered by Poincar\'e, even so 
not in the language of graphs.  \\

The \v{C}ech cohology is the cohomology of the nerve graph of a good cover,
where ``good" in the discrete means that the nerve is homotopic to $G$. \v{C}ech cohomology has
first been considered for graphs in  \cite{KnillTopology}, a paper which proposes a notion of what
a continuous map between graphs is, which is more tricky than one might think at first,
as classical topology badly fails as the topology generated by the distance is discrete 
making it unsuitable. What is important for a good notion of continuity is to merge homotopy with dimension.
The equivalence of discrete \v{C}ech cohomology with discrete simplicial cohomology is there by definition
just relying on the fact that homology is a homotopy invariant. \\

What was missing so far in graph theory is an analogue
of de Rham cohomology, where ``cubes" rather than ``simplices" play the fundamental role. But one can not
really look at a de Rham cohomology for general graphs, if one does not have a Cartesian product for which 
there is compatibility. 
A de Rham type cohomology has been mentioned in \cite{HuangYau2014}, where a generalized path cohomology 
introduced by \cite{GLMY2013} is considered for digraphs.
As their approach uses a functor from graphs to CW complexes and pulls back results from the product of 
CW complexes to digraphs, there is no relation with what we do here. Other takes on discrete de Rham cohomology 
study discrete notions for numerical purposes \cite{TaiWinther} or \cite{Boffi}.
Our approach to de Rham cohomology is purely combinatorial and restricted to finite constructions. \\

As in the continuum, the de Rham complex for a product graph $G \times H$ 
can use some derivatives also in the discrete, but this is merely language: 
while in simplicial cohomology, we write $\delta (x_1 x_2 \dots x_n )$
$= \sum_i (-1)^k x_1  \dots \hat{x}_i \dots x_n$ for the boundary of a simplex $x_1 x_2 \dots x_n$ 
in the de Rham approach, we can write $\delta x = \sum_i \partial_{x_i} x$ in the algebra 
which is the same thing. It just uses the derivative notion in a formal way. 
The de Rham connection will be needed in the K\"unneth connection, where we
look at the kernel of form Laplacians. K\"unneth will then be quite obvious. Without linking de Rham with 
simplicial cohomology the relation is nontrivial, as the dimension of the space of differential 
forms on the product graph is much larger than the product of the dimensions of the space of differential
forms on the factors. Already K\"unneth had to work though such difficulties and needed dozens of pages
of linear algebra reductions to tackle the issue. Our situation is also different in that we look at the
product of two arbitrary networks $G,H$ which by no means have to be geometric. \\

The analytic de Rham approach allows the derivative $d(f g)$ to be written 
as a sum of products $(df) g + (-1)^{|f|} f (dg)$ which is Leibniz formula and which reduces the exterior
derivative of the product to the exterior derivative of the factors in the same way than 
the gradient, curl or divergence reduces the exterior derivative to partial derivatives, which are the 
exterior derivatives in the $1$-dimensional factors. To illustrate this with school calculus: 
infinitesimally, the curl of $F=(P,Q)$ is $Q_x - P_y$. As a graph theorist we look 
at $P$ (for fixed $y$) as the $1$-form restricted to the first coordinate $x$ which means a function
on edges of the first graph and at $Q$ (for fixed $x$) as the $1$-form restricted to the second coordinate
which corresponds to edges in the second graph. The curl is so a line integral along 
a square. When reducing this to simplicial cohomology, the square needs to be 
broken up into triangles. Our product does that very explicitly even so it needs some care as the 
tensor product of finite dimensional algebras has a completely different dimension in general 
than the Cartesian product. K\"unnneth needed dozen of pages of rather messy linear algebra reductions 
to achieve this. The language of polynomials and the de Rham connection 
will allow us to make this more clear.  \\

When we start with a triangulated picture, there are no more two distinguished directions
present. In graph theory, we only can distinguish two directions if we look at the product $G \times H$ of two graphs.
As for manifolds, this structure could be allowed to be present locally only; its important however
that a product structure must be present
before we can even talk about discrete de Rham cohomology. As mentioned before, there is the possibility to see a 
graph as a triangularization of a manifold and use the Euclidean product structure to emulate a discrete de Rham 
cohomology. Notice however that this does not tickle down to the discretization. The relation would only 
exist functorially and is pretty useless when working with concrete networks.
We will not leave the discrete realm and show that $k$-forms on the product space can be related to 
products of forms in the two factors. And also, we do not only work with geometric graphs, 
which can be seen as discretizations of manifolds; we work with general finite simple graphs. \\

A finite simple graph $G$ naturally comes with a simplicial complex, given by the set
of all the complete subgraphs of $G$. This so called Whitney complex can be 
encoded algebraically in the ring of polynomials. For a triangular graph for example, we have the 
ring element $g = xyz+xy+yz+zx+x+y+z$, where the choice of the orientations of all the 
simplices is done arbitrarily. The ring element $g$ in turn defines a new graph $G_1$ 
in which the polynomial monoids form the vertices and two vertices are connected 
if one divides the other. It is important that it actually can be seen as a digraph, 
the direction is given which part is a factor of the other. In some sense, going from $G$ to $G_1$
frees us from having to chose an arbitrary orientation as the structure is now built in. By the way,
$G_1$ appears to have other nice features like having the Eulerian property allowing therefore a 
geodesic dynamical system \cite{knillgraphcoloring2}. \\

In our triangular graph example $G$, we get a graph $G_1$ with $7$ 
vertices because there are $7$ complete subgraphs of the triangle.
The graph $G_1$ is a wheel graph which shares the topological and cohomological 
properties with the triangle. It is even homeomorphic to the triangle in the sense of \cite{KnillTopology}
as we can find a $2$-dimensional open cover whose nerve is the triangle. Indeed:  \\
the \v{C}ech cohomology of $G_1$ is equivalent to the graph cohomology of $G$.  \\
The observation that $G$ and $G_1$ are homotopic proves that the \v{C}ech cohomology 
of a graph with respect to a good cover in the sense of \cite{KnillTopology}
is the same than the graph cohomology. As in the continuum, the relation between simplicial and
de Rham cohomology is not completely obvious as we will just see. \\

Now, if we take the product $G \times H$ of two graphs, like for example two complete graphs $G=K_2, H=K_2$
(illustrated in Figure~(\ref{figurederham})),
where the ring elements $g=(x+y+xy)$ and $h=u+v+uv$ encode the graph, then the product
$G \times H$ is encoded by a ring element $gh=xu+xv+xuv+yu+yv+yuv+xyu+xyv+xyuv$ which by looking at
division properties of the polynomial monoids produces the graph $G \times H$ with 
$9$ elements. It is the wheel graph $J=W_{10}$, which is a discrete square
with $16$ edges and $8$ triangles. The Laplacian $L=(d+d^*)^2$ 
with the usual exterior derivative is a block matrix decomposing into 
a $9 \times 9$ block $L_0$, a $16 \times 16$ block $L_1$ and a $8 \times 8$ block $L_2$. 
The algebraic representation of $f(J)$ is a polynomial element with $9+16+8=33$
monoid terms. The graph $J_1$ associated with $f(J)$ would have $33$ vertices already. 
From topological, algebraic or homotopical considerations, the graphs $J$ and $J_1$
are equivalent: they are homeomorphic, they are homotopic and have the same cohomology
and dimension. \\

A boundary on the product is defined by the Leibniz formula:
$$  \delta (f g) = (\delta f) g + (-1)^{|f|} f (\delta g) \;  $$
which then defines an exterior derivative $d$. 
As for \v{C}ech cohomology, we will see that there is an advantage in that we can work 
with a smaller dimensional vector spaces: the spaces $\Omega^k(G)$ and $\Omega^k(H)$
and even the tensor product $\Omega^k(G) \otimes \Omega^k(H)$ have in general 
smaller dimension than the vector spaces $\Omega^k(G \times H)$. See Figure~(\ref{poster}).
Having equivalence of cohomology can be a blessing when doing computations. \\

This will be useful when looking at discrete manifolds: 
De Rham cohomology can be used more generally also
when gluing product graphs to build fibre bundle, this product structure
does not have to be global. It is the same situation as for 
manifolds have in general only locally neighborhoods which can be written 
as products. \\

\begin{figure}[h]
\scalebox{0.32}{\includegraphics{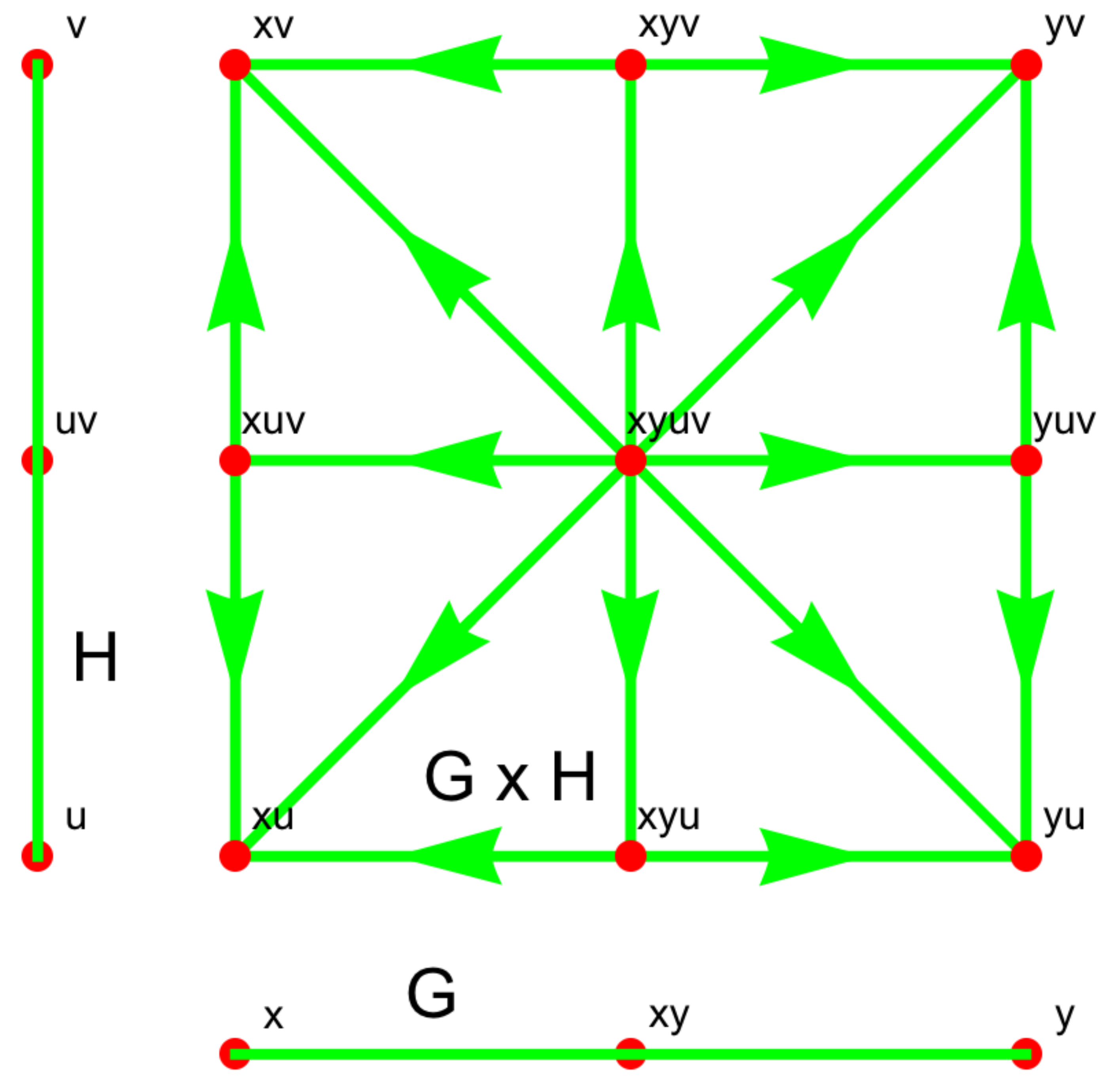}}
\caption{
\label{figurederham}
The product of $G=K_2=x+y+xy$ with $H=K_2=u+v+uv$ is the wheel graph 
$G \times H = xu+xv+xuv+yu+yv+yuv+xyu+xyv+xyuv$. It has a natural digraph
structure in which the direction is given by which part divides which. 
A form $f$ in $\Omega^*(G)$ is a function $f=a x + b y + k xy$, where $(a,b)$
represents the scalar function, the $0$-form and $k$ is the value of the 1-form.
Similarly $g \in \Omega^*(H)$ is a function $g=d u  + e v + l uv$. 
The product $fg \in \Omega^*(G) \otimes \Omega^*(H)$ 
gives a scalar function on the $9$ vertices of $G \times H$.  
It defines an element $F=\psi(fg)$ in $\Omega^*(G \times H)$ as the digraph structure 
defines an orientation on the simplices $\sigma$ and $F(\sigma)$ is the sum of the values of
$fg$ on the vertices. The reverse chain map $\phi$ is obtained by first defining from 
$F$ a function $h$ on the vertices of $G \times H$, then finding the functions $f$ and $g$
by solving a system of equations. Now $\psi(F)=fg$ is an element in $\Omega^*(G) \otimes \Omega^*(H)$. 
We see that differential forms and geometric objects are 
treated similarly, as is typical in quantum calculus setups
and that the Eilenberg-Zilber becomes combinatorial and explicit. 
}
\end{figure}

\begin{thm}[Discrete De-Rham Theorem]
The de Rham cohomology for a product graph $G \times H$ with boundary operation defined by the 
Leibniz rule $d(fg) = (df)g + (-1)^{|f|} f(dg)$ using the exterior derivatives $d$ on each factor
$G$ and $H$, is isomorphic to the cohomology on $G \times H$ defined by the exterior 
derivative $d$ on the graph $G \times H$. In short,
$$  H_{dR}^*(G \times H) \equiv H^*(G \times H)  \; . $$
\end{thm}

After de Rham \cite{DeRham1931}, new proofs of the de Rham theorem in the continuum were
given by A. Weil (1952) \cite{Weil1952} and H. Whitney \cite{Whitney1957}. 
Weil had outlined a proof already 
in 1947 in a letter to H. Cartan, a letter which initiated Cartan's theory of sheaves. 
Weil's argument used a staircase argument in a double complex.
The textbook proof of Bredon
\cite{Perez2010} is worked out in more detail in \cite{Perez2010}. See also \cite{Prasolov,Bott,Hatcher}. 
In the discrete, we can directly construct the 
chain maps $\phi$ and $\psi$. As we will see below, our goal will be
to find a one-to-one correspondence between harmonic forms on the
large simplex Laplacian $L=(d+d^*)^2$, which is a $v(G \times H) \times v(G \times H)$ matrix
and the de Rham Laplacian $L = (d_{dR}+d_{dR}^*)^2$,
which is a $(v(G) \cdot v(H)) \times (v(G) \cdot v(H))$ matrix.  This is done by a rather
explicit chain homotopy. 

\begin{proof}
Assume $G,H$ have vertices $x_1,\dots,x_n$ and  $y_1,\dots,y_m$.
The de Rham cohomology on the ring generated by functions $fg$ 
is defined by using the exterior derivatives on each 
product using the Leibniz rule. We have to relate it with the Whitney chain complex 
$\Omega^*(G \times H)$ which features functions on the simplices of $G \times H$.
The linear spaces $\Omega^k_{dR}(G \times H)$ and 
$\Omega^k(G \times H)$ have different dimensions $v_k(G \times H)$ and $\sum_{m+l=k} v_m(G) v_l(H)$
so that $\phi$ is not the inverse of $\psi$ (they are chain homotopic only, which essentially means that
they are equivalent modulo coboundaries on each side).
In both cases, denote by $d: \Omega_{dR}^{n} \to \Omega^{n+1}_{dR}$
and by $d: \Omega^n \to \Omega^{n+1}$ the exterior derivatives.
We will construct two linear maps $\phi,\psi$ and check that they are chain maps:
$d \psi = \psi d$ and $d \phi = \phi d$. 
This will establish the isomorphism of cohomology
$H^*(G \times H)$ to $H^*_{dR}(G \times H)$. \\

{\bf Construction of $\psi:\Omega_{dR}^k \to \Omega^k$:} \\
Start with an $k$-form $F \in \Omega^k_{dR}$. As it can be written as $f(x) g(y)$
in the variables $x_1,\dots,x_n,y_1,\dots,y_m$, it defines a function 
on the vertex set of $G \times H$. Given a $k$-simplex $\sigma$ in the graph 
$G \times H$.  Assigned to it the value of the vertex which belongs to a monomial
of degree $k$. If there is one, it is unique. If there is none, assign the value $0$. \\

{\bf Construction of $\phi:\Omega^k \to \Omega^k_{dR}$:} \\
For every $k$-form $F$ in $\Omega^k(G \times H)$ we 
build a polynomial $\phi(F)=f(x) g(y)$ as follows:
the function $F$ defines a value $\overline{F}$ to the vertices of $G \times H$
by averaging the $F$ values of the $k$-simplices hitting $x$. We have
now a function $\overline{F}$ on the vertices of $G \times H$, written
as a polynomial $\overline{G}(x,y)$ in the variables $x_1,\dots,x_n,y_1,\dots,y_m$. 
One can now add a polynomial $P$ of smaller degree so that $\overline{G}(x,y) + P(x,y)$
can be factored as $f(x) g(y)$.  \\

{\bf $\psi$ is a chain map: $\psi d = d \psi$}: \\
Proof: Given a $(k-1)$-de Rham form $F=f(x)g(y)$, build the $k$-form 
$dF$ using the Leibniz rule, then use $\psi$ to get from this 
a $k$-form $\psi dF$ on $G \times H$. By linearity, we can assume
that $F = A x^i y^j$ with $|i|+|j|=k-1$. The function $dF(x,y)$ 
takes the value $\pm A$ on each polynomial $x^{i'} y^{j'}$ with $|i'|+|j'|=k$. 
Now, $\psi dF$ is a $k$-form on $G \times H$ which takes the value $\pm A$
on those vertices $x^{i'} y^{j'}$. But $\psi F$ is a $k-1$ form which assigs
the value $A$ to every simplex having $x^i y^j$ as a vertex. Now
$d\psi(F)(x,y)$ is exactly $A$ if $x y$ is of the form $x^{i'} y^{j'}$. \\

{\bf $\phi$ is a chain map: $\phi d = d \phi$}: \\
Proof: given a $k-1$-form $F$ on the product graph $G \times K$, we build 
the $(k+1)$-form $dF$ and produce from this a polynomial $\psi(dF)$ in the de 
Rham complex. We check that this is the same than $d \psi(F)$. 
by linearity we can look at a $k-1$-form $F$ which assigns 
the value $A$ to a $(k-1)$ simplex $\sigma$ of $G \times H$. Now take a
$k$-simplex which has $\sigma$ in its boundary, then $dF(\sigma) = A$. 
But now, $\phi(dF)$ assigns the value $A$ to the vertex $z$ which is in 
$\sigma$ and has degree $k$. Now, $\phi(F)$ assigns the value $A$ to 
the vertex $z'$ which is in $\sigma$ and has degree $k-1$. This is 
connected to $z$ so that $d \phi(F)$ is the same value.
\end{proof}

It would be nice to have a more intuitive understanding of the function
$\theta: \Omega^k_{dR}(G \times H) \to \Omega^{k-1}(G \times H)$ satisfying
the chain homotopy condition
$$ \phi \psi - \psi \phi = d \theta +\theta d  \; . $$
The map $\theta$ is related to 
the lower degree polynomial $P$ added to $F$ so that $F(x,y) + P(x,y)=f(x) g(y)$. 
This involves solving a system of linear equations. The dimension of the space of 
scalar functions on $G \times H$ is $v(G) \times v(H)$, where $v(G)$ is the number of 
simplices in $G$. The space of functions $f(x) g(y)$ has dimension $v(G) + v(H)$. 
The image of $\theta$ consists of all lower degree polynomials (as usual satisfying $f(0,0)=0$)
which is a space of dimension $n=(v(G)-1) v(H) + v(G) (v(H)-1)$. To solve for $F(x,y)=f(x) g(y)$
we have to solve a system of equations for $n$ variables and this solution gives us $P$.
$$
 \xymatrix{\Omega_{dR}^{k-1} \ar[r]^d  \ar[d]^\psi       & \Omega_{dR}^k \ar[d]^\psi \ar[ld]_\theta \\
           \Omega^{k-1}      \ar[r]^d  \ar@<1ex>[u]^\phi & \Omega^k      \ar@<1ex>[u]^\phi           }
$$

\begin{figure}[h]
\scalebox{0.32}{\includegraphics{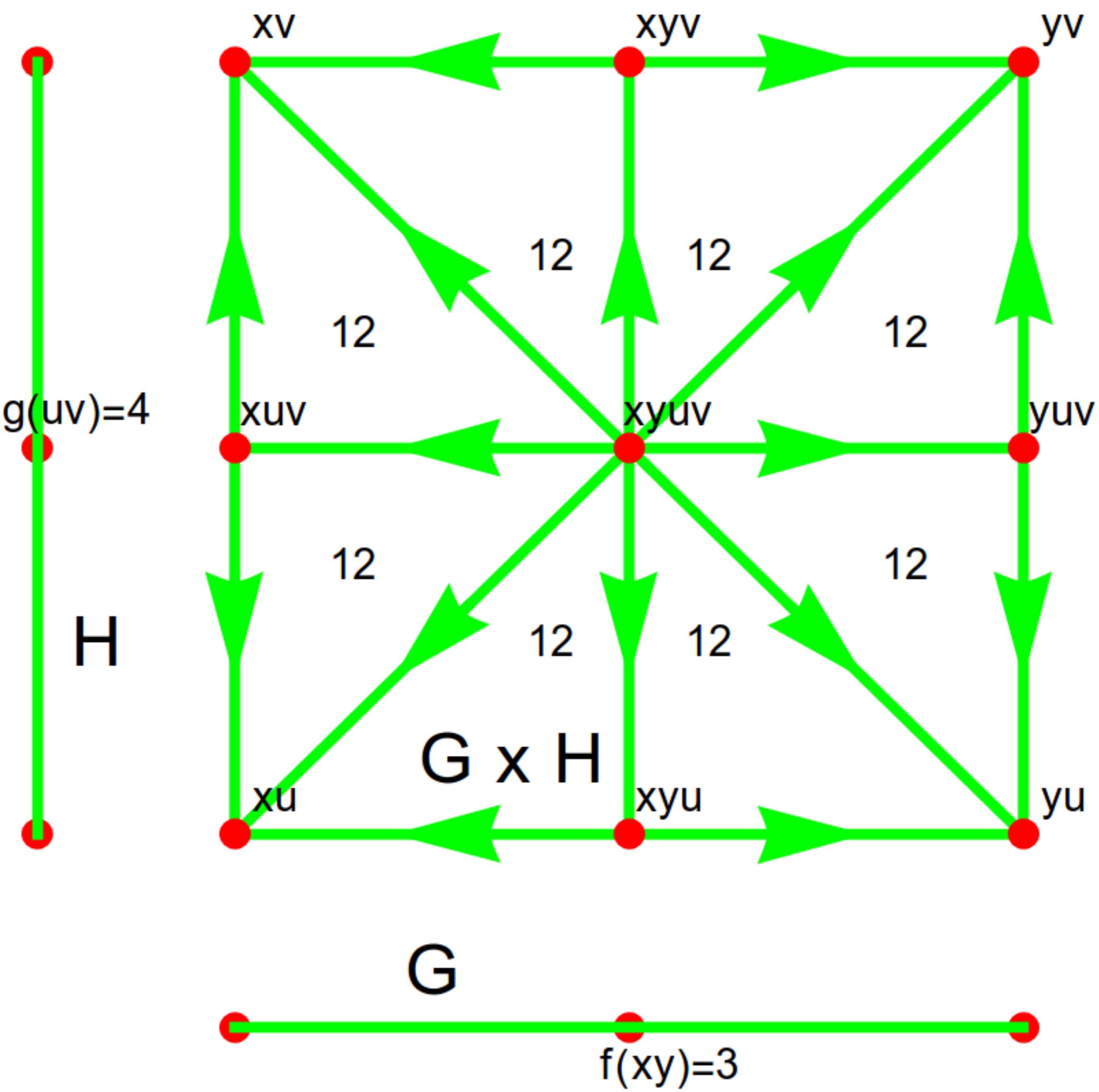}}
\caption{
\label{derhamexample1}
An example illustrating the chain homotopy in the case $G=K_2,H=K_2$:
a $2$-form in $\Omega^2_{dR}(G \times H)$ is given as a function
$f(xy) g(uv)$ as there are no $2$-forms in $G,H$. The chain map $\psi$
constructs from this a 2-form in $\Omega^2(G \times H)$, that is a 
function on the triangles of $G \times H$, as follows: 
the function $F(x,y,u,v) = f(x,y) g(u,v) = 12 x y u v$ is a function
on the vertices of $G \times H$ which is $12$ on the center point and
zero else. The $2$-form $\psi(F)$ assigns now the value
$12$ to all triangles having $xyuv$ as a vertex. The backwards chain map $\phi$
assigns to the node $xyuv$ the average of the values of the $2$-form on
adjacent triangles. There is no need here to add lower degree polynomial $d \theta$ 
as this already factors $f(x) g(y)$.
}
\end{figure}

\begin{figure}[h]
\scalebox{0.32}{\includegraphics{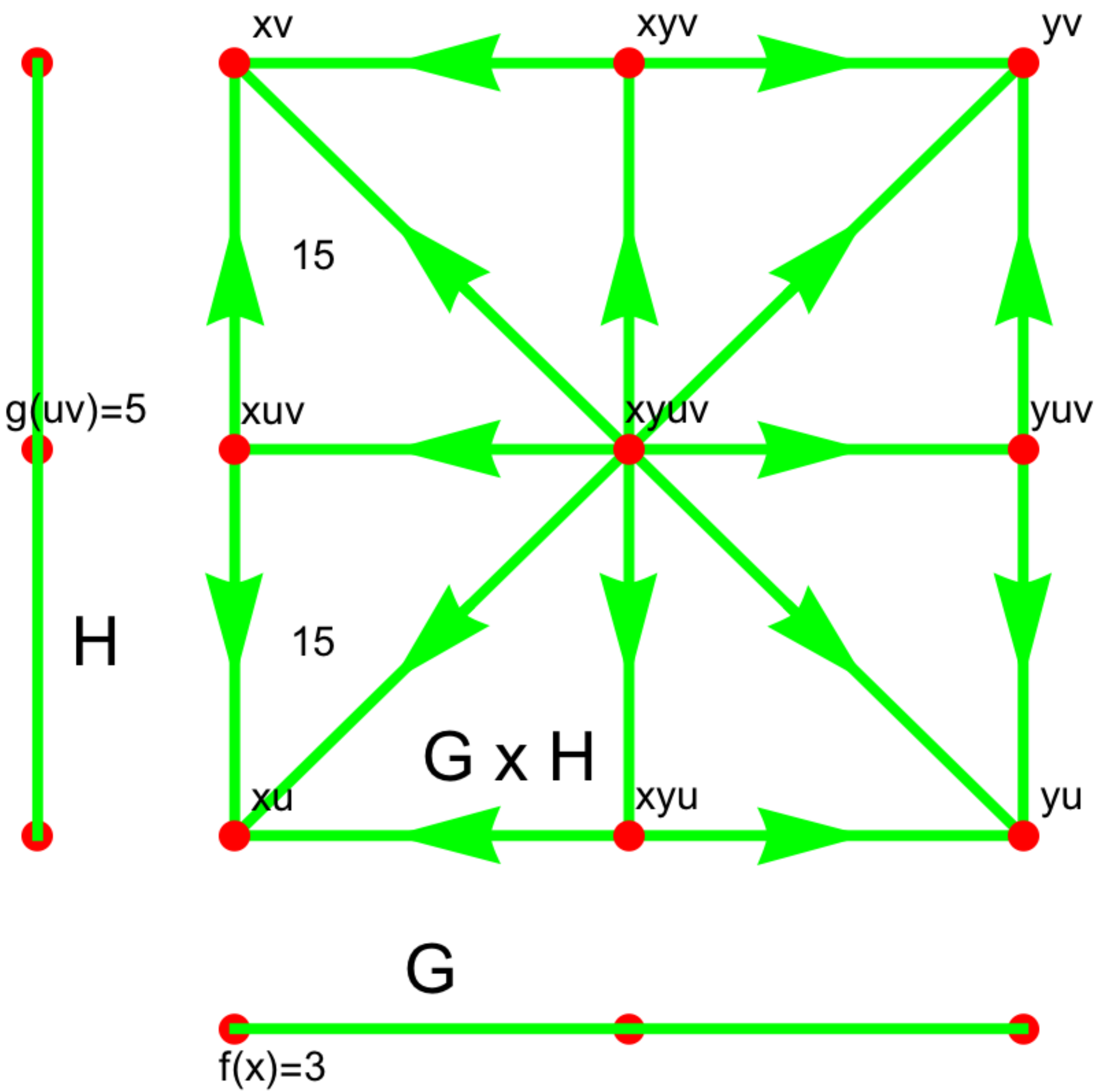}}
\caption{
\label{derhamexample2}
An example illustrating the chain maps in the case $G=K_2,H=K_2$:
a $1$-form $F$ in $\Omega^1_{dR}(G \times H) = \Omega^0(G) \otimes \Omega^1(H)
\oplus \Omega^1(G) \otimes \Omega^0(H)$ is represented algebraically as
$F=c xy (a u + b v) + C uv (A x + B y) = (5uv) (3x) = 15 xuv$. 
By looking at the coefficients, we get a function on the vertices of $G \times H$.
We have now assigned 
values to the four ``$1$-dimensional vertices" represented by 
monomials of degree $3$, in this example to $xuv$ only. The $1$-form $\psi F$ 
on $G \times H$ assigns values to all edges emanating from $xuv$ and assigns there the
value $15$ attached to the node $yuv$. In this case we already have a factorization $f=3x,g=5uv$.
In general, we have to add a gradient to $d\theta$ in such a way that the chain homotopy 
$\phi \psi - \psi \phi = d \theta + \theta d$ holds. 
}
\end{figure}

{\bf Examples:} \\
{\bf 1)} If $G$ is a forest with $n$ trees and $H$ is a forest with $m$ trees,
then the product graph $G \times H$ has $n*m$ components. The Betti vectors
are $b(G)=(n,0,\dots),b(H)=(m,0,\dots)$ and $b(G \times H)=(nm,0,0,\dots)$. Every
harmonic form on the product space $G \times H$ is of the form $f(x) g(y)$, where
$f,g$ are both locally constant. Because the number of vertices in $G \times H$
and $0$-simplices in $G \times H$ are the same, we don't have to translate between
simplicial and de Rham $0$-forms. \\
{\bf 2)} Let $G=H=K_2$, then $G \times H$ is a $2$-dimensional ball with $9$ vertices
$\{ xu,xv,yu,yv,xyu,xyv,xyuv,xuv,yuv \}$. 
Given a $1$-form $F$ in the product graph $G \times H$. It is a function on the 
edges of $G \times H$. Given a vertex which is the product of a vertex and edge
like $xuv$, we assign the sum over all edge values hitting the vertex. \\

{\bf Remarks:} \\
{\bf 1)} In the complement of the set $R$ of chains which are
a product of graphs, the equivalence does hold any more. 
Take the chain $f= 7 xy + 3 x + 5 y$.
Its cohomology is different from the corresponding graph $G_f$ which is 
the $3$ vertex graph without edges.  \\
{\bf 2)} The graph $G_f$ of an abstract chain $f$ 
has as many vertices as there are simplices in $f$. 
The simplicial complex of $G_f$ therefore is already large. 
But the linear space of all $k$-forms on $G \times H$
is huge: if $v(G)$ is the number of simplices of $G$ and $v(H)$ the number of simplices
in $H$, then ${\rm Binomial}(v(G) \cdot v(H),k)$ is the dimension 
of the set of $k$-forms on $G \times H$. 
The de Rham complex is more manageable on a computer.

\section{Results}

\begin{thm}[Geometry of product]
If $G,H$ are geometric graphs of dimension $k,l$ then $G \times H$ is a 
geometric graph of dimension $k+l$. 
\end{thm}

\begin{proof}
We have only to show that $S((x,y))$ is a homotopy sphere. 
Depending on $(x,y)$ the graph is the union of two cylinders 
intersecting in a lower dimensional graph which is part of 
the sphere $S(x) \times S(y)$. Use the cylinder lemma. 
\end{proof}

\begin{coro}
If $G,H$ are geometric spheres of dimension $k,m$ then the join $G \star H$ is a geometric sphere
of dimension $n=k+m+1$.
\end{coro}
\begin{proof}
It is a graph of dimension $n+k+1$ as it is a quotient of a graph $G \times H \times L_n$,
where $L_l$ is a line graph with $l \geq 2$. This assures that for all points which are not affected by 
the identification, we have a geometric unit sphere of dimension $n-1$. For the vertices
$y$ which were subject of identification, the unit sphere $S(y)$ is a union of two $(n-1)$-dimensional 
cylinders which by the cylinder lemma is a sphere.
Besides verifying that every vertex $y$ has a unit sphere $S(y)$
which is a $(n-1)$-sphere, we have also to check that taking away a vertex $x$ from $G \star H$,
leads to a remaining graph which is a contractible ball of dimension $n$ which has a $(n-1)$-dimensional
boundary. 
\end{proof} 

Take $G_1=G \times K_1$. Given a vertex $x \in G_1$. It corresponds to a simplex in $G$. 
The set $S(x)$ consists of all simplices in $G$ containing $x$ united with
the set of all simplices inside $x$. \\

{\bf Examples:}  \\
{\bf 1)} Let $G$ be a cyclic graph $C_n$. Lets look at the sphere of the vertex $x$
which is now also a vertex in $G_1$. It becomes now the set of two vertices $\{ yx,xz \}$ in $G_1$
corresponding to two edges in $G$. This is a $0$ dimensional sphere. \\
{\bf 2)} Let $G=C_n$ again and let $e=xy$ be an edge in $G$ which becomes now a simplex in $G_1$. 
Its sphere consists of the two simplices $\{x,y\}$ which correspond to two vertices in $G$. 
The sphere is $S_0$, again a geometric graph. \\
{\bf 3)} In the case $d=2$ like if $G$ is an ecosahedron,  assume that 
$S(x)=(y_1,y_2,\dots ,y_k)$ is the old unit sphere of a vertex $x$.
The new sphere of $x$ in $G_1$ contains the points $xy_k$ which correspond to edges
in the old graph $G$. The triangles $x y_k y_{k+1}$ are the new connections in that
sphere. We have again a geometric graph. 
{\bf 4)} Now assume that $e=xy$ is an edge in the old two dimensional graph $G$ which becomes
a vertex in $G_1$. The sphere $S(e)$ consists of the vertices $x,y$ which are the 
third points in the triangles $xyz_k$ as well as the two triangles $xyz_k$. 
These two original vertices and triangles form  4 points in the new graph $G_1$ which form a
circular graph $C_4$. This is our unit sphere. \\
{\bf 5)} Now assume that $t=xyz$ is a triangle in the two dimensional graph $G$. 
It becomes a vertex in the new graph $G_1$. The unit sphere $S(t)$ consists
of all simplices $x,y,z,xy,yz,xz$ which is a $6$-gon. 
Examples 3)-5) together show that that $G_1$ is always geometric if $G$ is two dimensional. 
The unit spheres always are $C_4$ and $C_6$. In particular, $G_1$ is Eulerian. \\
{\bf 6}) In the case $d=3$, if $x$ is an original vertex, then the sphere consists
of all edges and triangles in the original sphere by induction this is
$S(x) \times K_1$. If $e=(x,y)$ is an edge. Then the sphere consist of the points
$x,y$ as well as all triangles and tetrahedra containing $e$. This is a double
suspension of a $1$-dimensional sphere and so $2$-dimensional. If $t=(a,b,c)$
is a triangle, then the unit sphere is a suspension of the $6$-gon $a,b,c,ab,ac,bc$. \\

A theorem from 1964 assures that only spheres are manifolds which are joins
\cite{KwunRaymond}. Since geometric graphs naturally define compact manifolds,
where unit balls are filled up to become the charts, this 
holds also for geometric graphs. Intuitively, the reason is clear. In order
that the unit spheres of the identified end points are spheres, we better need 
the factors $A,B$ of the join $A \star B$ to be spheres, so that $G$
is of the form $A \star B$ where $A,B$ are spheres. \\

\begin{coro}[Kwun-Raymond]
If $G$ is a $n$-dimensional geometric graph which is the join $A \star B$
for two other graphs, then $G$ is a homotopy sphere. 
\end{coro} 

Now, we come to the main result as announced in the title: 

\begin{thm}[K\"unneth]
Given two finite simple graphs $G,H$.
The cohomology groups of $G \times H$ are related to the cohomology groups
of $G$ and $H$ by 
$$  H^k(G \times H) = \sum_{i+j=k} H^i(G) \otimes H^j(H) \; . $$
\end{thm}

\begin{proof}
We know from Hodge theory that the de Rham Laplacian $L=(d+d^*)^2 = d d^* + d^* d$ 
of $G$ restricted to $k$-forms has a kernel ${\rm ker}(L_k)$ of dimension is $H^k(G)$. \\ 
We have $d(f g) = (df)g + (-1)^{|f|} f(dg)$ and $d^*d(f g) = (d^*df) g + f (d^* dg)
+ (df) (d^*g) + (d^*f) (d g)$, so that 
$$ L(f g) = (Lf) g + f (Lg) + 2 (df) (d^* g) + 2 (d^* f) (dg) \; . $$
Hodge theory also assures that that $L f = 0$ is equivalent to $df=d^* f = 0$, so that if $f,g$
are harmonic forms in $G$ or $H$ respectively, then $fg$ is a harmonic form in $G \times H$. 
We see that we can construct the cohomology classes of the product $H^i(G) \otimes H^j(H)$ 
from the cohomology classes of the factors $G,H$. It is obvious therefore that
$H_{dR}^k(G \times H)$ contains the vector space $\sum_{i+j=k} H^i(G) \otimes H^j(H)$. \\
Now assume that $fg$ is harmonic in $G \times H$. We want to show that
$Lf=0$ or $Lg=0$. 
Again we use that the kernel of $L$ is the intersection of the kernel of 
$d$ and the kernel of $d^*$. Now look at the equation 
$$ L(fg) = (Lf) g + f (Lg) + 2 (df) (d^* g) + 2 (d^* f) (dg)   = 0 \; ,$$
where $f$ is a $i$-form and $g$ is a $j$-form. When we look at the terms, then
$f,(Lf)$ are polynomials in $x_k$ of degree $i+1$ and $Lg,g$ are polynomials in $y_k$ of 
degree $j+1$. But $(df) (d^*g)$ consists of polynomials of degree $i+2$ in $x_k$ and of degree $j$ in $y_k$
and $(d^*f g)$ consists of polynomials of degree $i+2$ in the $x_k$ and of degree $j+1$ in $y_k$. Because
the degrees are different, we see that $2 (df) (d^* g) = 2 (d^* f) (dg)=0$. 
Now, either $df=d^*f=0$ which implies $Lf=0$. But then, 
$f (Lg)=0$ so that also $(Lg)=0$. An other possibility is $dg=d^*g=0$ which implies $Lg=0$ and in turn 
implies $Lf=0$. A third possibility is $df=0=dg=0$. But from $Lf g + f Lg = 0$ 
follows then $(d d^* f) g + f d d^* g =0$. 
Now take the inner product with $fg$ to get $\langle f,dd^*f \rangle \langle g,g \rangle  + 
\langle f,f \rangle \langle g,dd^* g \rangle =0$ which implies that $f=g=0$, demonstrating  that the
third case is not possible. We have seen that $Lf=0$ and $Lg=0$ and verified that every kernel element
in $L(fg)$ necessarily has to have the property that both $Lf=0$ and $Lg=0$.  \\
Having established that the de Rham cohomology of $G \times H$ satisfies the K\"unneth formula
and that the de Rham theorem assures that it is equivalent to graph cohomology, we are done.
\end{proof}

{\bf Remark:} In the continuum, the equivalence of the chain complex $\Omega^*(X \times Y)$ 
with the tensor product $\Omega^*(X) \otimes \Omega^*(Y)$ on which the exterior derivative
is given by the Leibniz formula is subject of the
Eilenberg-Zilber theorem from 1953 \cite{EilenbergZilber}. They use the nerve of 
product coverings and so \v{C}ech type ideas to verify
the equivalence of the cohomologies and establish so the K\"unneth formula more elegantly. 
The above computation just reflects the fact that the tensor product of two chain complexes with exterior
derivative given by the Leibniz formula is again a chain complex and that we have an 
Eilenberg-Zilber formulation as in the continuum,
where two chain complexes $A,B$ are chain homotopic, if there are chain maps $\phi:A \to B,
\psi B \to A$ such that the composition $\psi \phi$ is the identity 
and $\phi \psi$ is chain homotopic to the identity in the sense that there is 
a linear map $\theta: \Omega^n(G \times H) \to \Omega^{n-1}_{dR}(G \times H)$ satisfying
$\phi \psi - \psi \phi = d \theta +\theta d$. 

\begin{coro}[Eilenberg-Zilber]
\label{eilenbergzilber}
Given two finite simple graphs $G,H$. 
There is a chain map $\psi$ from  $\Omega^*_{dR}(G \times H) = \Omega^*(G) \otimes \Omega^*(H)$
to the chain complex $\Omega^*(G \times H)$ and a chain map $\phi$ in the reverse direction
such that $\phi$ and $\psi$ are chain homotopic. 
\end{coro}

{\bf Examples:}  \\
{\bf 1)} Let $G$ be the house graph and $H$ the sun graph $S_{1,0,0,0}$. Both have the 
Betti vector $(1,1)$ and the product has the cohomology of the $2$-torus as it is homotopic
to the 2-torus. The graph $G_1$ has dimension $37/24$, has $v_0(G) =12$ vertices, $v_1(G)=18$
edges and $v_2(G)=6$ triangles. 
the graph $H_1$ has dimension $1$ and $v_0(H)=10$ vertices and $v_1(H)=10$ edges. The product 
$G \times H$ is a graph of dimension 
$61/24$. It has 120 vertices and 480 edges. The kernel of the Laplacian $L_1$ on $1$-forms of $G_1$
is spanned by $(8,-8,8,8,-3,-2,-3,0,-3,-1,-1,-3,-1,0,-8,2,-8,1)$, \\
the null space of $L_1$ on $1$-forms of $H_1$ is spanned by 
$(-1,1,-1,-1,1$ $,1,1,1,0,0)$. The Laplacian $L_1$ on the product graph $G \times H$ is a
$480 \times 480$ matrix. The Laplacian $L_1$ on the simplicial set $\Omega^1(G) \otimes \Omega^0(H)
\oplus \Omega^0(G) \otimes \Omega^1(H)$ is much smaller as it works
on a space whose dimension is the set of degree 3 monoids in $fg$ which is $v_0(G)*v_1(H) + v_1(G)*v_0(H)$.  \\
{\bf 2)} The cohomology of $C_4$ has the Betti vector $(1,1)$. The
cohomology of $C_4 \times C_4$ is $(1,2,1)$. \\
{\bf 3)} The cohomology $G \times K_n$ is the same as the cohomology of $G$. Similarly, the
cohomology of the cylinder $T \times C$ obtained by taking the products of a couple of
circles with the product of a couple of interval graphs is the same than of the torus $T$
because the cube, the product of interval graphs is contractible. \\
{\bf 4)} The Betti vector of $G=C_{n_1} \times \cdots \times C_{n_k}$ with $n_i \geq 4$ is the k'th row
of the Pascal triangle $(B(k,0),B(k,1),B(k,2), \dots,B(k,k))$. Easier to state, $p_G(x)=(1+x)^k$. \\

The reason for the following inequality is that higher dimensional simplices spawn more
vertices in the graph $G_1$. If the dimension is not 
uniform, then the dimension of the enhanced graph $G_1=G \times K_1$ will increase.
First of all, the construction of $G_1$ allows to define a local dimension for all
simplices $x$ inside $G$: it is just the usual dimension $1+{\rm dim}(S(x))$ 
in the enhanced graph $G_1$, where the simplices like $x$ are vertices. 

\begin{lemma}[Dimension inequality]
Given a finite simple graph $G$. Then
${\rm dim}(G \times K_1) \geq {\rm dim}(G)$.
\end{lemma}

For $K_2$, or any two point graph, we have equality. For all three point graphs different from
the three point graph $xy+x+y+z$, we have equality too, while for $xy+x+y+z$ we have
a first inequality ${\rm dim}(G)=2/3$ and $\delta(G)={\rm dim}(G_1)-{\rm dim}(G)=1/3$. 
Denote by ${\rm dim}(S(x))$ the dimension of the sphere in $G_1$
and by ${\rm dim}(x)=(1+{\rm dim}(S(x)))$ the dimension of the vertex $x$ in $G_1$. 
The average of all these numbers is ${\rm dim}(G_1)$. We can not use induction as the
unit spheres $S(x)$ are not of the form $H \times K_1$ with a smaller dimensional graph
in general. However, we can do it with a chain. We therefore prove a stronger statement
for chains $f = \sum_i x^i$: let the degree $1$ monomials be the vertices and the 
sphere $S(x)$ of a vertex $x$ consist of the sum $g$ of degree $2$ monomials in $f$ which contain
$x$ divided by $x$. The small dimension $\underline{{\rm dim}}(f)$ of a chain $f$
is now recursively defined as the average of the small dimensions of the unit spheres 
of the vertices. The large dimension $\overline{{\rm dim}}(f)$ of the chain 
is the usual dimension of the graph $G_f$. For a chain $f$ which comes from a
graph $G$, the small dimension is the dimension of $G$, the large dimension is the dimension
of $G_1$ which is the average over all large dimensions of the unit spheres $S(x)$, 
over all vertices $x$ in $G_1$. The dimension inequality follows from the more general
statement: 

\begin{lemma}[Generalized dimension inequality]
If $f$ is a chain then $\overline{{\rm dim}}(f) \geq \underline{{\rm dim}}(f)$. 
\end{lemma}

\begin{proof}
The more general statement can now be proven with induction with respect
to a lexicographic ordering which looks first at the
the number of variables appearing in the chain and then at the number of monomials.
This is a total ordering so that induction works. 
Unlike for graphs, the unit sphere of a vertex $x$ is now a chain 
with one variable less for which we can look at the small and large dimension.
Also the unit sphere of a simplex $x$ is a chain for which we can look at the
small and large dimension. The unit sphere of a simplex has in general the same
number of variables but less monomials. 
\end{proof}

{\bf Example.} \\
The chain $f=xyzw + xyz + xy + xw + x + y$ has two vertices $x,y$.
(A) Lets look at $x$ first: the unit sphere of $x$ is $S_f(x)=yzw+yz+y+w$ which has two vertices $y,w$. 
(i) The vertex $w$ in $S_f(x)$ has the unit sphere $S(w)=yz$ which has no vertices so that $w$ has dimension $0$. \\
(ii) The vertex $y$ in $S_f(x)$ has the sphere $S(y)=zw+z$,
a chain with one vertex $z$ which has a $0$-dimensional sphere $w$ so that $z$ has dimension $1$,
$S(y)=z+zw$ has dimension $1$ and $y$ has dimension $2$ in $f$. \\
The dimension of the vertex $x$ is therefore $(0+2)/2=1$. \\
(B) The unit sphere of $y$ in $f$ is $-xzw-xz$ which has no vertex so that $y$ has dimension $0$.
(C) The chain has small dimension $(0+1)/2=1/2$. \\
The two computations (A)-(C) show that $\underline{{\rm dim}}(f)=1/2$. 
The large dimension of the chain is the average of the large dimensions of the monomials in $f$.
It is the dimension of a graph $G_f$ with 6 vertices and 11 edges. It has dimension ${\rm dim}(G_f)=2.7333$
which is $\overline{{\rm dim}}(f)$.  \\

{\bf Remarks:} \\
{\bf 1)} The inequality looks like the known inequality between Hausdorff dimension and 
inductive dimension in the continuum. But we do not see any relation yet. \\
{\bf 2)} When iterating the refinement construction, we get a sequence of graphs $G_n$
with a dimension which converges to an integer. 

Lets elaborate on this remark a bit more. 
The largest simplex will be the seed which grows fastest and take over. 

\begin{propo}[Refinement asymptotics]
Start with a finite simple graph $G=G_0$ and define recursively $G_n=G_{n-1} \times K_1$, 
then ${\rm dim}(G_n) \to d$, if $K_{d+1}$ is the largest clique in $G$. 
\end{propo}
\begin{proof}
The dimensions can only grow or stay by the lemma. Since the dimensions are
bounded by $d$ for all $n$, the dimensions must converge to some value 
smaller or equal than $d$. The limiting value is $d$ because the number
$v_d(G_n)$ of $K_{d+1}$ subgraphs of $G_n$ grows exponentially and the
growth of that large geometric cluster $H_n$ formed by the union of the
$K_{d+1}$ simplices takes over: the number $v_l(G_n)$ of $K_{l+1}$ 
subgraphs of $G_n$ is dominated by $v_l(H_n)$ of $K_{l+1}$ subgraphs of $H_n$. 
\end{proof}

\begin{figure}[h]
\scalebox{0.12}{\includegraphics{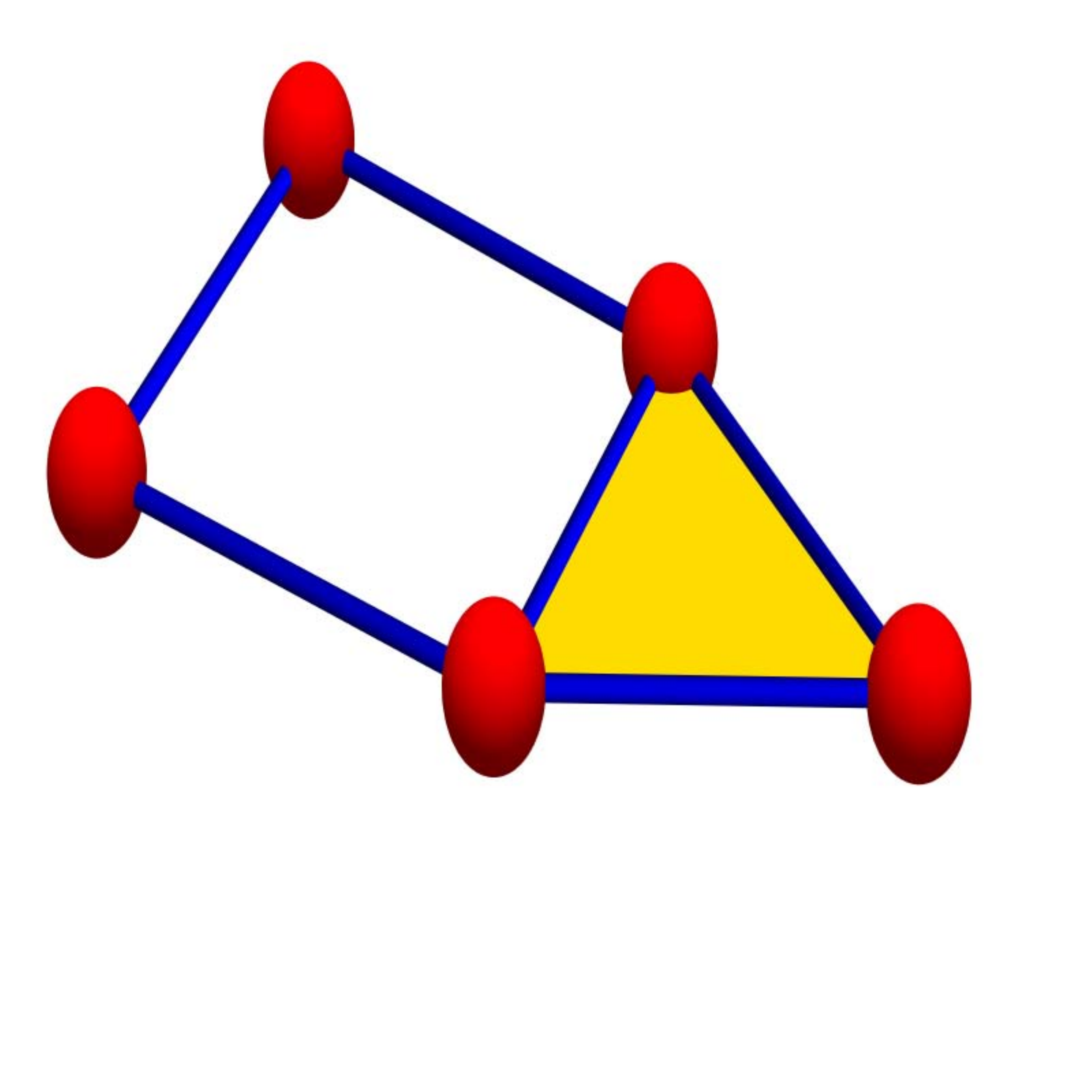}}
\scalebox{0.12}{\includegraphics{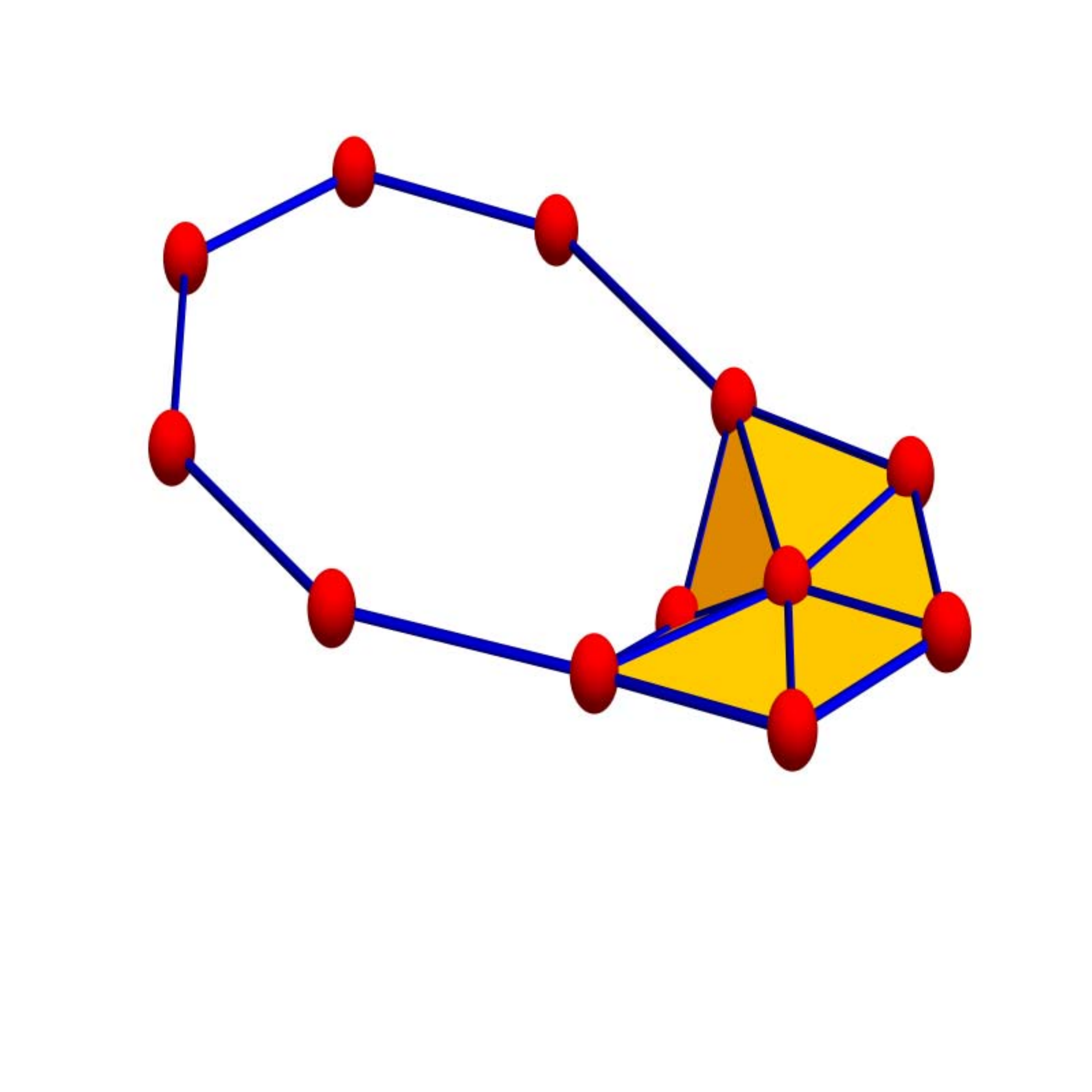}}
\scalebox{0.12}{\includegraphics{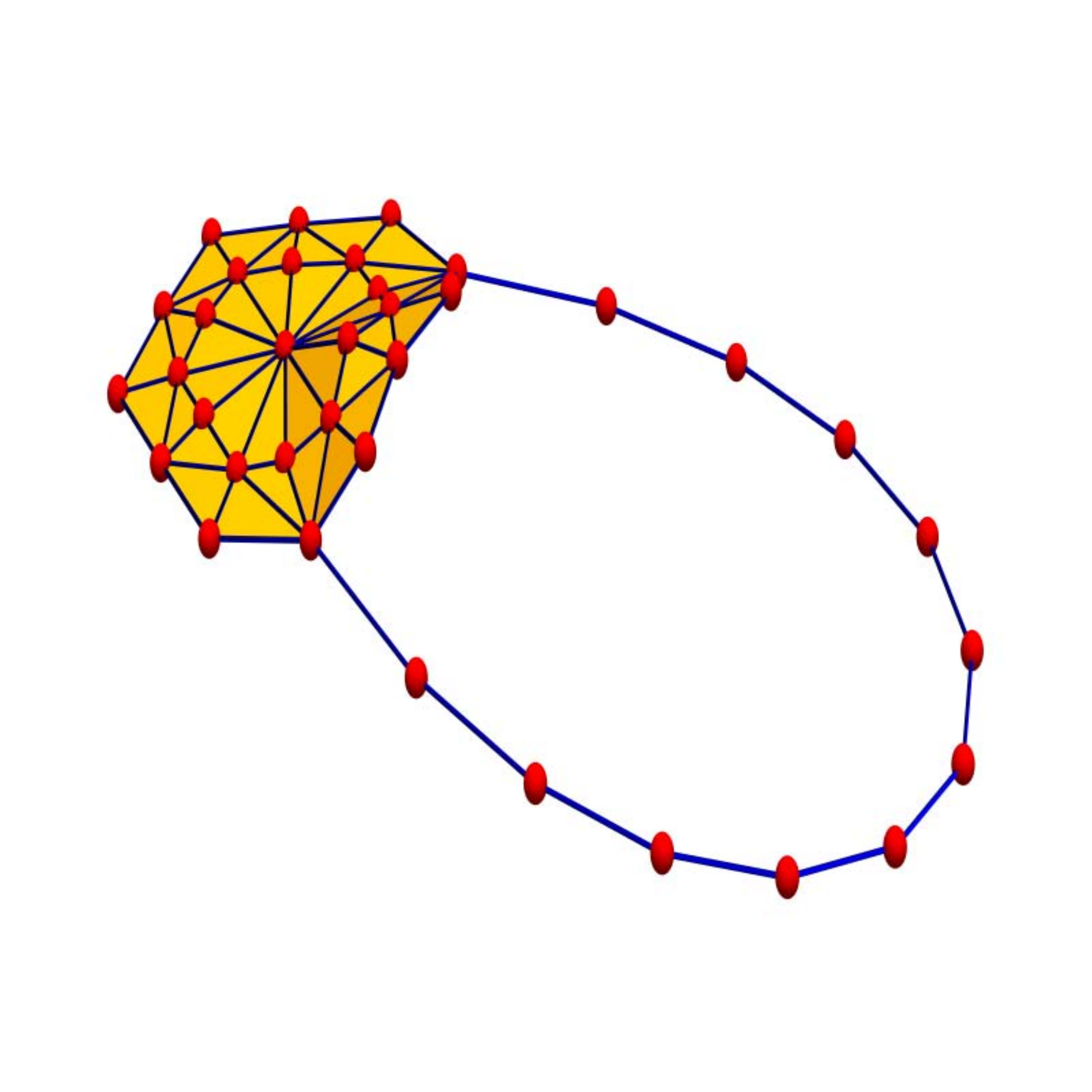}}
\caption{
\label{euclideanlimit}
Starting with the house graph $G$ to the left, we build $G_1 = G \times K_2$,
then $G_2 = G_1 \times K_1$ etc. The graph remains non-geometric
but the geometric part takes over more and more weight as higher dimensional 
simplices divide into more pieces. In the limit we get to a geometric graph. 
We see that even when starting with a random graph, the refinement process
moves the graph to a geometric graph. This is why geometric graphs are so 
natural. What happens from $G_n \to G_{n+1}$ is that the simplices of $G_n$
become the vertices of $G_{n+1}$. As we like to think about simplices as
basic ``points", this picture is a natural evolution. 
}
\end{figure}

We see that when starting with an arbitrary network, we asymptotically 
get to a geometric graph. This could be interesting if we look at a
dynamical emergence of graphs through refinement or evolution.
It would explain, why geometric and manifold-like structures with integer
dimension appear naturally: the lower dimensional structures grow
slower and are overtaken by the growth of largest structures which make sure
that a larger and larger part of $G_n$ will be geometric. The geometric part
grows fast like the novo-vacuum in Schild's ladder
where physics is described by ``Sarumpaet rules" \cite{Egan}. Egan's novel 
starts with with the words: 
{\it In the beginning was a graph, more like diamond than graphite.} 

\begin{figure}[h]
\scalebox{0.40}{\includegraphics{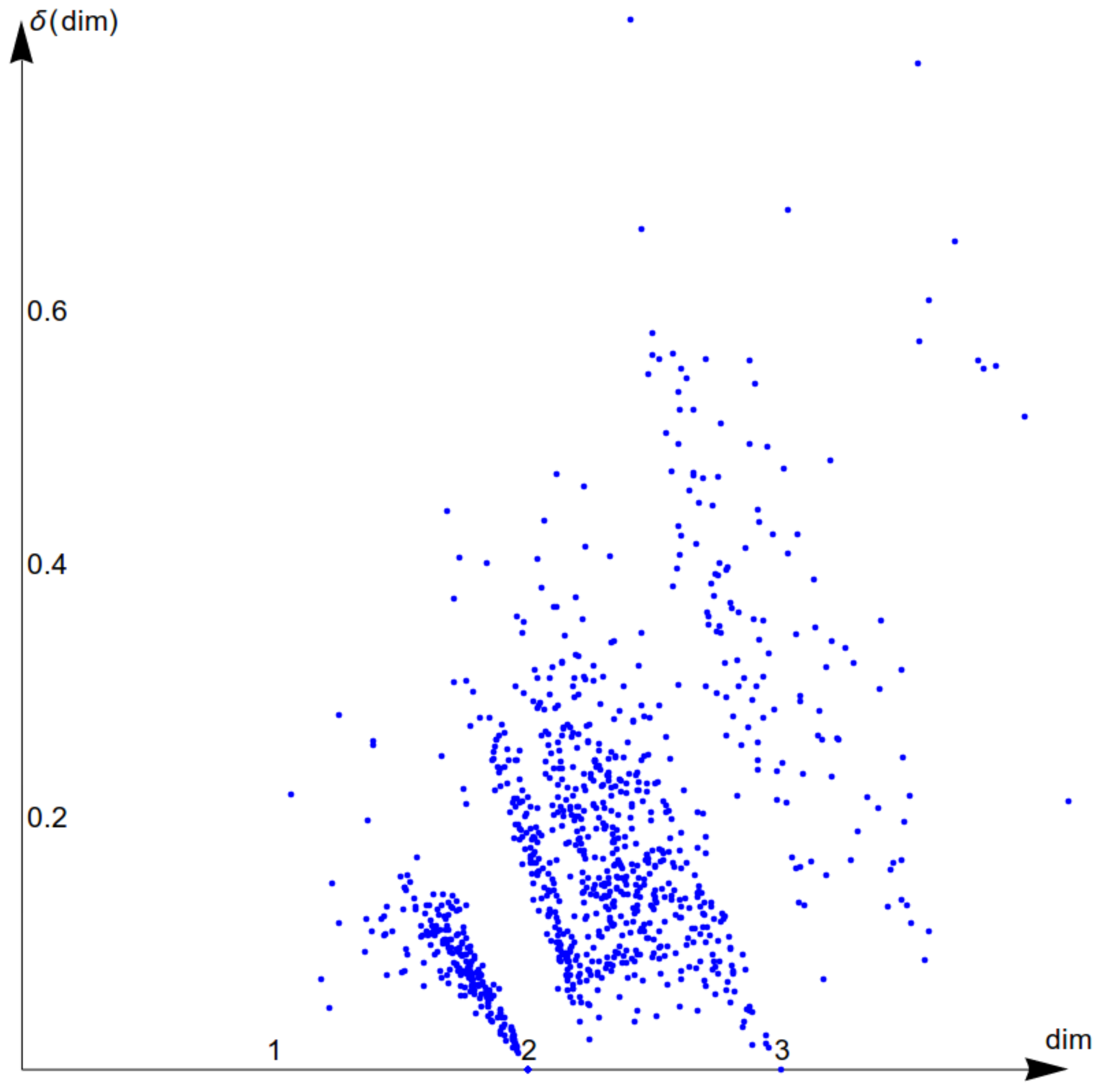}}
\caption{
\label{dimensiondiscrepancy}
The dimension inequality is illustrated quantitatively by computing the 
discrepancy $\delta = {\rm dim}(G \times K_1) - {\rm dim}(G)$ for
2000 random Erd\"os-R\'enyi graphs in $E(12,0.5)$ in relation to the dimension ${\rm dim}(G)$. 
The inequality becomes an equality for geometric graphs. There, the
dimension of $G$ is an integer and agrees with the dimension of the enhanced graph 
$G_1 = G \times K_1$.
}
\end{figure}

The dimension of $G \times K_1$ is not necessarily equal to the dimension of $G$. For example, the dimension
of the lollipop graph $G$ is $5/2=2.5$, the dimension of $G_1 = G \times K_1$ is $11/4=2.75$ is slightly higher. 
While $G \times K_1$ has a cover for which the nerve graph is $G$, the dimensions don't agree because
higher dimensional parts of $G$ contribute more to $G_1$.  \\

We have looked a bit at the statistics 
and computed the difference ${\rm dim}(G \times K_1) - {\rm dim}(G)$
which is related to the variance of the dimension random variable $x \to {\dim}_G(x)$ on a graph.
See Figure~(\ref{dimensiondiscrepancy}). \\
The preservation of dimension is an essential ingredient of a graph homeomorphisms, as it is in the
continuum. It is no surprise that in the fractal case, we have the same inequality like
Hausdorff dimension: ${\rm dim}(X \times Y)  \geq {\rm dim}(X) + {\rm dim}(Y)$. 
Note however that the equality ${\rm dim}(G \times H ) = {\rm dim}(G_1) + {\rm dim}(H_1)$ 
holds in full generality, for any finite simple graph. We use the notation ${\rm dim}(G)(x)={\rm dim}(S_G(x)$
where $S_G(x)$ is the unit sphere of the vertex $x$ in $G$. 

\begin{thm}[Dimension of product]
\label{inductivedimension}
Given two finite simple graphs $G,H$. The dimension of $G \times H$ is
equal to the sum of the dimensions of $G_1 = G \times K_1$ and $H_1 = H \times K_1$. 
More precisely, the dimension of any vertex $(x,y)$ in $G \times H$ the dimension
is the sum of the dimensions of $x$ in $G_1$ and $y$ in $H_1$:
$$   {\rm dim}(G \times H)(x,y) = {\rm dim}(G_1)(x) + {\rm dim}(H_1)(y) \; . $$
\end{thm}

\begin{proof}
Use induction with respect to the sum of dimension. Fix $x \in G_1, y \in H_1$ 
and assume ${\rm dim}(B(x))=a$, ${\rm dim}(B(y))=b$ so that
$a-1={\rm dim}(S(x))$ and $b-1={\rm dim}(S(y))$. 
We only have to show that the unit sphere of $(x,y)$ in $(G \times H)$
has dimension $a+b-1$ so that $(x,y)$ has dimension $a+b$. 
The sphere $S(x,y)$ is the union of two graphs $B(x) \times S(y) \cup S(x) \times B(y)$, 
which overlap in $S(x) \times S(y)$. 
The two graphs have both dimension $c=a+b-1$ and $S(x) \times S(y)$ has dimension $c-1=a+b-2$. 
Lemma: if two graphs $A,B$ have dimension $c$ and their intersection has dimension $c-1$ 
then $A \cup B$ has dimension $c$. 
\end{proof}

For example, if $G$ is the {\bf tadpole graph} and $H$ is the house graph,
then ${\rm dim}(G_1) = 71/44$ and ${\rm dim}(G_2) = 37/24$. 
The product graph $G \times H$ has dimension $833/264 = 71/44 + 37/24$.

\begin{coro}[Dimension inequality]
If $G,H$ are finite simple graphs, then ${\rm dim}(G \times H) \geq {\rm dim}(G) + {\rm dim}(H)$. 
\end{coro}

\begin{proof}
We have ${\rm dim}(G_1) \geq {\rm dim}(G)$ and ${\rm dim}(H_1) \geq {\rm dim}(H)$.
The result follows from Theorem~(\ref{inductivedimension}).
\end{proof}

Given a geometric graph $G$. A subgraph $U$ of $G$ is called convex if 
between any two vertices $x,y$ in $U$, there exists a shortest connection between $x,y$
which is inside $U$. Note that shortest connections are not necessarily unique. 
The convex hull of a finite set $Y$ of vertices in $G$ is the intersection of all convex
subgraphs of $G$ which contain $Y$. 

\begin{thm}[Topology]
a) If $G$ is geometric, then $G_1=G \times K_1$ is homeomorphic to $G$ in the sense of \cite{KnillTopology}. \\
b) If $G$ has no triangles, then $G_1 = G \times K_1$ is homeomorphic to $G$ in the classical sense of
topological graph theory. 
\end{thm}

\begin{proof}
a) Assume the graph has dimension $k$. 
Let $x_1,\dots,x_n$ be the vertices of $G$. We have to find a cover $\{ U(x_i) \}$
of $G_1$ such that the nerve graph of the cover is the original graph $G$. As
$U(x_i)$ we look at the unit sphere $S_G(x_i)$ in the original graph and build the
smallest convex hull in $G_1$ containing all the vertices of this sphere. 
The intersection of two $U(x),U(y)$ is $k$-dimensional and nonempty if and only if
$x,y$ were connected in the original graph. 
Note that according to the definition \cite{KnillTopology},
intersections $U(x) \cap U(y)$ which are lower dimensional do not count as
edges in the nerve graph.  \\
b) The graph $G_1$ is a refinement, where every edge has got an additional vertex. 
\end{proof} 

{\bf Examples.} \\
{\bf 1)} If $G=C_n$ then the $U(x_i)$ are balls of radius $2$ in the new graph $G_1$ 
centered at the original vertices. Two sets $U(x),U(y)$ overlap in line graphs of
diameter $2$ if $x$ and $y$ were connected in $G$. \\

\begin{thm}[Continuity]
If $G$ and $H$ are homeomorphic geometric graphs and $K$ is a third geometric graph, then
$G \times K$ is homeomorphic to $H \times K$. 
\end{thm}
\begin{proof}
We use that $G$ and $G_1$ are homeomorphic and that $H$ and $H_1$ are homeomorphic. 
if $U_i$ is a cover of $G_1$ and $V_i$ is a cover of $H_1$ such that the nerve graphs
of are the same and $W_j$ is a cover of $K_1$, then we have a cover $(U_i \times W_j)$
and $(V_i \times W_j)$ of $G \times K$ and $H \times K$. 
\end{proof}

{\bf Remark:} We don't yet know whether this can be proven for arbitrary finite simple graphs. One
first has to show that in full generality that $G$ and $G_1$ are homeomorphic. This is not
yet done. \\

Let $G$ be a finite simple graph. Define the {\bf curvature of a simplex $x$ in $G$}
as the usual curvature \cite{cherngaussbonnet} of the vertex $x$ in the graph $G \times K_1$.

\begin{thm}
The simplex curvature $K(x)$ satisfies the Gauss-Bonnet relation $\sum_{x} K(x) = \chi(G)$.
The function $K(x)$ is equal to the expectation $E[i_f(x)]$ when integrating over all 
colorings $f$ on the simplex graph.
\end{thm}

See \cite{indexexpectation,colorcurvature}. \\

{\bf Remark:} 
We have already a natural "curvature" for a graph located on simplices, it is constant
$1$ on the even dimensional simplices and $-1$ on the odd dimensional simplices. But this
has little to do with the usual curvature. \\

Now, we look at symmetries: 

\begin{lemma}
Given two finite simple graphs $G,H$. The automorphism group ${\rm Aut}(G)$ 
is a subgroup of the automorphism group of $G \times H$. 
\end{lemma}

\begin{proof}
Each element $T$ of ${\rm Aut}(G)$ produces a permutation of the variables $x_1,\dots,x_n$
appearing in the algebraic description $g$ of $G$. If $h$ is the algebraic description of $H$
using variables $y_1,\dots,y_m$, then $A$ acts on $fg$ by permutations of the elements $x_1,\dots,x_n$
and produces a symmetry of $G \times H$. 
\end{proof} 

{\bf Remarks:} \\
{\bf 1)}  The case $G_1 = G \times K_1 \sim C_8$ with $G=C_4$ shows that 
the symmetry group can become bigger.  \\
{\bf 2)}  If the vertices of $G_1 = G \times K_1$ form a group and 
also the vertices of $H_1 = H \times K_1$ have this property then 
the vertices of $G \times H$ form a group, the product group. \\

Next we look at homotopy: 

\begin{lemma}
If $G=A \cup B$ and $H=C \cup D$ are two finite simple graphs
and the subgraph $A$ of $G$ is homotopic to the subgraph $C$ of $H$ 
and $B$ is homotopic to $D$ and $A \cap B$ is homotopic to $C \cap D$,
then $A \cup B$ is homotopic to $C \cup D$. 
\end{lemma}

\begin{propo}
The graph $G_1 = G \times K_1$ is homotopic to $G$. 
\end{propo}

\begin{proof}
Use induction with respect to the size of the graph. Its true for $K_1$. 
Given a graph $G$ with $n$ vertices. Add a new vertex $x$ to get a larger
graph $H$ with $n+1$ vertices. The unit sphere $S(x)$ is in $G$. 
Now $G_1 = G \times K_1$ is a subgraph of $H_1 = H \times K_1$. 
Let $y_1,\dots,y_k$ denote the simplices in $S(x)$, they are the vertices in $S_1(x) 
= S(x) \times K_1$.  The graph $H_1$ has new vertices $N=\{y_1x,\dots,y_kx\}$ which were 
not in $G_1$. Our induction assumption is that $G_1$ is homotopic to $G$
and since $S(x)$ is a subgraph of $G$ also $S_1(x)$ is homotopic to $S(x)$. 
But then the graph in $H$ generated by $S(x) \cup \{x\}$ is
homotopic to the graph in $H_1$ generated by $S_1(x) \cup N$. 
Now $H$ is a union of two graphs $S(x) \cup x$ and $G$ intersecting
in $S(x)$ and $H_1$ is the union of two graphs $S_1(x) \cup N$ and $G_1$
intersecting in $S_1(x)$. Use the lemma. 
\end{proof}

\begin{thm}
Let $G,H,K$ be finite simple graphs. 
If $G,H$ are homotopic, then $G \times K$ and $H \times K$ are
homotopic. 
\end{thm}

\begin{proof}
Show it on an algebraic level: 
if $g \to g+k$ is a homotopy step, then $xg \to x(g+k)$ 
is a homotopy step on each fibre. 
\end{proof} 

It follows that the product defines a group operation
on the homotopy classes. This monoid defines then a Grothendieck group.  \\
The Cartesian product also defines a direct sum on isomorphism classes of 
vector bundles.  \\

{\bf Example.} \\
{\bf 1)} The product of two trees is contractible. \\
{\bf 2)} The product $C_4 \times H$ is never contractible. \\

Finally, lets look at the chromatic number $c(G)$. 

\begin{lemma}[Minimal coloring]
The graph $G_1= G \times K_0$ satisfies
$$ c(G_1) \leq c(G) \; . $$
It is minimally colorable. If $K_{n+1}$ is the largest
complete subgraph of $G$, then $c(G_1)=n+1$. 
\end{lemma}

\begin{proof}
If $x$ is a vertex in $G_1$, denote by $k={\rm Dim}(x)$ the dimension of the corresponding
simplex $x$ in $G$. The function ${\rm Dim}$ is a coloring because two adjacent
vertices have different dimension. Because ${\rm Dim}$ takes values from $0$ to $n$
if $K_{n+1}$ is the largest simplex in $G$, the chromatic number is $n+1$. 
\end{proof} 

{\bf Remarks.}  \\
{\bf 1)} We see that for graphs $G_1$ we can immediatly compute the chromatic number
by just looking at the largest clique. \\
{\bf 2)} For geometric graphs $G$, this means that $G_1$ is minimally colorable
and therefore is Eulerian in the sense of \cite{KnillEulerian}. For two dimensional 
graphs, this is equivalent to Eulerian in the classical sense. \\ 

We can actually compute the chromatic number of any product graph $G \times H$: 

\begin{thm}[Chromatic number of product]
\label{chromaticnumberofproduct}
If $G$ has a largest clique $K_{n}$ and $H$ has
a largest clique $K_m$, then $c(G \times H) = n+m-1$. 
\end{thm}

\begin{proof}
Again just look at the function ${\rm Dim}$. It is locally 
injective and is so a coloring. 
\end{proof}

This result actually has the previous lemma as a corollary 
because $K_1$ has the largest clique $K_1$ so that 
$c(G \times K_1) = n+1-1=n$.  \\

{\bf Examples.} \\
{\bf 1} The chromatic number of a product of two trees is $3$. \\
{\bf 2} The chromatic number of $K_2 \times K_2$ is $3$. 
The chromatic number of $K_2 \times K_2 \times K_2$ is $4$. 

\section{Examples}

{\bf Example}. \\
Let $G$ be the house graph and $H$ the Lollipop graph. We have ${\rm dim}(G)=22/15,
b_0(G)=1,b_1(G)=1,b_k(G)=0,k>1$ and $b_0(H)=1,b_k(H)=0, k>0$, 
${\rm dim}(H)=5/2$. The graphs $G_1=G \times K_1, H_1=H \times K_2$ satisfy ${\rm dim}(G_1)=37/24$
and ${\rm dim}(H_1)=11/4$. The graph $G \times H$ has dimension $103/24$. The dimensions of
the individual points are given by the dimension spectrum 
${\rm dimspec}(G_1)= $\{ 1, 1, 1, 7/4, 2, 2, 2, 7/4, 1, 2, 1, 2 \}$,
{\rm dimspec}(H_1) = \{3, 11/4, 3, 3, 3, 3, 3, 3, 1, 3, 3, 3, 3, 3, 3, 3, 1\}$. 
The dimension spectrum of $G \times H$ is the Cartesian product of these two 
lists: 
${\rm dimspec}(G \times H) = 
\{$15/4, 15/4, 15/4, 4, 4, 4, 4, 4, 4, 4, 4, 4, 4, 4, 4, 4, 4, 4, 4, 4, 
4, 4, 4, 4, 2, 2, 2, 9/2, 19/4, 19/4, 19/4, 19/4, 5, 5, 5, 19/4, 5, 
5, 5, 19/4, 5, 5, 5, 19/4, 5, 5, 5, 19/4, 5, 5, 5, 19/4, 5, 5, 5, 
19/4, 5, 5, 5, 11/4, 3, 3, 3, 9/2, 15/4, 19/4, 19/4, 4, 5, 19/4, 4, 
5, 19/4, 4, 5, 19/4, 4, 5, 19/4, 4, 5, 19/4, 4, 5, 19/4, 4, 5, 11/4, 
2, 3, 15/4, 4, 4, 4, 4, 4, 4, 4, 2, 19/4, 5, 5, 5, 5, 5, 5, 5, 3, 4, 
4, 4, 4, 4, 4, 4, 4, 4, 4, 4, 4, 19/4, 5, 5, 5, 19/4, 5, 5, 5, 19/4, 
5, 5, 5, 19/4, 5, 5, 5, 19/4, 4, 5, 19/4, 4, 5, 19/4, 4, 5, 19/4, 4, 
5, 4, 4, 4, 4, 5, 5, 5, 5, 4, 4, 4, 4, 4, 4, 19/4, 5, 5, 5, 19/4, 5, 
5, 5, 19/4, 4, 5, 19/4, 4, 5, 4, 4, 5, 5, 4, 4, 4, 19/4, 5, 5, 5, 
19/4, 4, 5, 4, 5, 2, 2, 2, 11/4, 3, 3, 3, 11/4, 2, 3, 2, 3$\}$. 
This list had been computed directly by computing the dimensions of each 
unit sphere in $G \times H$. The K\"unneth formula tells that $G \times H$
has the Betti numbers $b_0=1,b_1=1, b_k=0$ again for $k>1$. \\

\begin{figure}[h]
\scalebox{0.22}{\includegraphics{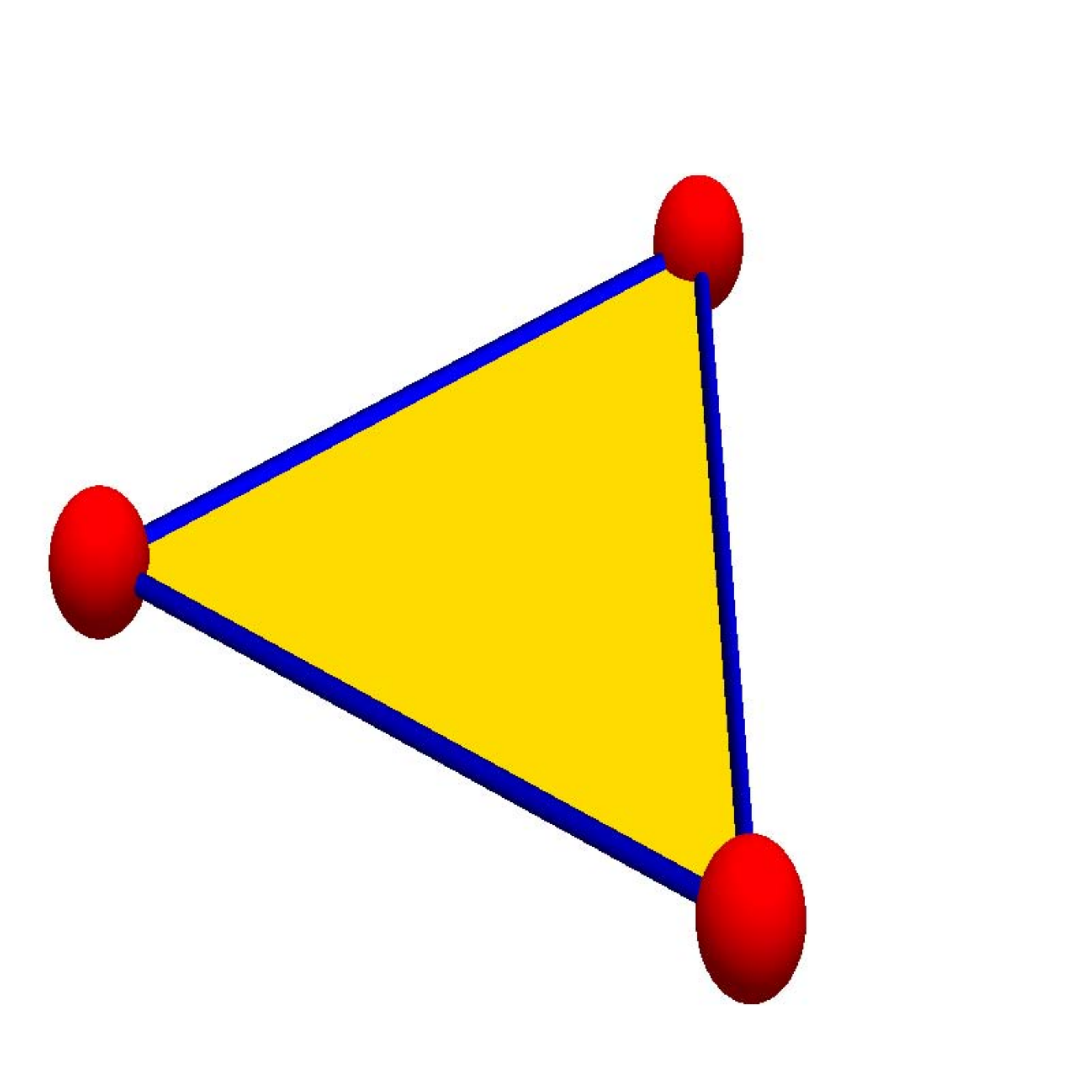}}
\scalebox{0.22}{\includegraphics{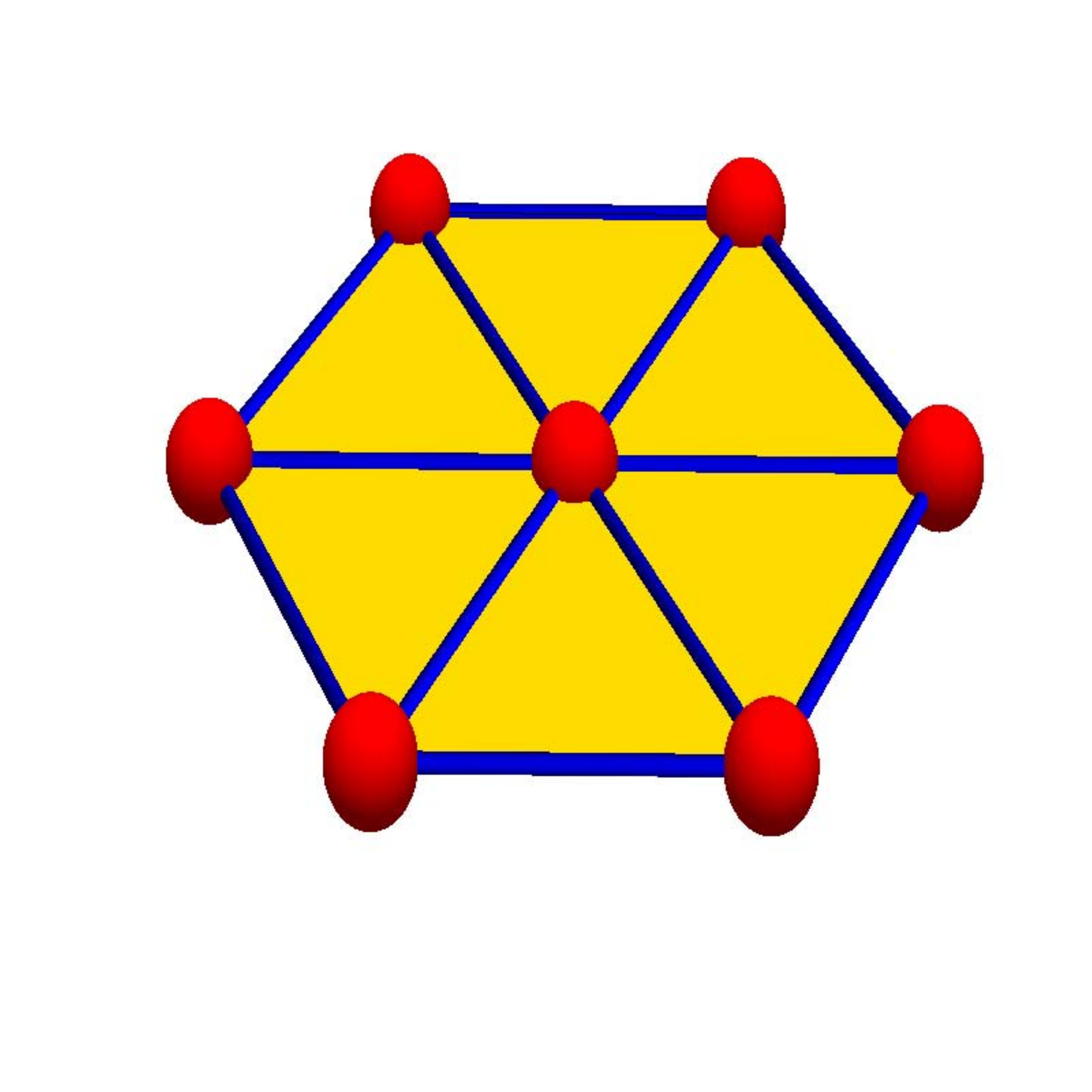}}
\scalebox{0.22}{\includegraphics{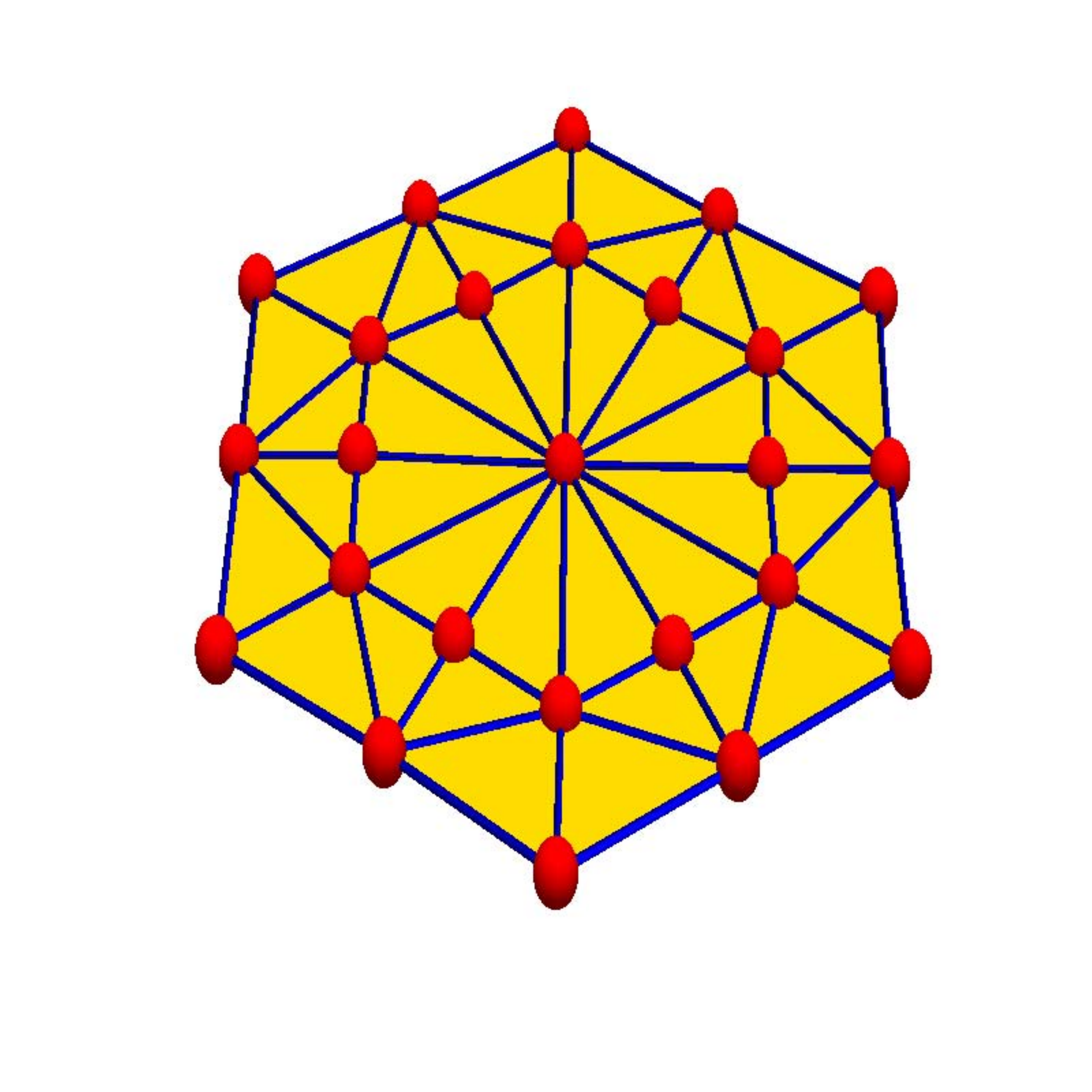}}
\scalebox{0.22}{\includegraphics{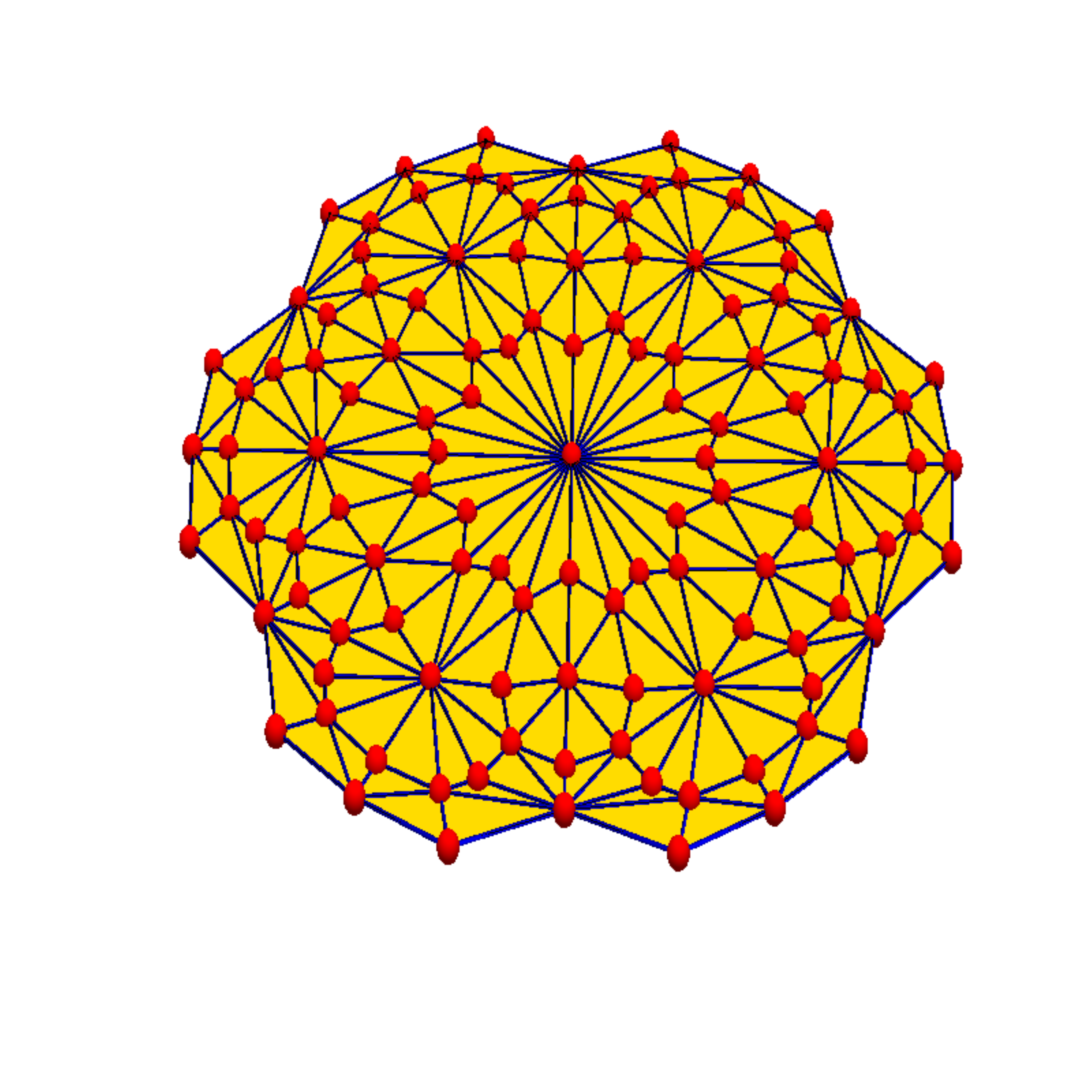}}
\caption{
$K_3, K_3 \times K_1, (K_3 \times K_1) \times K_1$, 
$((K_3 \times K_1) \times K_1) \times K_1$. 
}
\end{figure}

\begin{figure}[h]
\scalebox{0.22}{\includegraphics{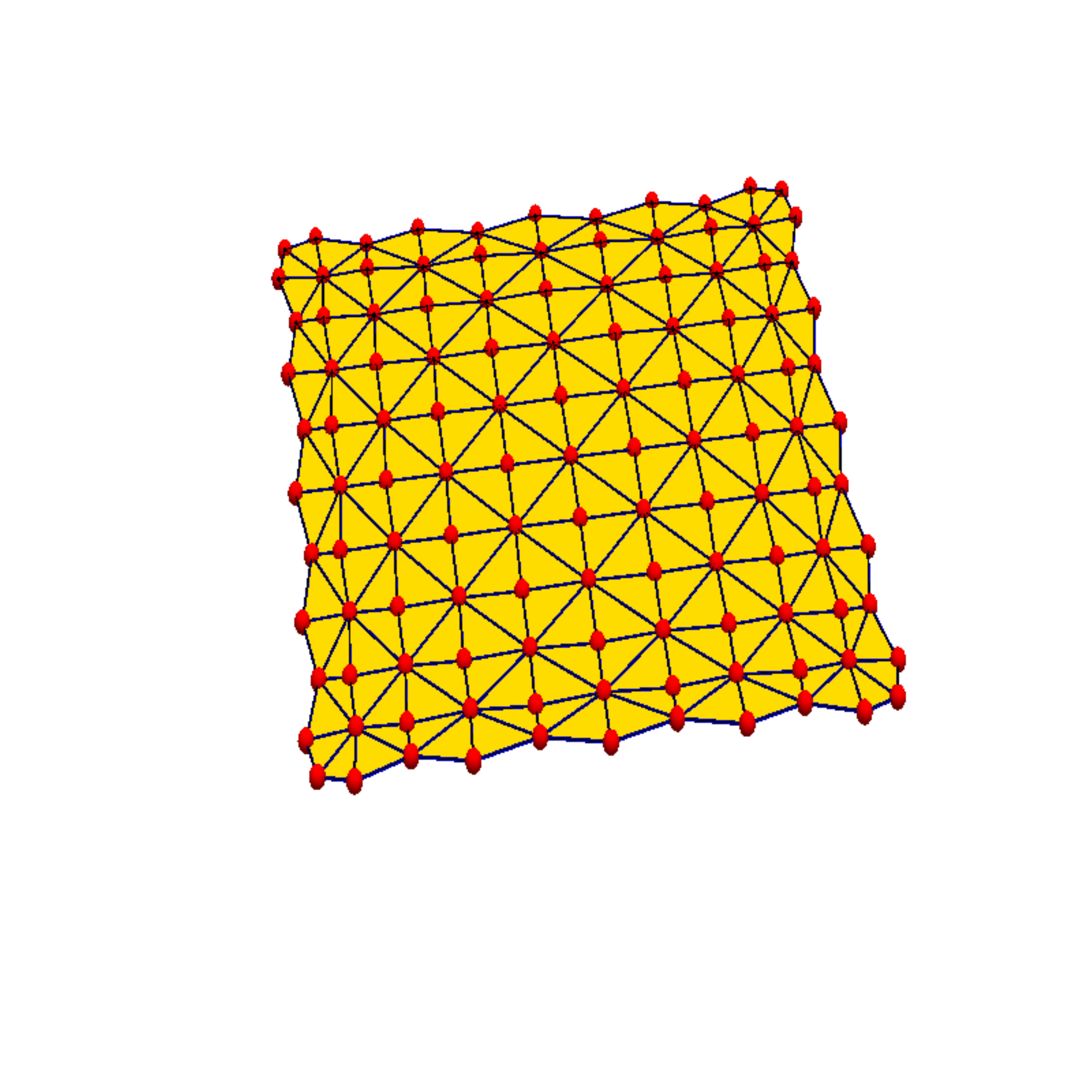}}
\scalebox{0.22}{\includegraphics{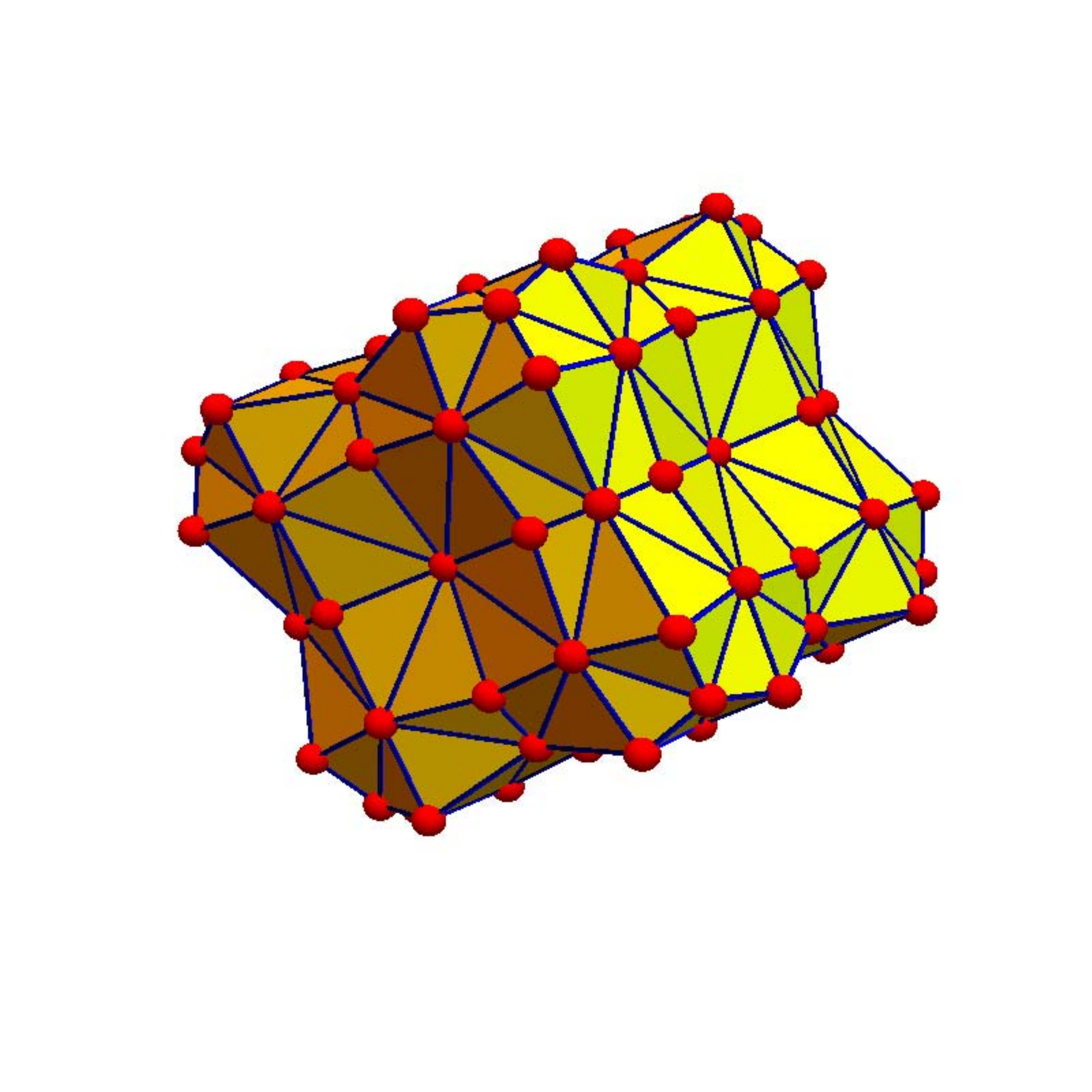}}
\caption{
The graph $L_6 \times L_6$ where $L_n$ is the line graph.
This is a discrete square. The second picture shows
$L_3 \times L_3 \times L_3$, a discrete cube. 
}
\end{figure}

\begin{figure}[h]
\scalebox{0.22}{\includegraphics{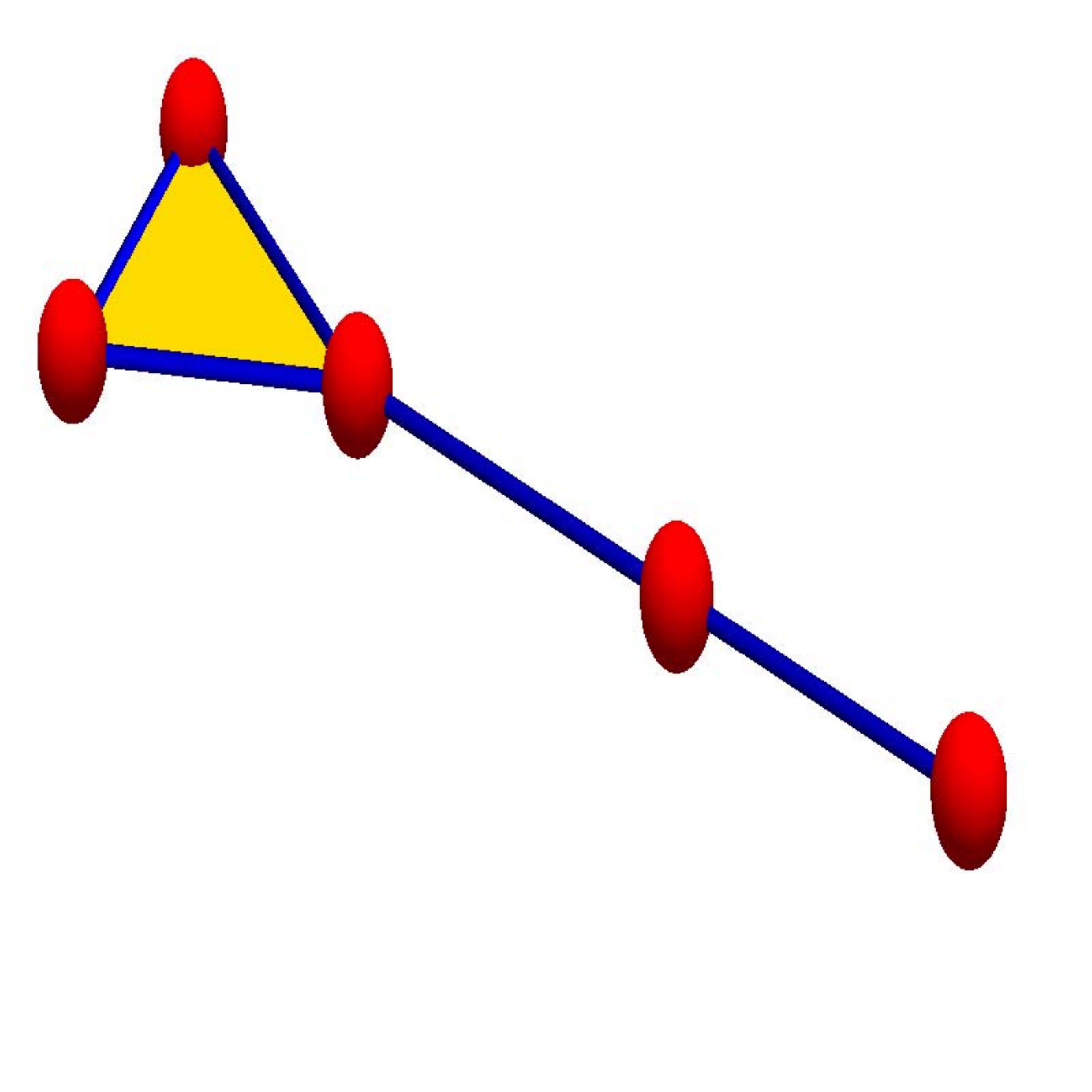}}
\scalebox{0.22}{\includegraphics{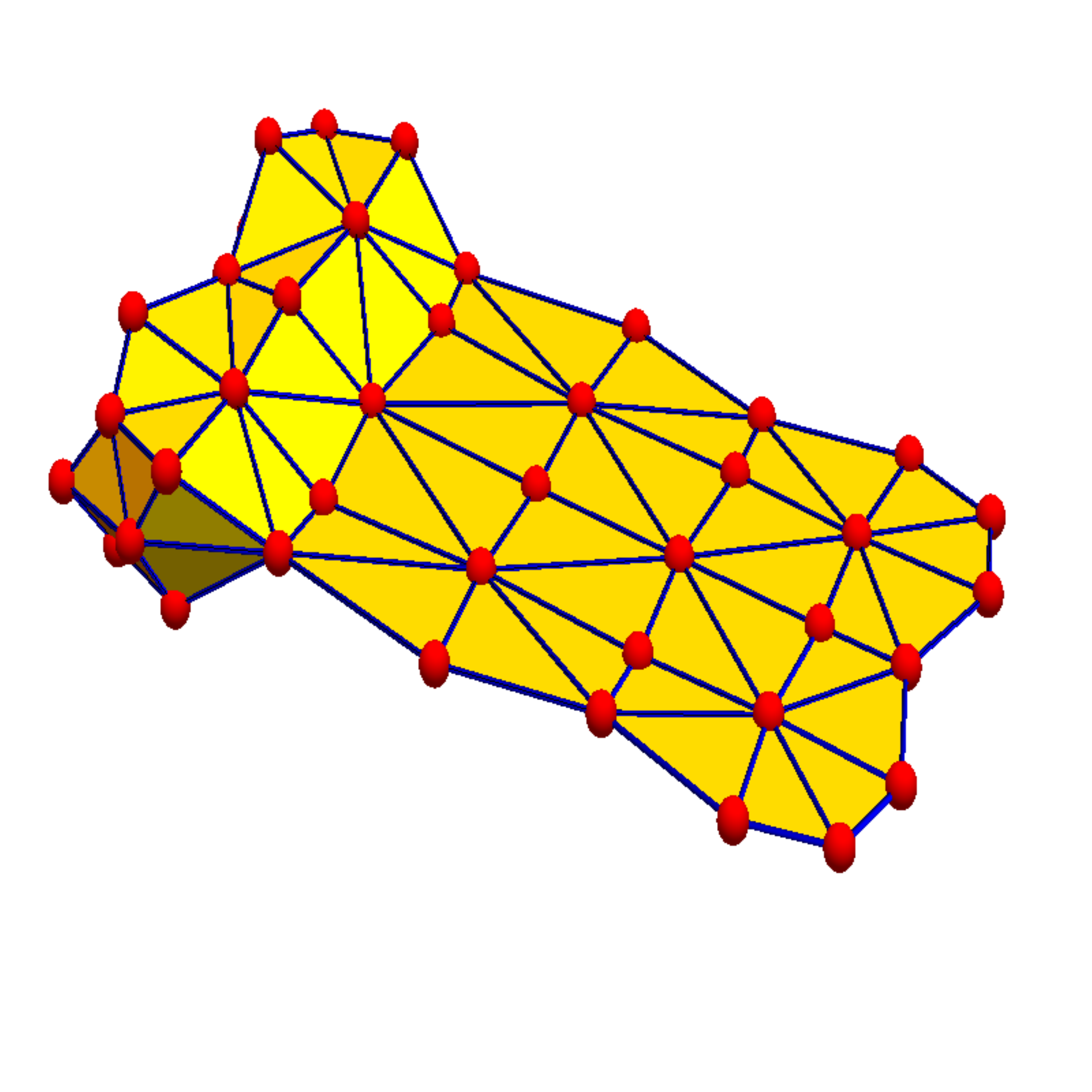}}
\caption{
The lollipop graph $G$ is partly $2$-dimensional, partly 
$1$-dimensional so that it is a graph with fractal dimension. 
The graph $G \times L_3$ is partly $2$-dimensional and partly 
$3$-dimensional, also a fractal dimension. 
}
\end{figure}

{\bf Example}. \\
{\bf 1)} For an octahedron $G$, the ring element is $f_G$ is
$a+b+ab+c+ac+bc+abc+d+ad+bd+abd+e+ae+ce$ $+ace+de+ade+f+bf+cf+bcf+df+bdf+ef+cef+def$.
The graph $G \times K_1$ is a barycentric subdivision of $G$ of the same dimension. 
For a geometric graph $G$, the graph $G \times K_1$ is geometric again.

\begin{figure}[h]
\scalebox{0.22}{\includegraphics{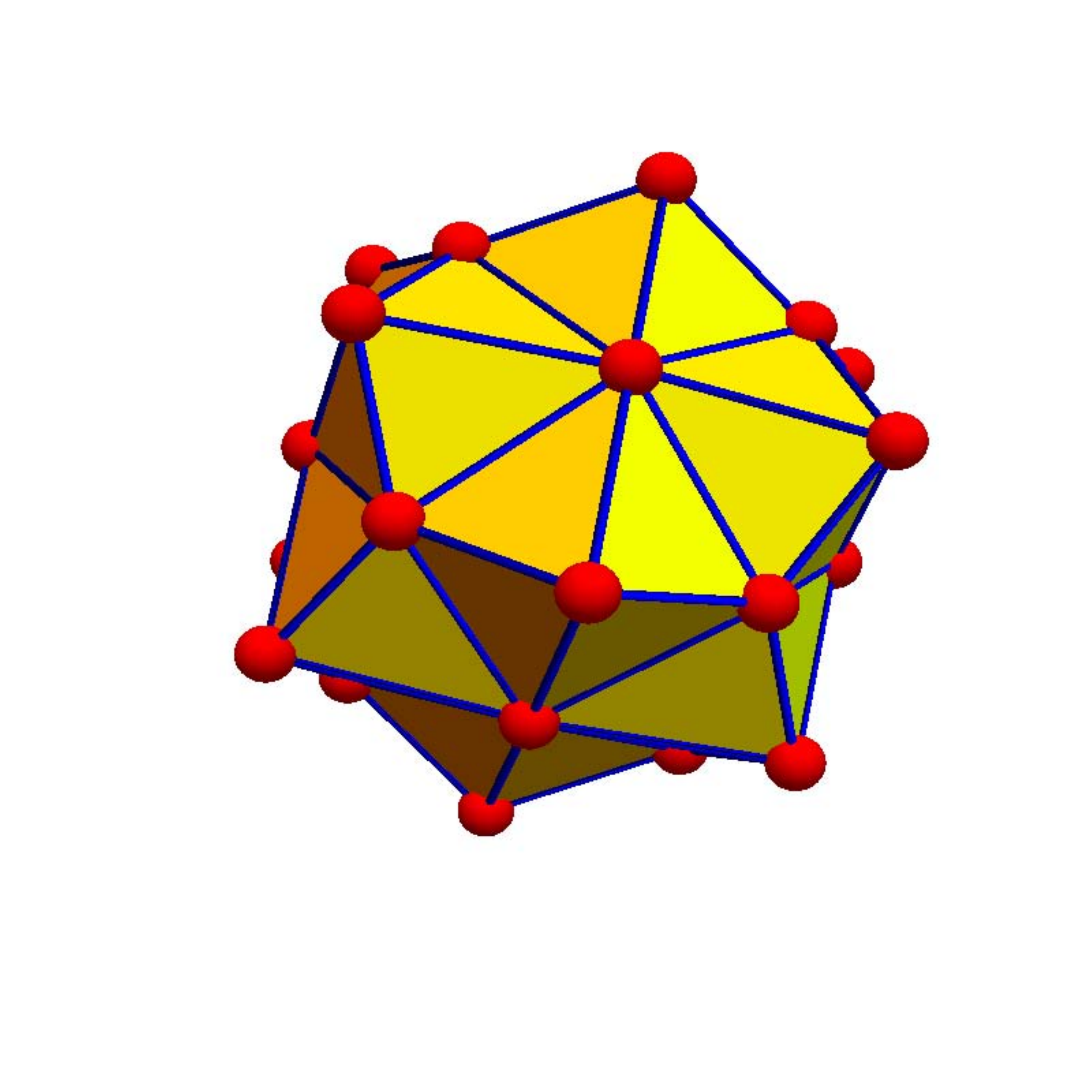}}
\scalebox{0.22}{\includegraphics{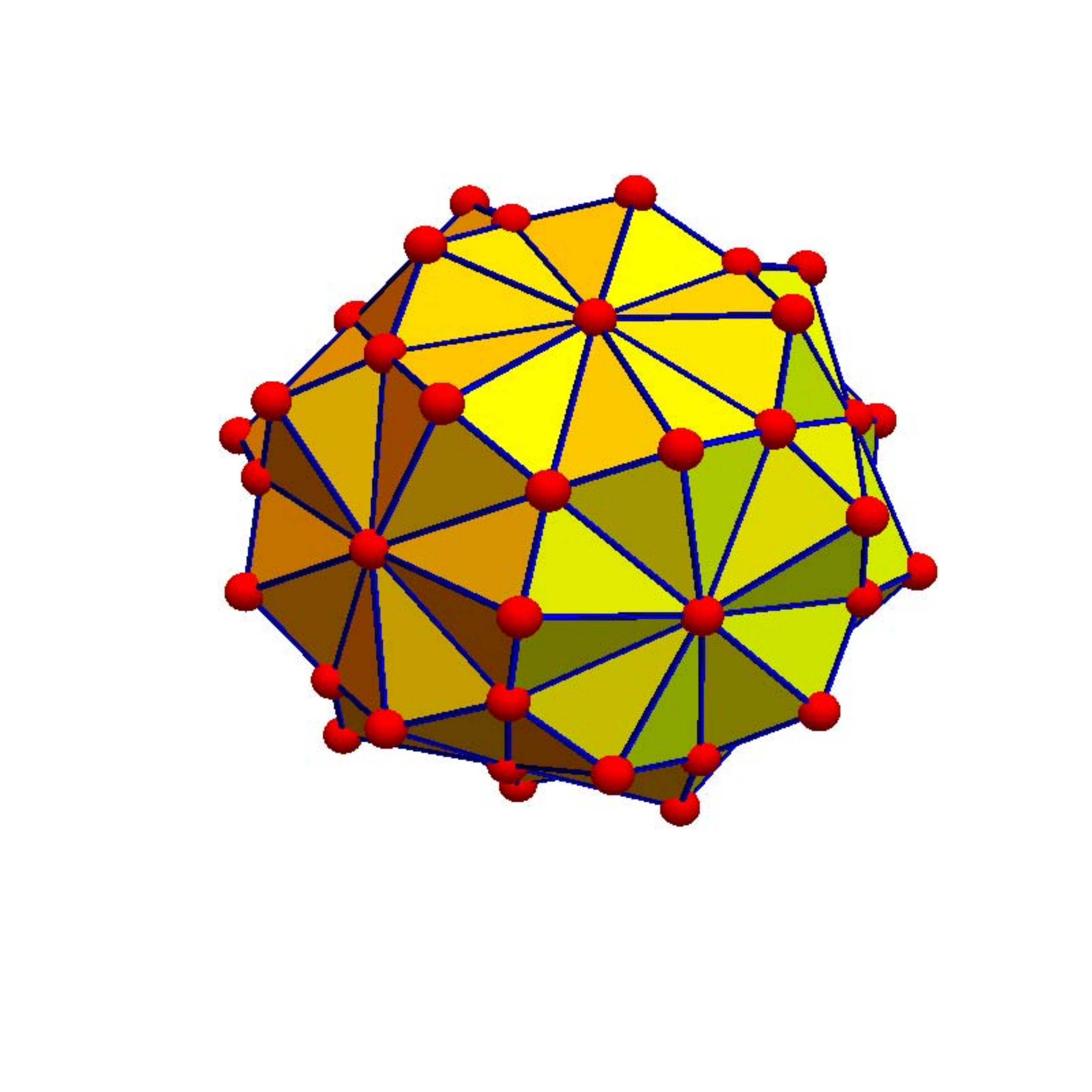}}
\caption{
The product $O \times K_1$ where $O$ is the octahedron graph.
The second picture shows $I \times K_1$, where $I$ is the icosahedron graph. 
}
\end{figure}

{\bf 2)} Let $O$ be the octahedron, the product $O \times C_4$ is a $3$-dimensional graph. It is a
discrete incarnation of the $3$-manifold $S^2 \times S^1$. 

\begin{figure}[h]
\scalebox{0.22}{\includegraphics{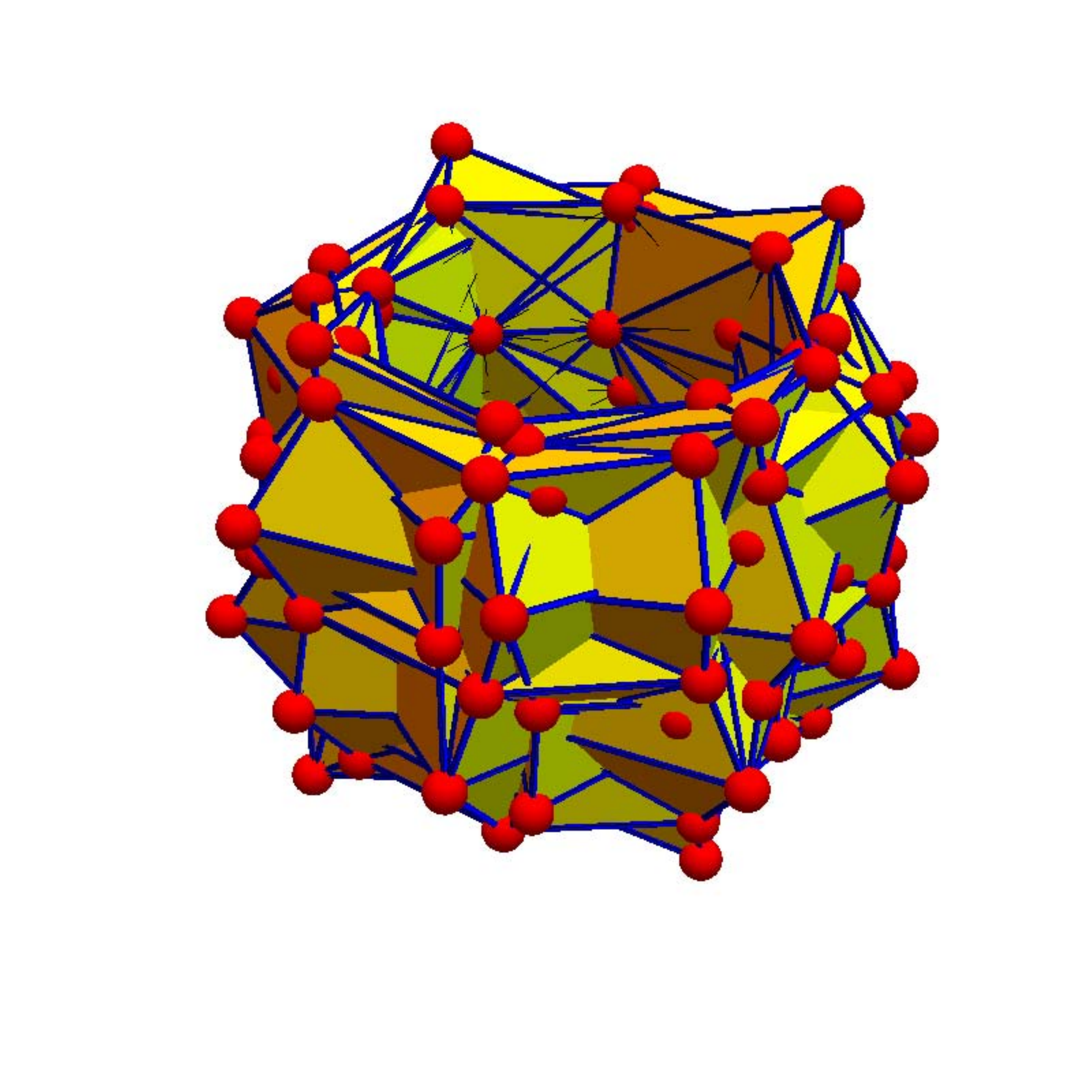}}
\scalebox{0.22}{\includegraphics{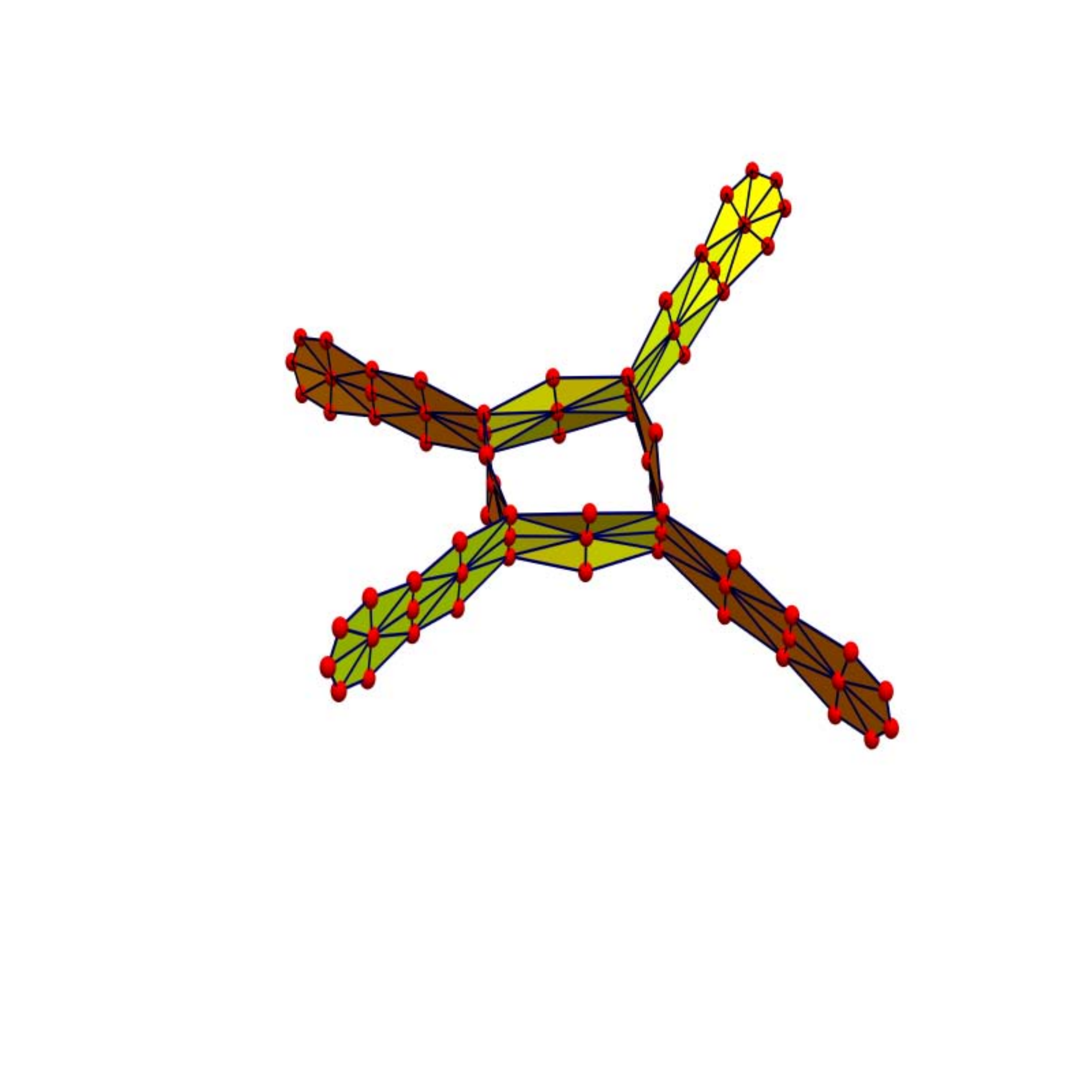}}
\caption{
The product $O \times C_4$ where $O$ is the octahedron graph. 
This is a discrete 3 manifold.
The second figure shows the product of the sun graph $S_{2,2,2,2}$ 
(a cyclic graph $C_4$ with rays) and with $K_2$. 
}
\end{figure}

\begin{figure}[h]
\scalebox{0.22}{\includegraphics{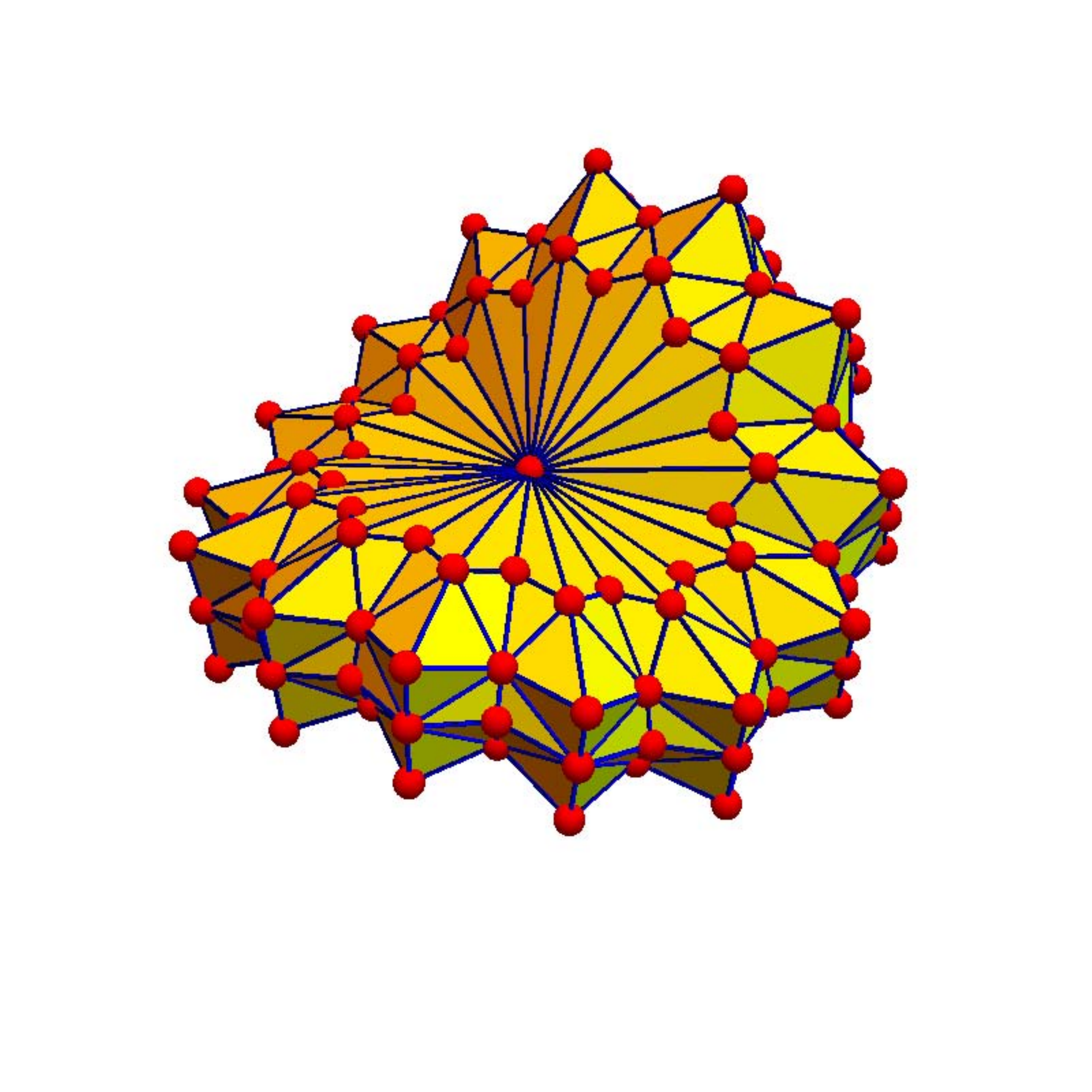}}
\scalebox{0.22}{\includegraphics{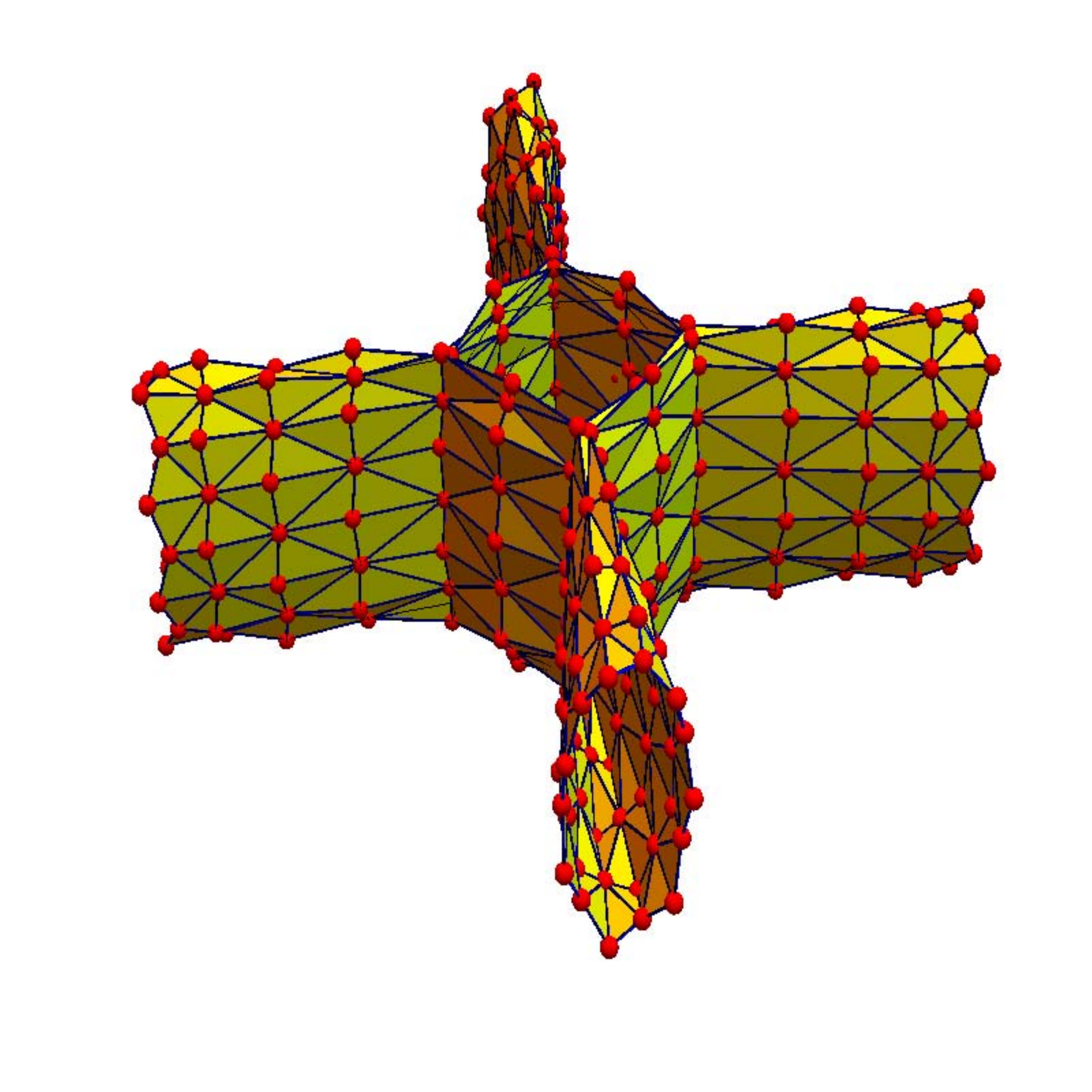}}
\caption{
To the left, the product of a wheel graph with a line graph:
both factors are homotopic to a point so that the product is 
homotopic to a point. 
To the right, the product of a sun graph $S(4,4,4,4)$ 
with a cyclic graph $C_7$: both factors are homotopic to a circle
so that the product is homotopic to a torus. 
}
\end{figure}

{\bf 3)} A triangle is given by $K_3=x + y  + z + xy + yz + zx + xyz$.
Lets take the product with $K_2$ represented by $u + v + uv$. It 
becomes a $3$-dimensional graph. 

\begin{figure}[h]
\scalebox{0.22}{\includegraphics{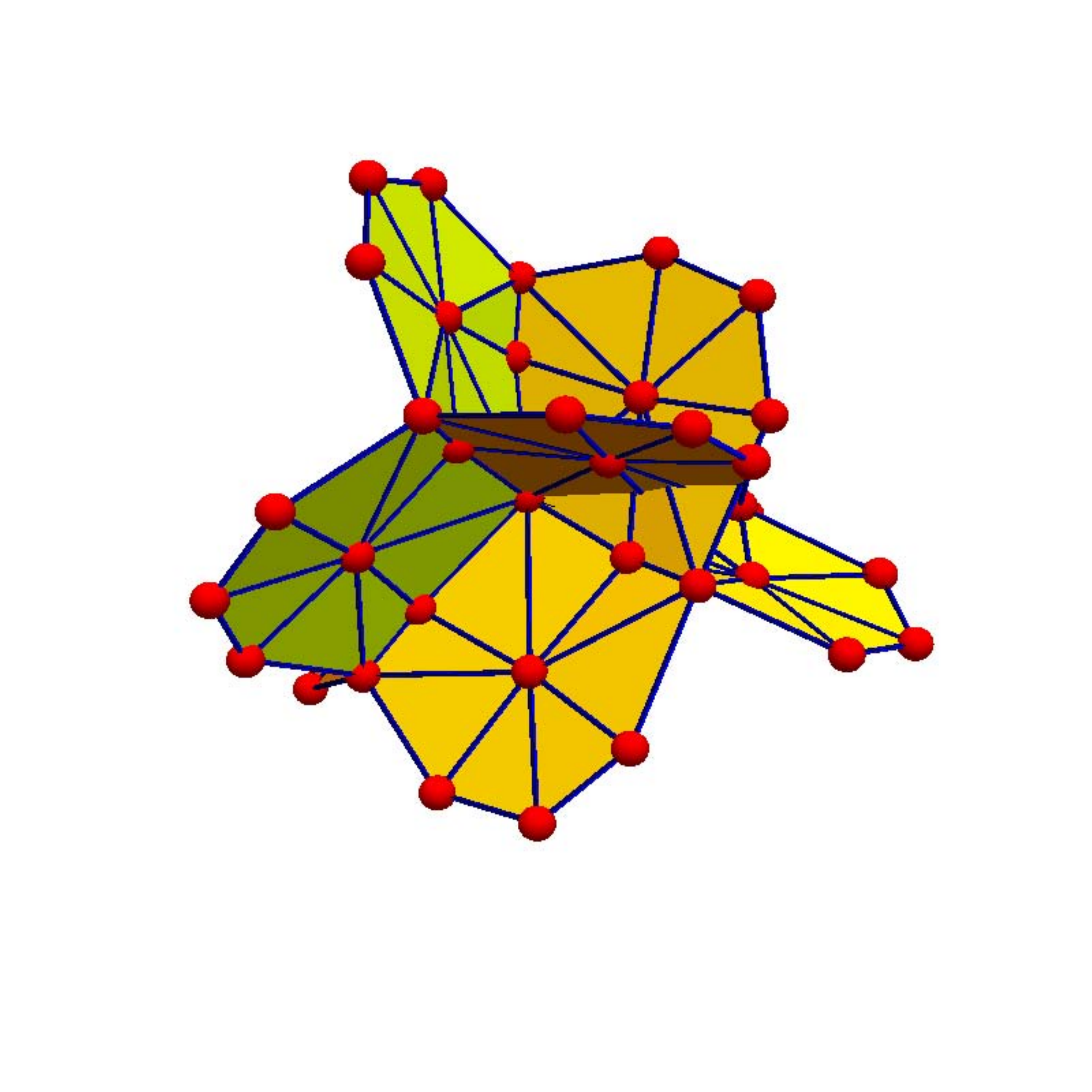}}
\scalebox{0.22}{\includegraphics{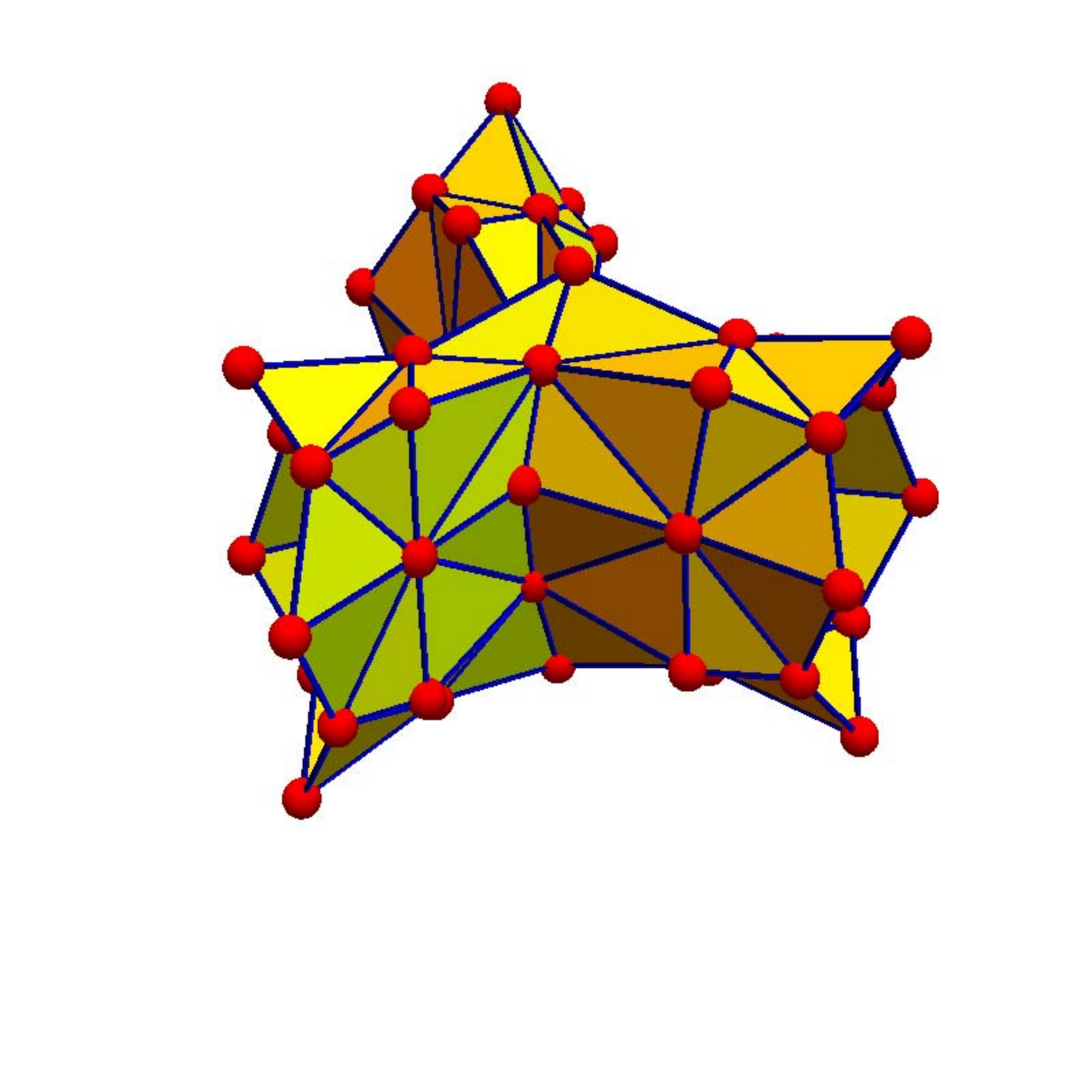}}
\caption{
The product of two star graphs $S_3$ is a $2$-dimensional graph.
The product of $C_4$ with the star graph $S_3$ is $2$-dimensional.
It is not geometric even so it looks like a cobordism
between $C_4 \cup C_4$ and $C_4$. }
\end{figure}

{\bf 4)} The product of $K_3$ with $K_2$ is a $3$-dimensional graph. 
It is a $3$-dimensional ball of radius $1$ with one interior point and all other
points on the boundary. The curvature of the interior point is $0$,
the curvature of nine of the boundary points are $1/6$ and three are $-1/6$. 
By Gauss-Bonnet, the sum of the curvatures is $1$, which is the Euler characteristic 
of the ball.  

\begin{figure}[h]
\scalebox{0.22}{\includegraphics{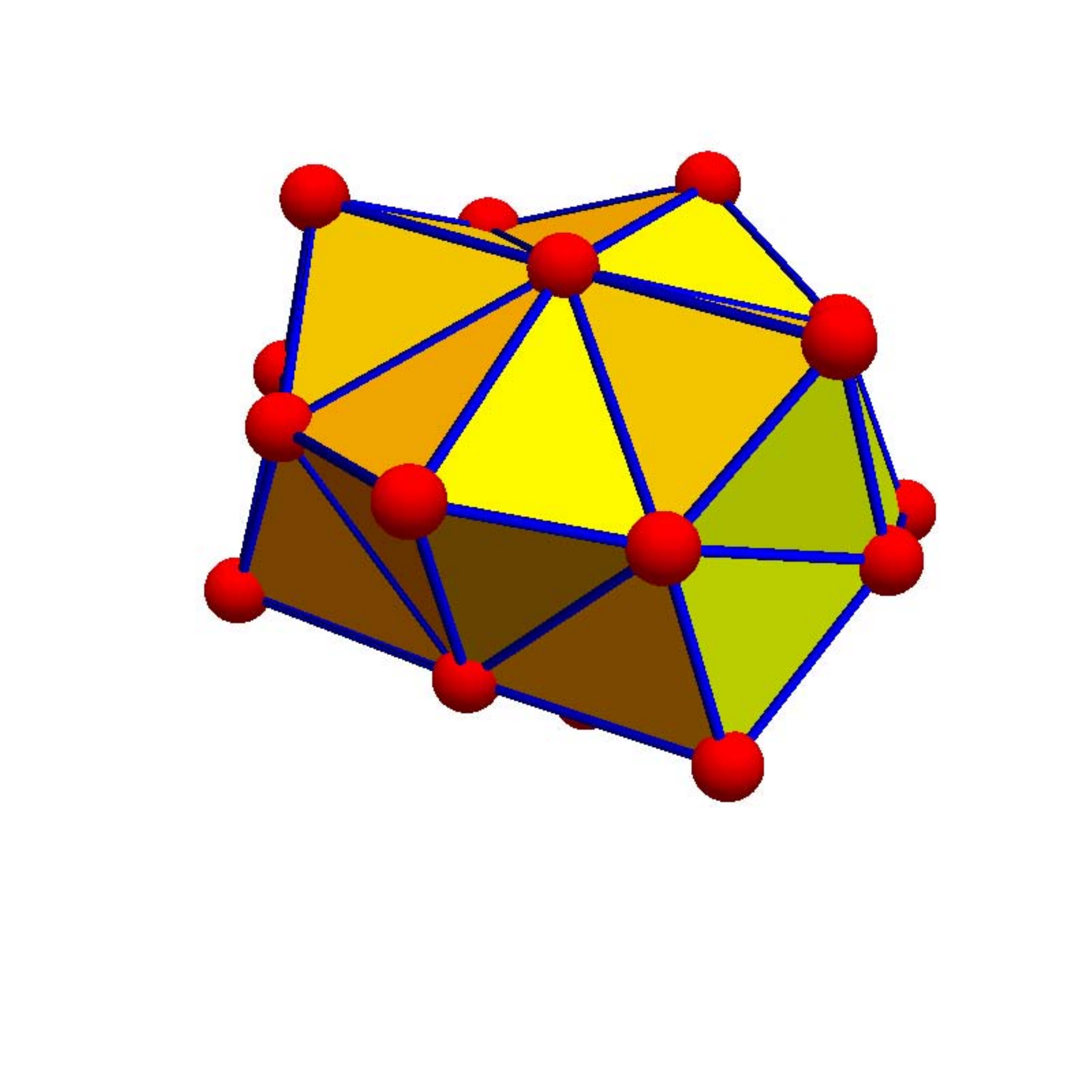}}
\scalebox{0.22}{\includegraphics{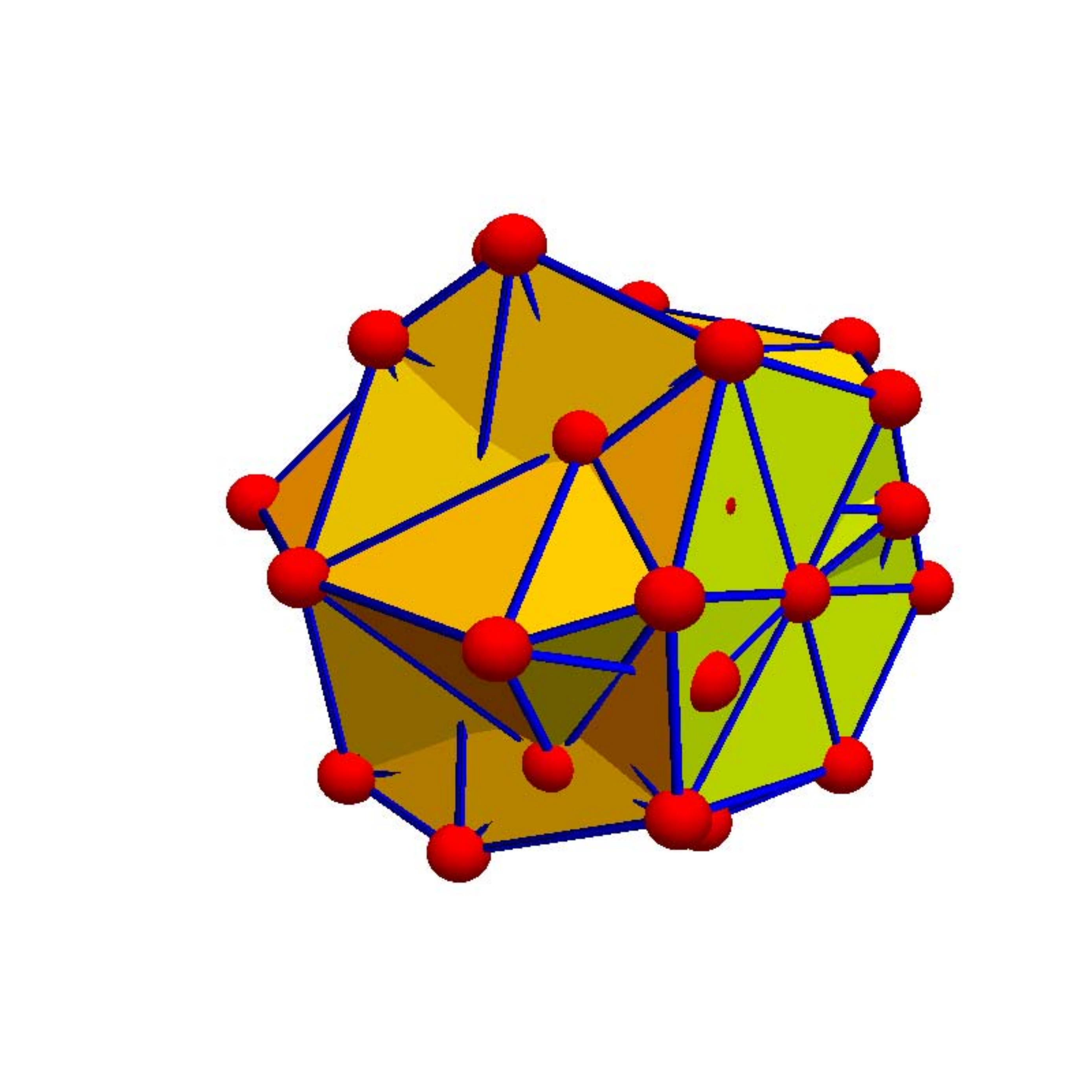}}
\caption{
The product of $K_3$ with $K_2$ is a $3$-dimensional ball of radius $1$. Its
boundary is a $2$-dimensional sphere. The second picture shows the product
$K_3$ with $K_3$. It is a $4$-dimensional contractible graph whose boundary is
a $3$-dimensional sphere. 
}
\end{figure}

\begin{figure}[h]
\scalebox{0.22}{\includegraphics{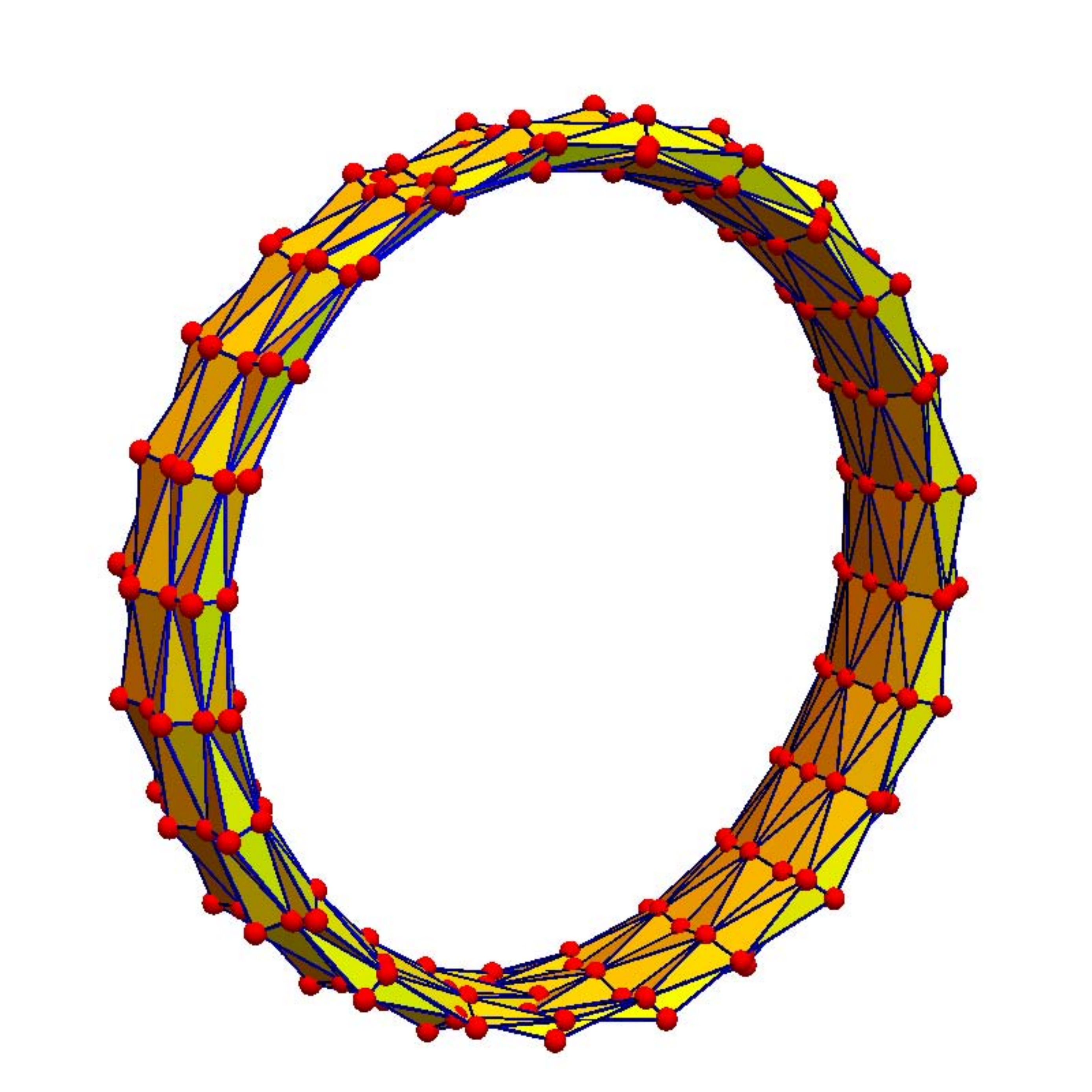}}
\scalebox{0.22}{\includegraphics{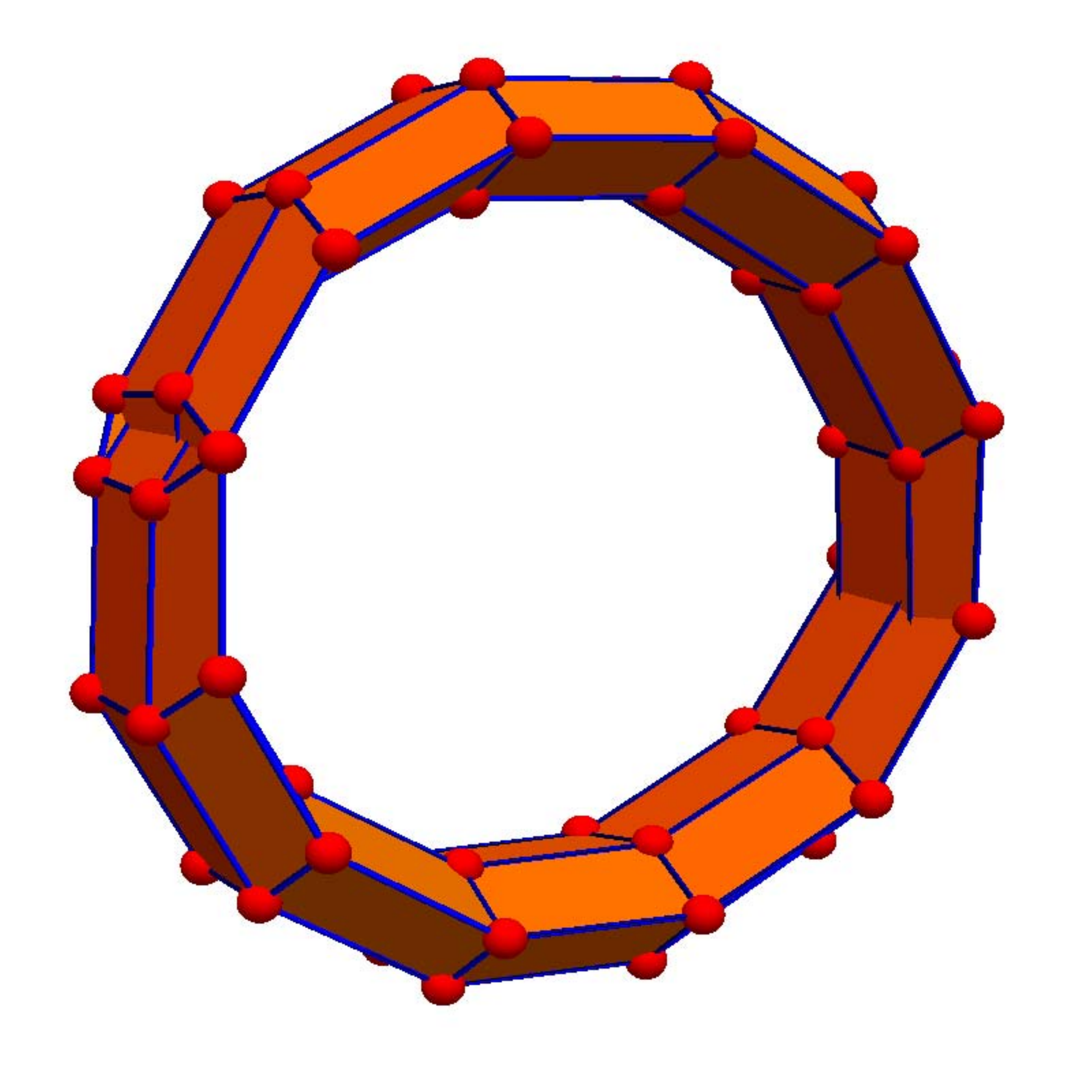}}
\caption{
The product of $C_6$ with $C_4$ is a discrete $2$-dimensional torus. 
The traditional cartesian product $C_{12}$ with $C_5$ is a $1$-dimensional graph
as it has no triangles. 
}
\end{figure}

\begin{figure}[h]
\scalebox{0.34}{\includegraphics{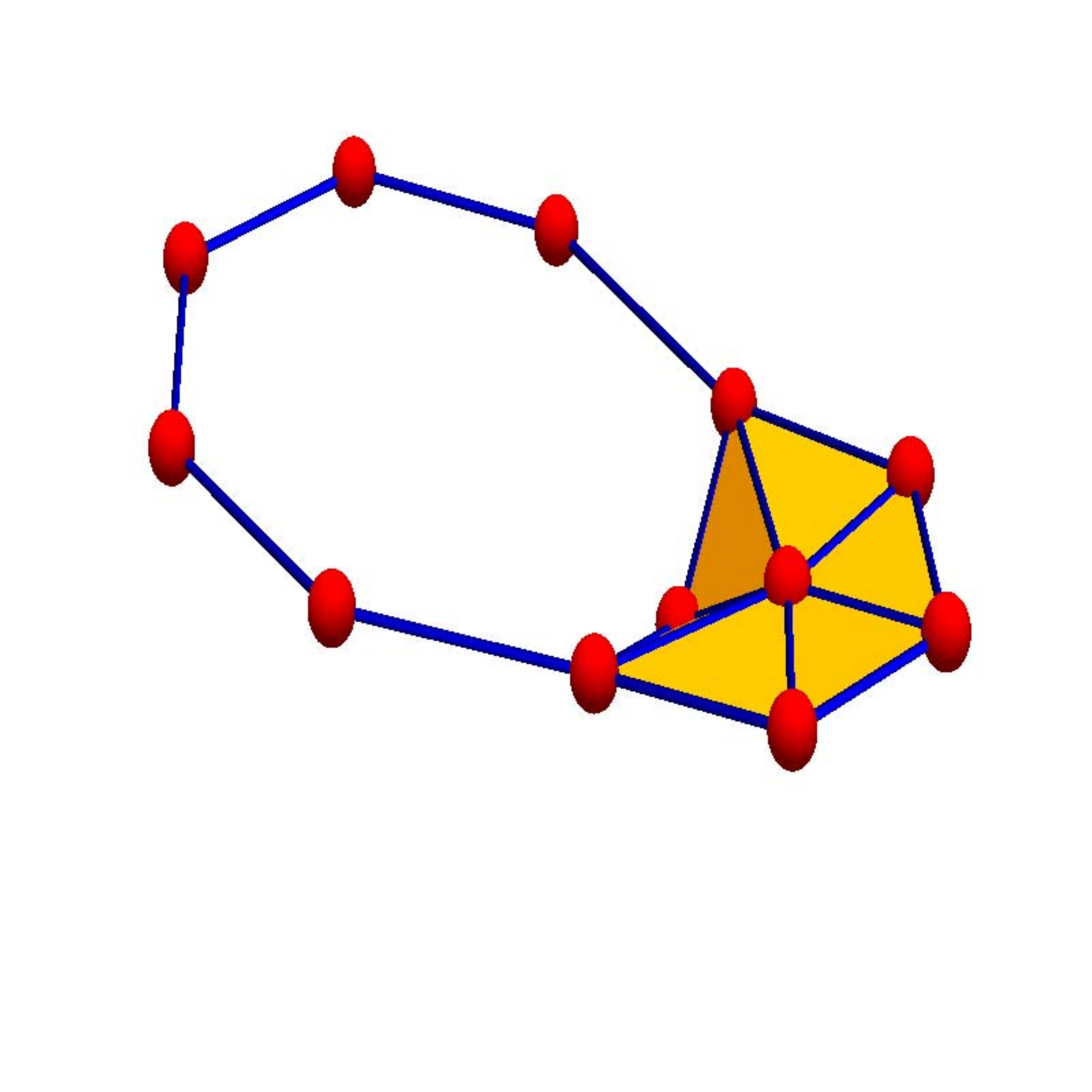}}
\scalebox{0.34}{\includegraphics{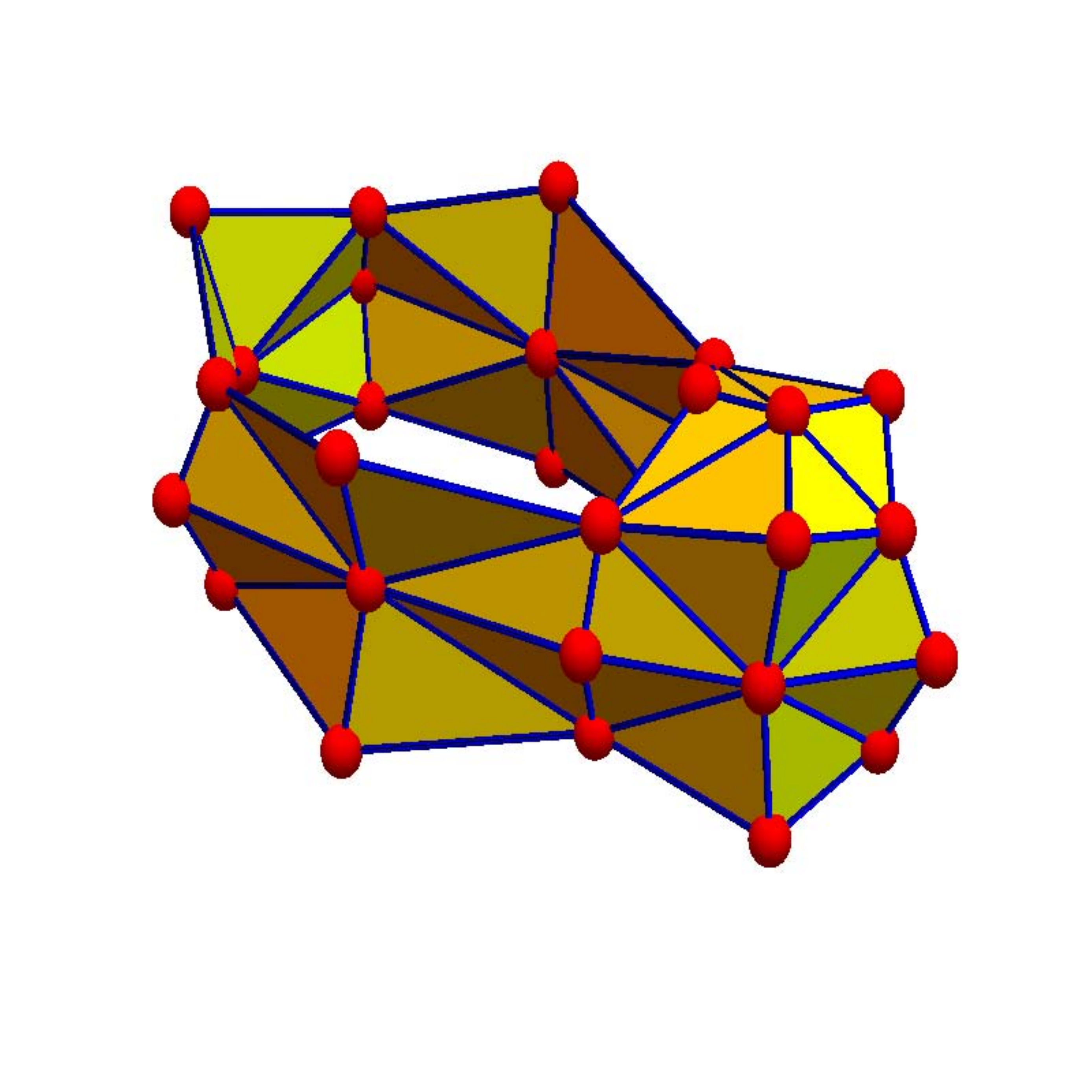}}
\caption{
The graph product of the house graph $G$ with $H=K_2$. 
We see first the enhanced graph $G_1 = G \times K_1$ and then the graph 
$G \times H$.  }
\end{figure}

\section{Lose ends}

{\bf Dimension}: \\
What is the expectation of $X(G) = {\rm dim}(G_1) - {\rm dim}(G)$ on the Erd\"os-R\'enyi probability space.
This random variable measures in some sense the distance from a geometric graph as for 
geometric graphs, the value is zero. The value of $X(G)$ is related to the variance of the dimension
spectrum, ${\rm dim}(x)$ which is a random variable on vertices of $G$. It would be nice
to quantify this more. \\

{\bf Hausdorff dimension}:\\
Due to the analogy of results with Hausdorff dimension ${\rm dim}_H$,
one can ask whether there is more to it 
and whether for any subsets $X,Y$ of Euclidean space ${\bf R}^k$, 
there are sequences of graphs $G_n,H_n$ such 
that ${\rm dim}(G_n) \to {\rm dim}_H(X)$ and ${\rm dim}(H_n) \to {\rm dim}_H(Y)$.  \\

{\bf Curvature}: \\
The construction $G \to G \times K_1$ gives a convenient natural curvature
on the simplices of a graph $G$. Gauss-Bonnet holds as they add up to Euler characteristic.
It is also true that curvature is the average over all functions \cite{indexexpectation,colorcurvature}  \\
What is the relation between the curvature of $G,H$ and $G \times H$? 
The example of $C_n \times C_m$ shows that the product of two zero curvature spaces
can develop curvature. The case $(K_3 \times K_1) \times K_1$ shows 
that the curvature can become very negative. \\

{\bf Spectrum}: \\
What is the relation between the spectrum of $G \times H$ with the spectrum 
of $G$ and $H$? We know that the Laplacians live on different spaces. 
The eigenvalues appear to correlated to the eigenvalues $\lambda \mu$, where 
$\lambda$ is an eigenvalue of $L(G \times K_1)$ and $\mu$ is an eigenvalue of $L(H \times K_1)$. \\

{\bf Chromatic number}: \\
The case $C_5 \to C_{10}$ shows that the chromatic
number can become smaller. The product of two 
geometric graphs is Eulerian in the sense \cite{KnillEulerian} which 
shows that if $G$ is a geometric graph, then $G_1 = G \times K_1$ is Eulerian. 
The map $G \to G_1$ regularizes in some way as for geometric graphs, the unit 
spheres of $G_1$ are Eulerian spheres. Of course, it would be nice to see a relation between the
chromatic polynomials $c(G)$ and $c(G_1)$ of $G$ and $G_1$. It is not always true
that $c_G(x)$ divides $c_{G_1}(x)$ but it is often the case. For the graph $G=K_3$ for
example, $c_G(x)=x^3-3x^2+2x$ and $c_{G_1}(x)=62*x - 191*x^2 + 240*x^3 - 160*x^4 + 60*x^5 - 12*x^6 + x^7$.
Now $c_{G_1}(x)/c_G(x)=31 - 49*x + 31*x^2 - 9*x^3 + x^4$. 
For $G=K_4$, this chromatic fraction of the graph $G$ is 
$-8834963 + 21070773x - 23177760x^2 + 15536807x^3 - 7055131x^4 + 2278645x^5 - 533896x^6 + 90682x^7 - 
 10932x^8 + 890x^9 - 44x^{10} + x^{11}$. Already for $G=C_4$, the chromatic fraction is not a polynomial.
Here is more random example, where the chromatic fraction is a rational function
and not a polynomial is the graph $G$ with $7$ vertices and edges
$\{ (1, 4), (1, 6), (2, 4), (2, 6), (2, 7), (3, 6), (4, 6), (5, 7), (5, 8), (6, 7)) \}$ for which 
the chromatic fraction is only rational.
Still, the chromatic number is $3$ for both graphs. 
With respect to the known relation $c(G_1) \leq c(G)$, we have inequality for all graphs $C_{2n+1}$.
For complete graphs $K_n$, both sides are $n+1$,
for trees, both sides $2$. For geometric graphs, all graphs $G_1$ have the property that $c(G_1 \times K_0)=c(G_1)$. 
On the list of all $38$ connected graphs with $4$ vertices, we have equality $c(G_1)=c(G)$, on the list of all 
$728$ connected graphs with $5$ vertices there are a dozen 
for which the chromatic number has dropped, the reason
always being a $5$-cycle of chromatic number $3$ becoming $2$-colorable after the refinement.  \\

\begin{figure}[h]
\scalebox{0.25}{\includegraphics{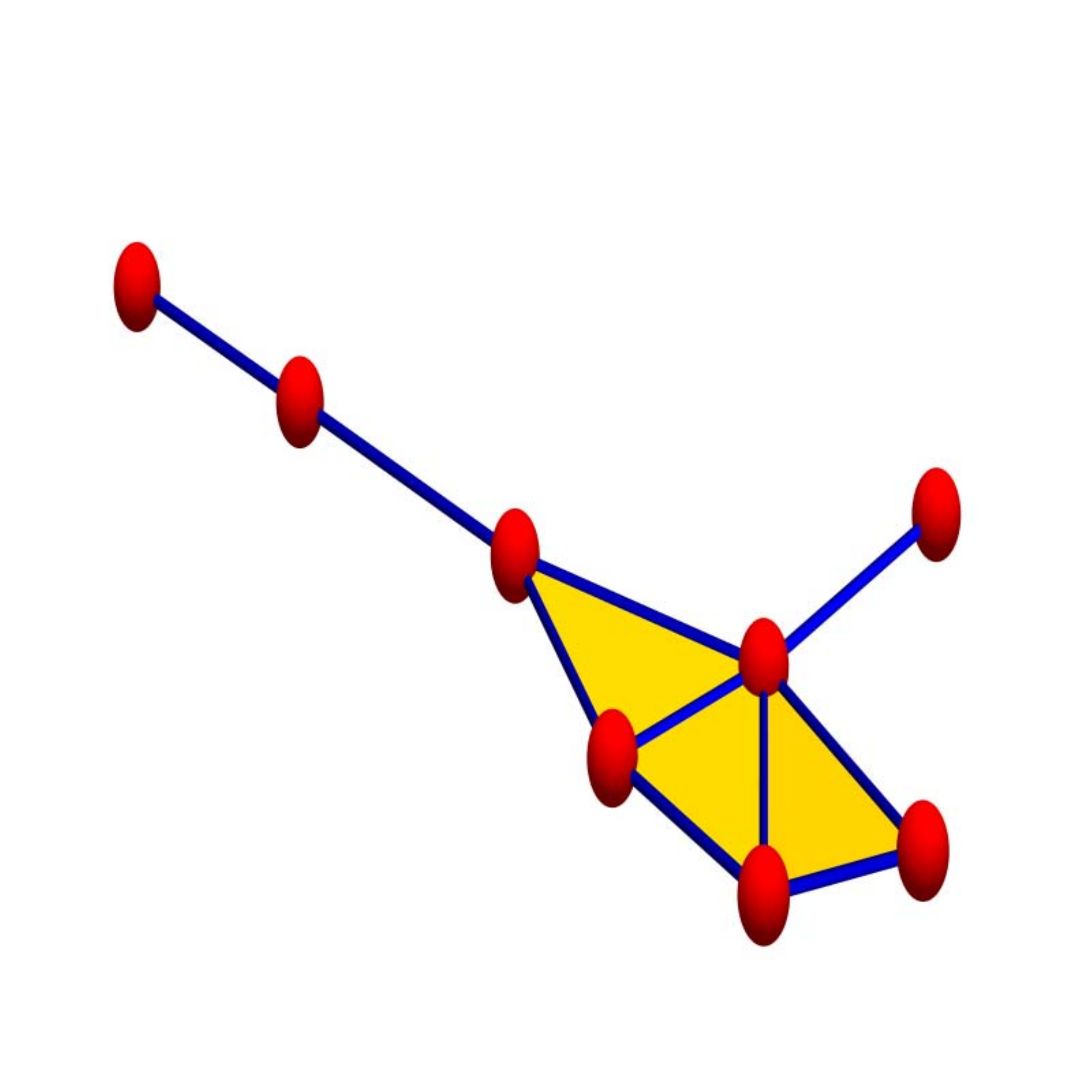}}
\scalebox{0.25}{\includegraphics{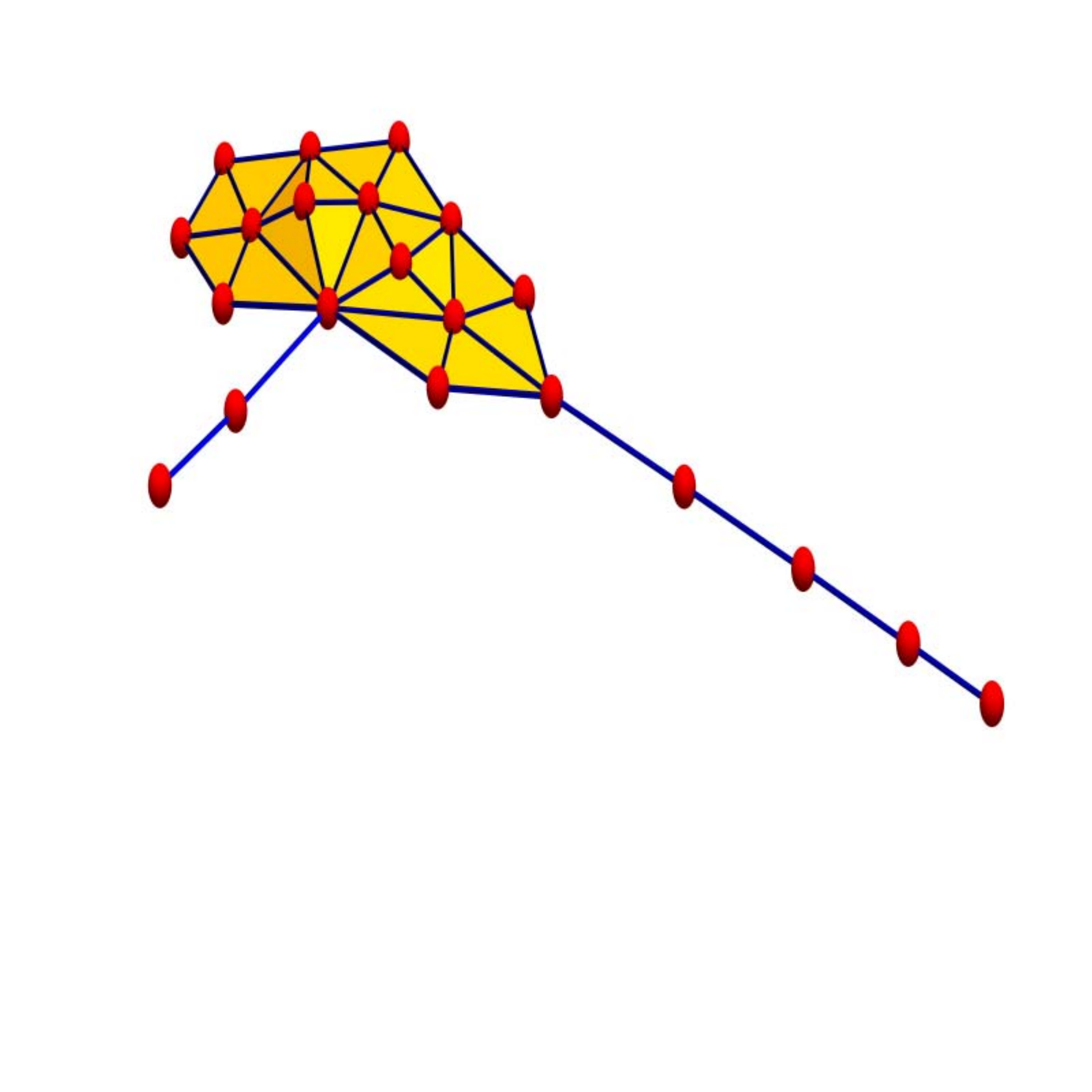}}
\caption{
A graph $G$ and the enhanced graph $G_1=G \times K_0$. Both graphs
have chromatic number $3$. The chromatic polynomial of $G$ is
$(-2 + x)^3 (-1 + x)^4 x$,
the chromatic polynomial of $G_1$ is
$(-2 + x)(-1 + x)^7 x(7 - 5x + x^2)(12193 - 43730x + 72229x^2 - 
  72238x^3 + 48366x^4 - 22614x^5 + 7465x^6 - 1715x^7 + 262x^8 - 
  24x^9 + x^{10})$. The chromatic fraction is not a polynomial. 
Still, the chromatic number is $3$ for both graphs. 
}
\end{figure}

{\bf Constructing spheres}: \\
What kind of spheres appear as intersections of unit spheres in 
discrete tori $C_{n_1} \times \cdots \times C_{n_k}$ with $n_k \geq 4$. 
We have seen that they all are unions of solid cylinders glued along a torus. 
Assume we have obtained a class of spheres like that. Now take the product of 
such spheres and take again unit spheres. 
We get a larger and larger class of spheres. Can one characterize them? 
Similarly, we can ask what kind of spheres can be obtained by applying the 
join operation to spheres, starting with a smaller class of spheres.  \\

{\bf Determinant}: \\
Can we say something about the complexity ${\rm Det}(L)$ of the 
product, where ${\rm Det}$ is the pseudo determinant \cite{cauchybinet}.
If $G=K_3, H=K_1$, then the complexities of $G_1,H_1$
are $716800$ and $3$. The complexity of $G_1 \times G_2$ is already
\begin{tiny}$99446661202899312347583662774063513018614477321403222712572446349721600$. \end{tiny}
The pseudo determinant of the Laplacian of $G_1$ is $2240$ and of $H_1$ is $3$. 
The pseudo determinant of the Laplacian of $G \times H$ is $21101889958483152$. \\

{\bf Homotopy classes}: \\
If $G_1$ is homotopic to $G_2$ and $H_1$ is homotopic to $H_2$, then 
$G_1 \times G_2$ is homotopic to $H_1 \times H_2$. The sum and product therefore
produces a sum and product in the ring of homotopy classes. The contractible ring 
elements form a class representing the $1$ element. The empty graph is the $0$ element. 
Since also $\delta G_1$ is homotopic to $\delta G_2$. 
Homotopy extends to the ring of chains. Define contractibility in the same way
by telling $G$ is contractible if there exists $x$ such that both $S(x)$ and $G(x)$ 
are contractible where $G(x)$ is the non-constant part of the polynomial where $x$ is replaced by $1$.  
The unit sphere is the unit sphere of the graph $G_f$. 
Homotopy is so defined algebraically as reduction steps $f \to f - S(x)$ with contractible $S(x)$ or reverse steps. 
Contractibility is defined for $f$ in a polynomial ring as the property that there is a variable $x_i$ such 
that $S(x_i)$ and $f-S(x_i)$ are both contractible. For example a line graph $xy+yz+x+y+z$ is contractible 
as $S(z)=z+yz$ and $f-S(z)=xy+x+y$ are both contractible. The cycle graph $xy+yz+zw+wx+x+y+z+w$ is not contractible.
Also all its unit spheres like $S(x)=xy+wx$ or $S(y)=xy+yz$ which represent two point graphs are not contractible. \\

\begin{figure}[h]
\scalebox{0.36}{\includegraphics{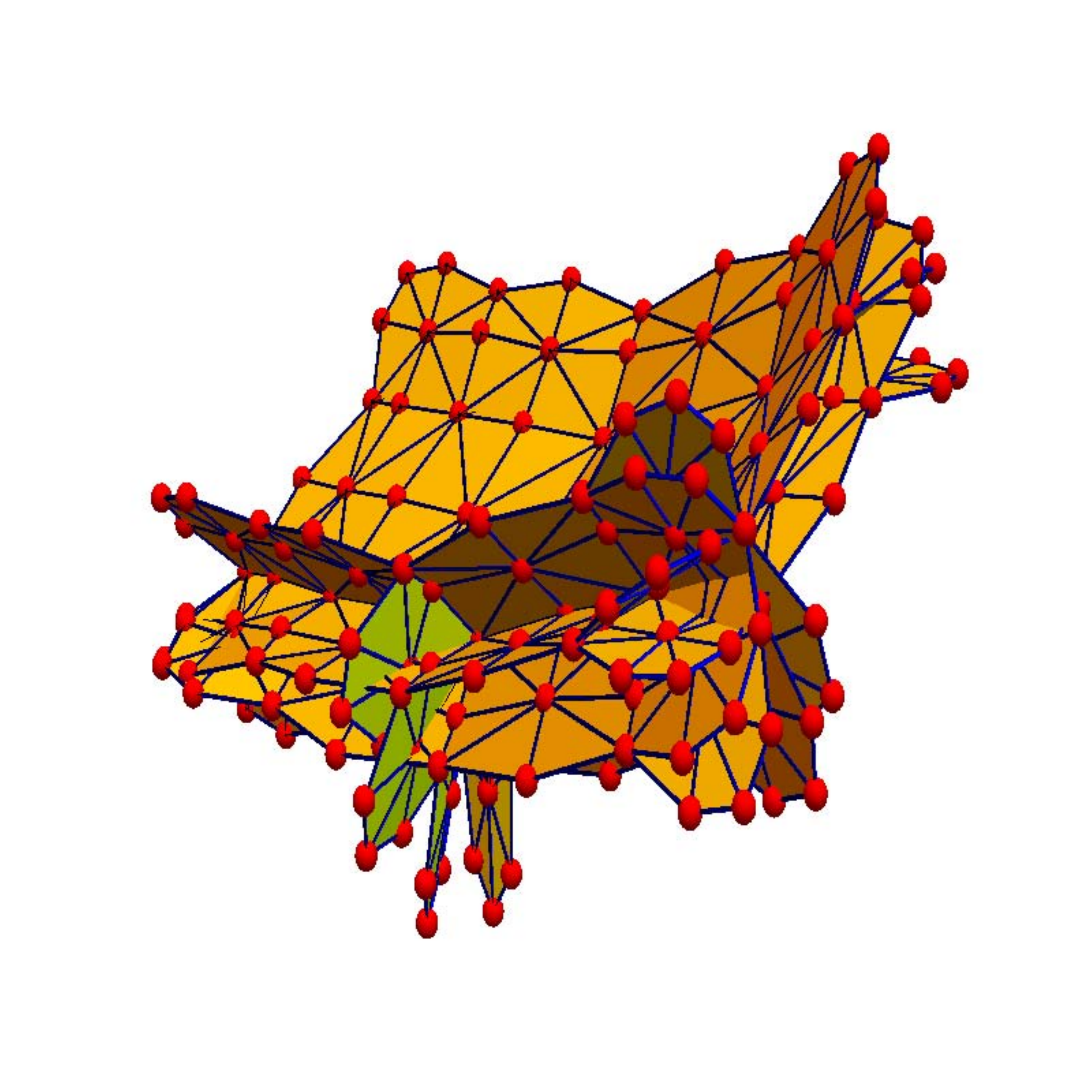}}
\caption{
The product of two trees is a $2$-dimensional graph which is contractible
as it is the product of two graphs which are homotopic to a point. 
}
\end{figure}

{\bf Suspension and join}: \\
The topological notions of ``suspension" and ``join" needs the Cartesian product. 
The definition of a suspension of a graph $G$ is
$SG = G \times I/\{ (x,0) \sim (y,0), (x,1) \sim (y,1) \}$. 
Also in the graph case, it produces from a $k$-sphere a $(k+1)$-sphere. 
More generally, the join of two graphs $G,H$ can be defined as
$$ (G \times H \times K_2)/\{ (x,y_1,0) \sim (x,y_2,0), (x_1,y,1) \sim (x_2,y,1) \} \; . $$
An obvious questions is whether results like the double suspension theorem
of Cannon-Edwards can be proven directly in 
graph theory. It certainly is true, as one can see by building corresponding
manifolds from the graphs. But it would be nice to have a direct discrete proof.
The double suspension of a graph $G$ is $S^2G = (G \star S_0) \star S_0$. A homology 
$n$-sphere is a graph $G$ which is geometric of dimension $n$, which has the Poincar\'e
polynomial $p_G(x)=1+x^n$ but which is not a $n$-sphere.
The discrete Cannon-Edwards
theorem tells that if $G$ is a homology sphere, then $H=S^2G$ is a sphere. To prove this
within graph theory, one has to show that every unit sphere in $H$ is a $(n-1)$-sphere
and that removing one vertex in $H$ produces a $n$-ball, a contractible $n$-dimensional
graph with $n-1$ dimensional sphere boundary.  \\

\begin{figure}[h]
\scalebox{0.18}{\includegraphics{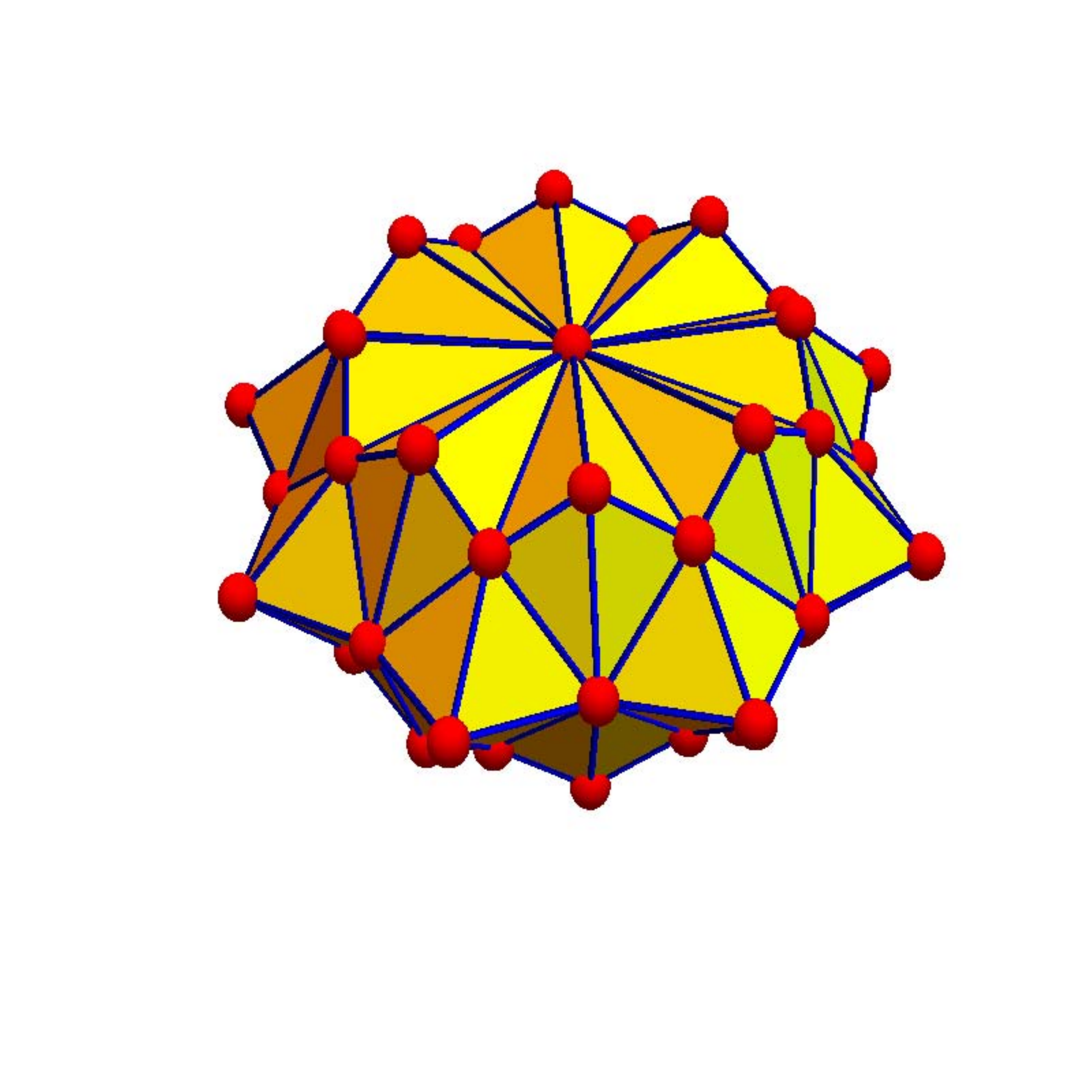}}
\scalebox{0.18}{\includegraphics{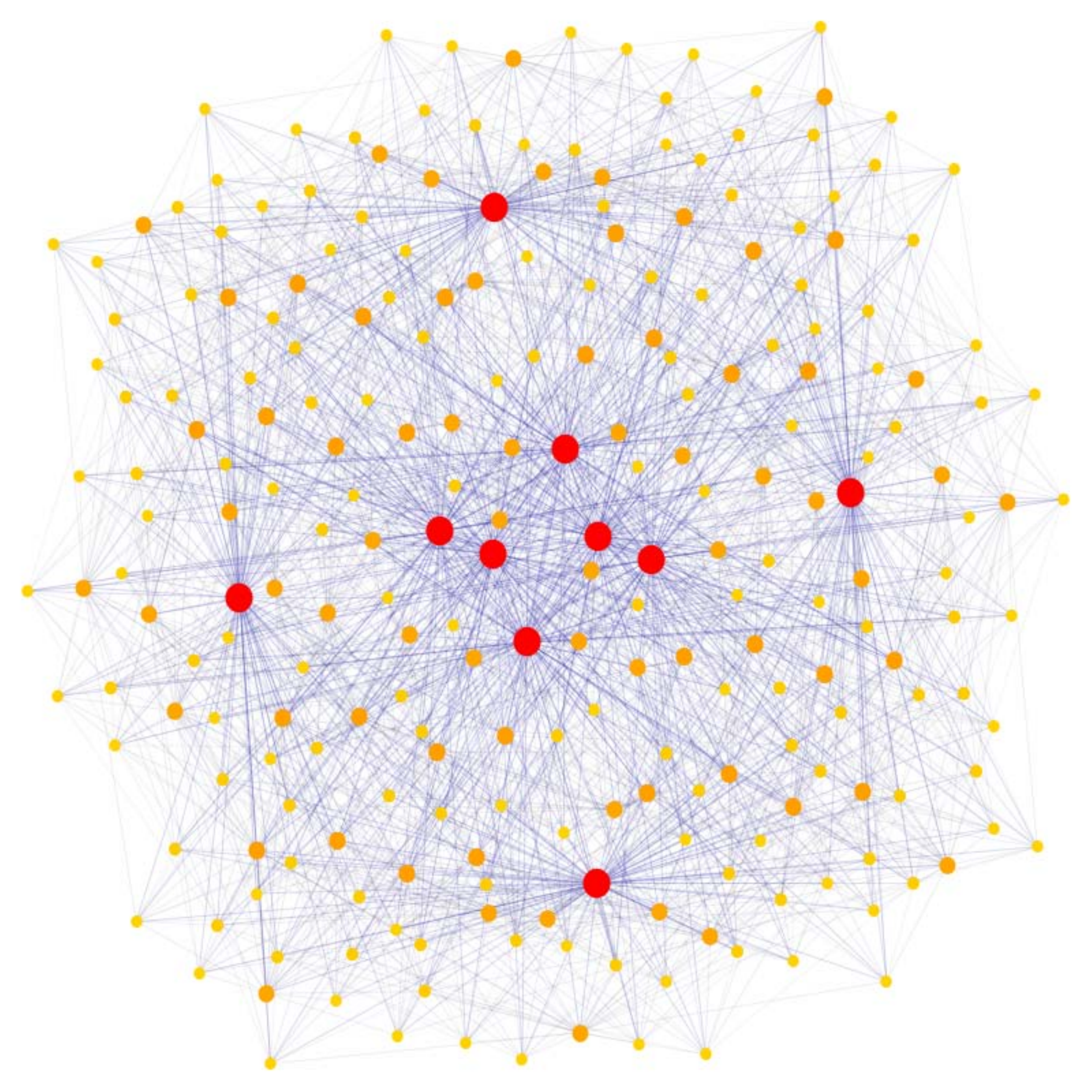}}
\scalebox{0.25}{\includegraphics{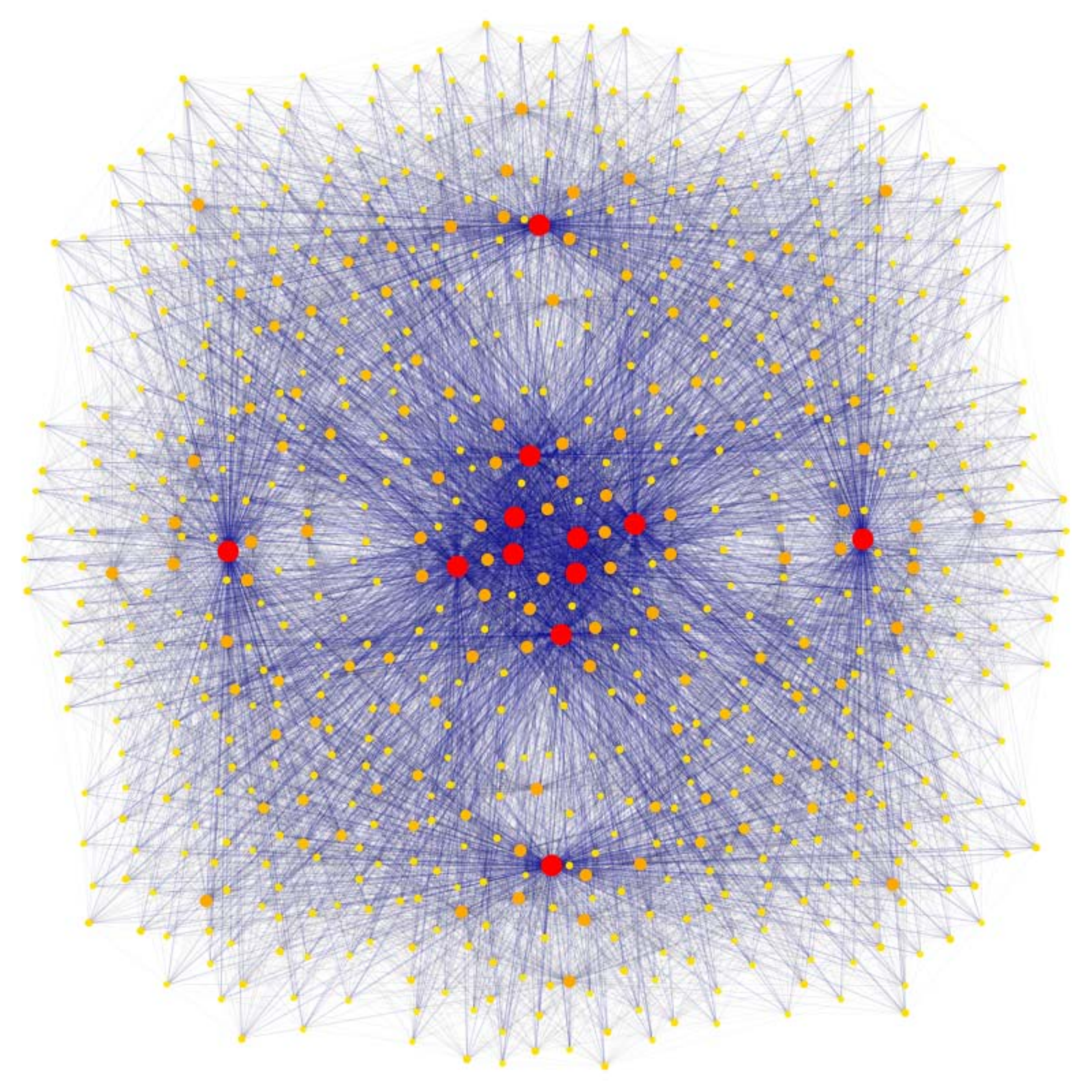}}
\caption{
In this figure, we have taken joins by taking the interval $K_2$. 
The join $C_8 \star S_0$ of a cycle graph with a $0$-dimensional 
sphere is a $2$-dimensional sphere seen to the left. 
The join $O \star S_4$ of the octahedron with a circular graph is
a $4$-dimensional geometric sphere seen to the middle. 
The join $O \star O$ of the octahedron graph $O$ with itself
is a 5-dimensional geometric sphere. It has $v_0=728$ vertices,
$v_1=14168$ edges, $72960$ triangles. }
\end{figure}

\begin{figure}[h]
\scalebox{0.24}{\includegraphics{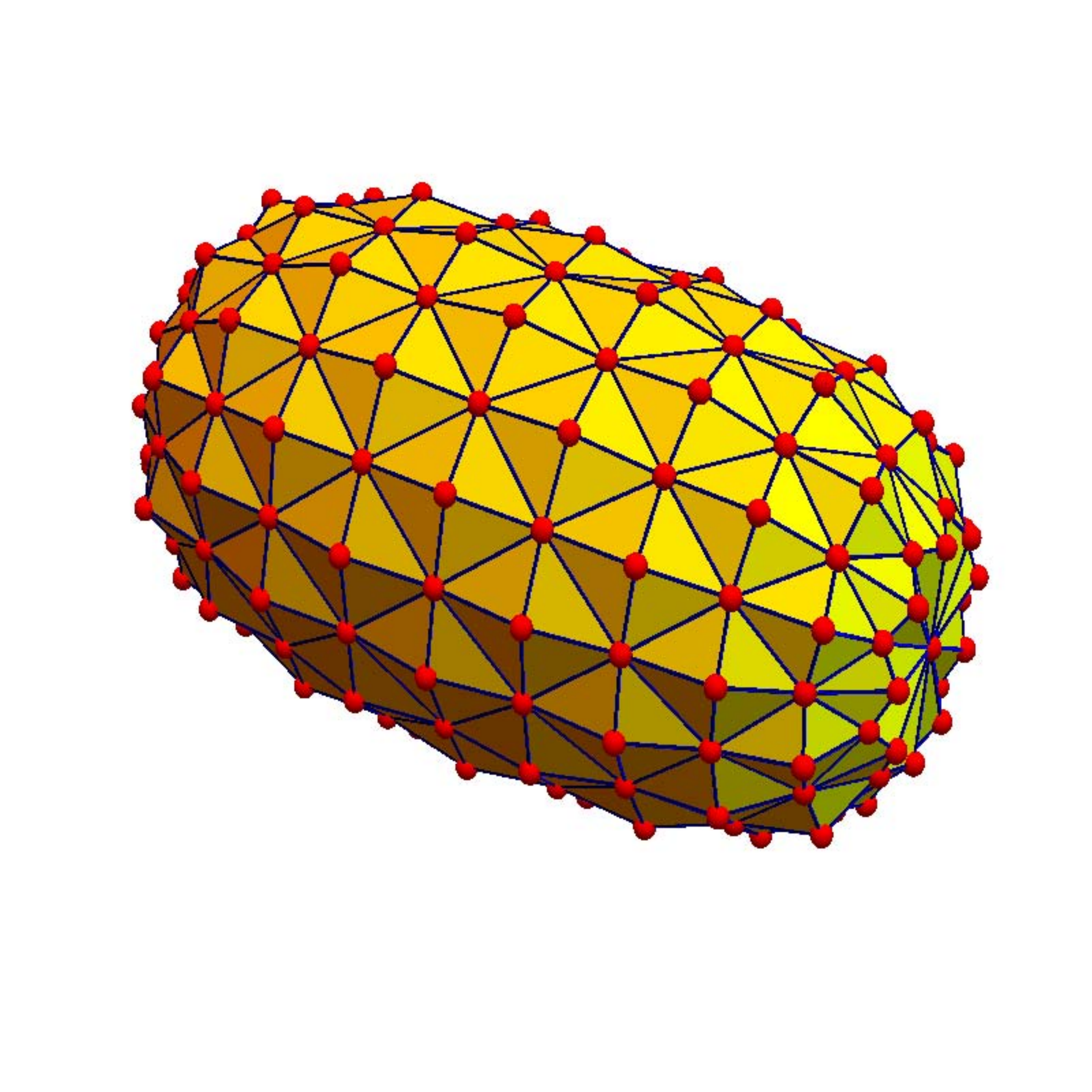}}
\scalebox{0.18}{\includegraphics{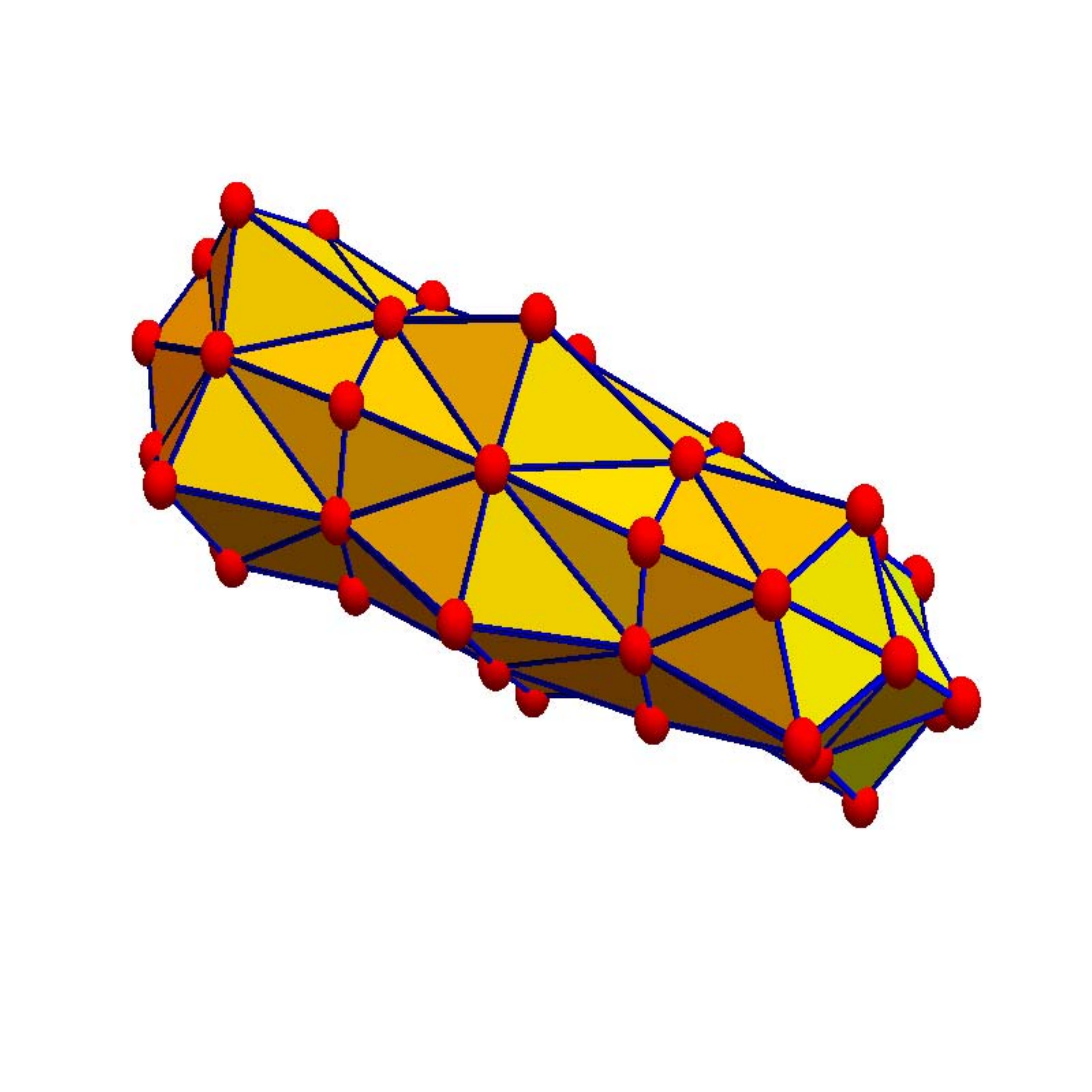}}
\scalebox{0.18}{\includegraphics{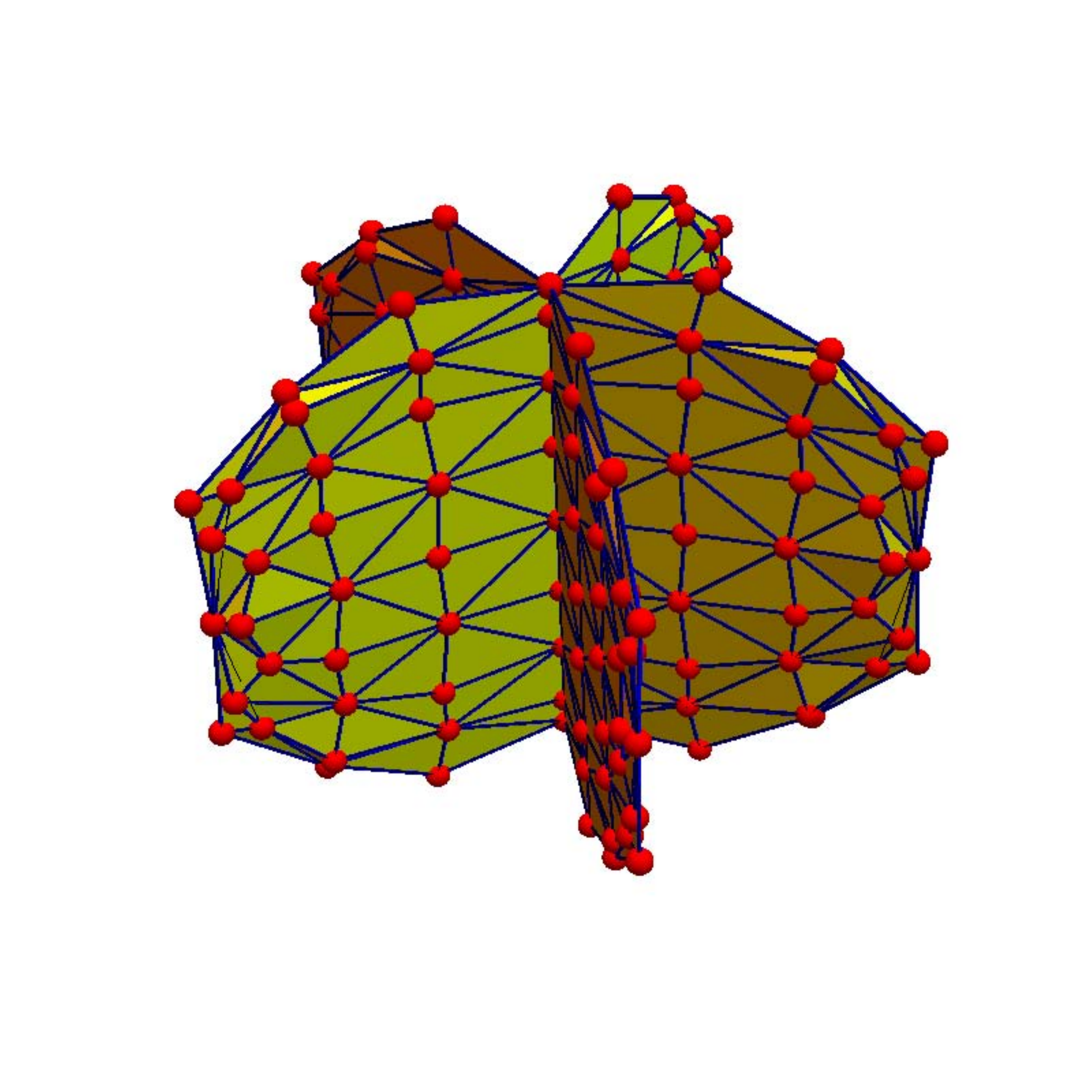}}
\caption{
This figure shows more joins. This time, we replace $K_2$ with a longer line graph $L_n$
representing an interval. The first case is $C_8 \star S_0$ with 
interval graph $L_4$. It is a $2$-dimensional sphere. The second
case is $K_2 \star K_2$ with interval graph $L_4$. It is a three dimensional
ball with two dimensional sphere as boundary. The third 
example is a product of two Erd\"os-R\'enyi graphs with 5 vertices,
edge probability $0.2$ and using the interval graph $L_3$. 
}
\end{figure}

{\bf Higher order curvature}:  \\
It is natural to inquire whether for $2$-dimensional geometric 
graphs $G_1 = G \times K_1$, the higher order curvature $K(x)=2S_1(x)-S_2(x)$ produces 
a Gauss-Bonnet result. This is not the case as for the octahedron $G$, the graph $G_1$ 
has a sum of the second order curvatures which is $0$. The culprit are the degree 
4 vertices. \\

{\bf The sequence of graphs $G_n$.} \\
What happens with $G_n$ defined recursivly by $G_n=G_{n-1} \times K_1$ is
that the minimal curvature goes to $-\infty$, at least if $G_0$ is two dimensional. 
It would be interesting to know how the
curvature spectrum $\kappa(G_n) = [{\rm min}(K_{G_n}(x),{\rm max}(K_{G_n}(x)]$ grows 
for higher dimensional graphs. 
The numerics is tough as the graphs sizes explode exponentially. L
If $G_0=K_3$, then the curvature spectra are $\kappa(G_0) = [1/3,1/3]$,
$\kappa(G_1)=[0,1/6], \kappa(G_2)=[-1,1/3], \dots$.
If $G_0=K_4$, then $\kappa(G_0)=[1/4,1/4], \kappa(G_1)=[0,1/6]$. \\

{\bf Global De Rham}: \\
We have constructed a chain homotopy between de Rham and simplicial chains. 
If we define a discrete $d$-manifold as a graph which locally can be wrtten as a product
$U_i = G_{i1} \times \cdots \times G_{id}$ of networks such that on the intersection 
$U_i \cap U_j$ there are coordinate change graph homeomorphisms, then the 
chain homotopy can be pushed through. We don't even need the graphs $G_{ij}$ to have to 
be $1$-dimensional, nor have they be to be geometric. The combinatorial 
de Rham theorem allows to deal with such discrete manifolds more effectively 
similarly as the continuum de Rham theorem does.  \\

{\bf Moving frames}: \\
While a full classical Hopf-Rynov theorem on graphs is impossible because of the finiteness of
any reasonable tangent space and therefore, a quantum flow is needed for a reasonable notion
of geodesic map,  there is weaker discrete
Hopf-Rynov theorem which tells that under some conditions, there is a global unique flow which 
locally minimizes length \cite{knillgraphcoloring2}. What is needed for a geometric graph 
of dimension $d$ is an Eulerian condition which is equivalent to the graph being 
minimally colorable with $d+1$ colors. In \cite{knillgraphcoloring2}, we made a stronger 
assumption and required all unit spheres to be projective in order to associate to an 
incoming ray an outgoing ray. There is an other approach which implements a 
``moving frame" idea of Cartan and works for all Eulerian graphs and in particular
with $G_1 = G \times K_0$ of $G \times H$ if $G,H$ are geometric. It in particular does not need the projective
assumption. A discrete moving frame is a pair $(x,\sigma)$, where $\sigma$ is a $d$-dimensional
simplex containing $x$. The frame propagation defines a geodesic flow in an Eulerian graph 
without the need of a projective involution: it is described inductively with respect to 
dimension: use the geodesic flow on the Eulerian unit sphere $S(x)$ to 
flip the frame $(x,\sigma)$ on the sphere $S(x)$ 
to the anitpodal side $(x',\sigma')$, where either the simplex $\sigma'$
or the vertex $x'$ has maximal distance (both cases are possible as the example in Figure~(\ref{antipodal})
shows), then reflect the vertex $x'$ along the face $\sigma' \setminus \{x'\}$ to a vertex $x''$ 
on the adjacent simplex $\sigma''$. This propagation $(x,\sigma) \to (x'',\sigma'')$ 
defines a variant of a geodesic flow for which the simplex $\sigma$ plays the role of the direction, 
as well (if we keep track of the frame in the simplex) produces a parallel transport on the graph. 
The point we want to make is that for all product graphs $G \times H$ and more generally 
for all discrete manifolds obtained by patching product graphs, where each factor is a 
geometric graph, there is a canonical notion of geodesic flow and parallel transport. 
In some sense, there is a unique Levi-Civita connection on such graphs. \\

\begin{figure}[h]
\scalebox{0.18}{\includegraphics{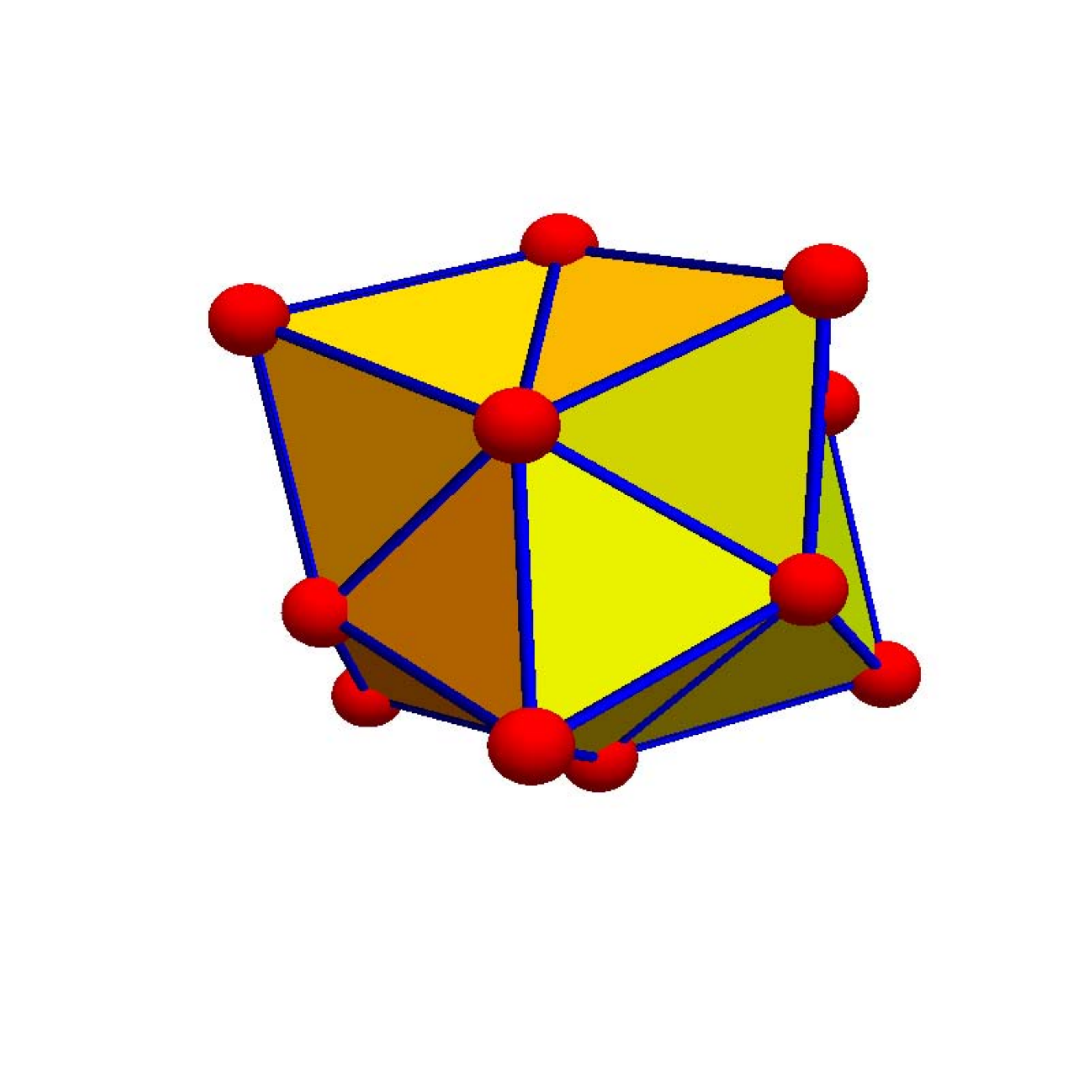}}
\scalebox{0.18}{\includegraphics{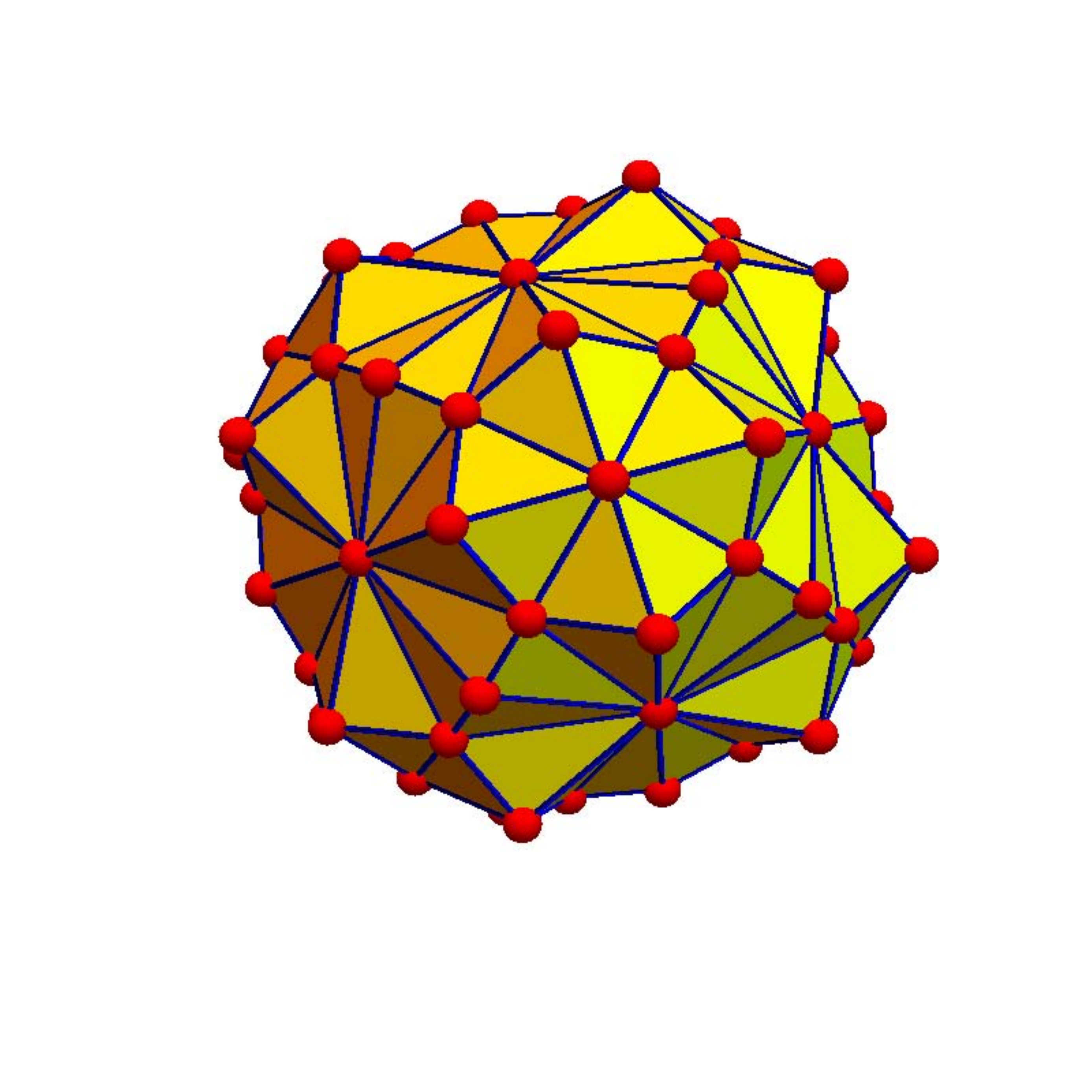}}
\caption{
\label{antipodal}
The stellated cube $G$ is a sphere with 14 vertices. It is Eulerian. 
The graph $G_1$ is also Eulerian with $74$ vertices. While $G$ has a natural
antipodal map (there is a unique opposite vertex with maximal distance), this
is no more the case for $G_1$. The antipodal point of some vertices is a triangle.
The discrete analogue of a moving fame associates to a vertex in the $d$-dimensional
graph a $d$ dimensional simplex. 
}
\end{figure}

\vspace{12pt}
\bibliographystyle{plain}

\end{document}